\documentclass[12pt]{amsart}

\usepackage{amssymb, amsfonts, amsmath, amsthm, amsbsy}
\ifx\luatexversion\undefined
\ifx\XeTeXversion\undefined
\usepackage[english]{babel}
\else
\usepackage{fontspec}
\usepackage{polyglossia}
\setmainlanguage{english}
\usepackage{microtype}
\fi
\else
\usepackage[no-math]{fontspec}
\usepackage[english]{babel}
\usepackage{lualatex-math}
\fi

\usepackage{todonotes, graphicx, hyperref, verbatim, enumerate}

\hypersetup{
    bookmarks=false,         
    pdftitle={},    
    pdfauthor={},     
    colorlinks=true,       
    linkcolor=black,          
    citecolor=black,        
    filecolor=black,      
    urlcolor=black           
}

\newcommand{\naturals}{\mathbb{N}}
\newcommand{\integers}{\mathbb{Z}}
\newcommand{\reals}{\mathbb{R}}

\newcommand{\fa}{\forall\,}

\newcommand{\nin}{\not\in}

\newcommand{\eps}{\varepsilon}

\newcommand{\deh}{\textup{d}}

\newcommand{\st}{\ \text{s.t.}\ }
\newcommand{\eg}{\text{e.g.}\ }
\newcommand{\ie}{\text{i.e.}\ }

\newcommand{\const}{C_\#}

\newcommand{\intr}[1]{\textup{int}\,#1}
\newcommand{\cl}[1]{\textup{cl}\,#1}

\newcommand{\inv}{^{-1}}

\newtheorem{thm}{Theorem}[section]
\newtheorem{lem}[thm]{Lemma}
\newtheorem{sublem}[thm]{Sub-lemma}
\newtheorem{prp}[thm]{Proposition}
\newtheorem{cor}[thm]{Corollary}
\theoremstyle{remark}
\newtheorem{rmk}[thm]{Remark}
\theoremstyle{definition}
\newtheorem{mydef}[thm]{Definition}
\newtheorem*{mydef*}{Definition}

\newtheorem*{mthm}{Main Theorem}

\numberwithin{equation}{section}

\usepackage{mathrsfs}

\usepackage[inline]{enumitem}
\setlist[enumerate,1]{label=\textup{(\alph*)}}
\setlist[enumerate,2]{label=\textup{\roman*.}}

\makeatletter
\providecommand\@dotsep{5}
\renewcommand{\listoftodos}[1][\@todonotes@todolistname]{%
  \@starttoc{tdo}{#1}}
\makeatother
\input{fum-p}

\author{Jacopo De Simoi}
\address{Jacopo De Simoi\\
  Department of Mathematics\\
  University of Toronto\\
  40 St George St. Toronto, ON, Canada M5S 2E4}
\email{{\tt jacopods@math.utoronto.ca}}
\urladdr{\href{http://www.math.utoronto.ca/jacopods}{http://www.math.utoronto.ca/jacopods}}

\author{Dmitry Dolgopyat}
\address{Dmitry Dolgopyat\\
  Department of Mathematics\\
  University of Maryland\\
  4417 Mathematics Bldg,  College Park,  MD 20742, USA}
\email{{\tt dmitry@math.umd.edu}}
\urladdr{http://www.math.umd.edu/\~{}dmitry}

\title{Dispersing Fermi--Ulam Models}
\hypersetup{
    pdftitle={Dispersing Fermi-Ulam Models},
    pdfauthor={Jacopo De Simoi, Dmitry Dolgopyat}
}

\def\tangu{{\Theta}}
\def\brb{{\bar{b}}}
\def\brc{{\bar{c}}}
\def\cT{{\mathcal{T}}}
\def\RmIII{{{I\!\!I\!\!I}}}
\def\RmII{{{I\!\!I}}}

\def\conen{{\mathfrak{N}}}
\def\conep{{\mathfrak{P}}}

\def\Kn{{\mathfrak{K}_{\bar n}}}
\def\coneIT{{\mathcal{C}_{I\tau}}}

\def\DS{\displaystyle}

\newcommand{\ignore}[1]{}

\begin{document}
\begin{abstract}
  We study a natural class of Fermi--Ulam Models that features good
  hyperbolicity properties and that we call \emph{dispersing
    Fermi--Ulam models}.  Using tools inspired by the theory of
  hyperbolic billiards we prove, under very mild
  complexity assumptions, a Growth Lemma for our systems.  This allows
  us to obtain ergodicity of dispersing Fermi--Ulam Models. 
 It follows  that almost every orbit of such systems is
  \emph{oscillatory}.
\end{abstract}
\maketitle

\tableofcontents
\section{Introduction.}
A Fermi--Ulam Model is a classical model of mathematical physics.
It describes a point mass moving freely between two infinitely heavy
walls.  One of the walls is fixed and the other one moves
periodically. Collisions with the walls are assumed to be elastic,
therefore the kinetic energy of the particle is conserved except at
collisions with the moving wall.  We denote the distance between the
two walls at time $t$ by $\ell(t)$.  We assume $\ell$ to be strictly
positive, Lipschitz continuous, piecewise smooth and periodic of
period $1$.

This model was introduced by Ulam, who wanted to obtain a simple model
for the stochastic acceleration, which according to Fermi \cite{F49,
  F54} is responsible for the presence of highly energetic particles
in cosmic rays.  Ulam and Wells performed numerical study of the
Fermi--Ulam model (see \cite{U}).  The authors were interested in
harmonic motion of the walls but due to limited power of their
computers they had to study less computationally intensive wall
motions. Namely, they assumed that velocity was either piecewise
constant or piecewise linear, since in that case the location of the
next collision can be found by solving either linear or quadratic
equation.  A few years after~\cite{U}, it has been pointed out by
Moser that if the motion of the wall is sufficiently smooth (in
particular, harmonic motions) then KAM theory implies that all
orbits have bounded velocities and so stochastic acceleration is
impossible.  The precise smoothness assumptions lneeded for the
application of KAM theory have been worked out by several
authors~\cite{Do, LL, P2, P1}.  However, Moser's argument does not
apply to the wall motions studied in \cite{U}. In fact, piecewise
smooth motions have been a subject of intensive numerical
investigations and several authors have reported the presence of
chaotic motions for certain parameter values (see \eg \cite{Br, CZ}).
The first rigorous result about the models studied in \cite{U} is due
to Zharnitsky, who proved in \cite{Z} the existence of unbounded
orbits for a range of parameters values. The next natural question is
how large is the set of orbits exhibiting stochastic acceleration.
In~\cite{fum}, we studied general wall motions such that the velocity
of the wall has only one discontinuity per period.  We
found\footnote{The results of \cite{fum} needed in the present paper
  are stated precisely in Section~\ref{SSCMP}.}  that the large energy
behavior of this system depends crucially on the value of a parameter
which, under the assumption that the discontinuity is at $0$, takes
the form
\begin{align}\label{DefDelta}
      \Delta &= \ell(0) [\ell'(0^{+})-\ell'(0^{-})] \int_0^1 \ell^{-2}(t) dt
\end{align}
where the second factor amounts to the velocity jump at $0$.  In
particular, we proved that the motion of the particle is chaotic for
large energies if $\Delta\not\in[0,4]$ and it is regular for large
energies if $\Delta\in(0,4)$.

While the large energy dynamical behavior depends only on the average
value of $\ell^{-2}$ and on the values of $\ell$ and its derivative at
the moment of jump (according to~\eqref{DefDelta}), the dependence of
the small energy dynamics on $\ell$ is more delicate.  It turns out
that the following property is sufficient to ensure stochastic
behavior for all energies.
\begin{mydef}
  A Fermi--Ulam model is said to be \emph{dispersing} if there exists
  $\cK>0$ so that $\ell''(t)\geq \cK$ for all $t$ where $\ell''$ is
  defined.
\end{mydef}
In this paper we  study the dynamics of dispersing
Fermi--Ulam models.  Note that for dispersing models, the value of
$\Delta$ defined by~\eqref{DefDelta} is necessarily negative.  Indeed,
the first and the last factors are positive while the second factor is
negative because periodicity implies that $\ell'(0^{-})=\ell'(1^{-})$
and the dispersing property implies that $\ell'(t)$ is increasing on
its interval of continuity.  Thus, according to~\cite{fum}, dispersing
Fermi--Ulam models are indeed stochastic for large energies.  The goal
of this paper is to show that stochasticity holds for \emph{all}
energies: we will prove that such systems are \emph{ergodic}.

To fix ideas, we take $\ell$ to be defined on the fundamental domain
$[0,1]$.  We assume that $\ell$ is $C^5$-smooth on $(0,1)$ and that it
can be smoothly extended to some neighborhood of $(0,1)$.  We assume the
fixed wall to be at the coordinate $z = 0$, and the coordinate of the
moving wall at time $t$ to be $z = -\ell(t)$.  Let $\exph$ denote the
\emph{extended phase space} of the system, defined as
\begin{align*}
  \exph=\{X=(t,z,v)\in\bR^3\st -\ell(t)\leq z\leq 0\}.
\end{align*}
where $z$ denotes the opposite of the distance between the point mass
and the fixed wall, $v$ is its velocity, with the positive direction
pointing away from the moving wall. The dynamics of the system is
described by the Hamiltonian flow $\flow{s}:\exph\to\exph$, which acts
on $\exph$ preserving the volume form
$\deh t\wedge \deh z \wedge \deh v$ (see Section~\ref{s_hyperbolicity}).

It will be more convenient to describe the dynamics on a suitable
Poincar\'e section.  Define the \emph{collision space}
$\csp=[0,1]\times[0,\infty) \ni x=(\ct,w)$.  The collision map
$\cm:(\ct,w)\mapsto(\ct',w')$ can be described as follows: a point
mass which leaves the moving wall at time $(\bmod\; 1)$ $\ct$ with
velocity $w$ relative to the moving wall will have its next collision
with the moving wall at time $(\bmod\; 1)$ equal to $\ct'$ and will
leave the moving wall with relative velocity $w'$ (which is thus
called \emph{post-collisional relative velocity}).  The invariant
volume form $\deh t\wedge \deh z \wedge \deh v$ induces an invariant
measure $\mu$ for $\cm$ where $$d\mu=(v+\dot{\ell}(t)) dv\wedge dt =
w\,dw\wedge dr.$$

Due to presence of singularities (the issue will be covered in detail
in Section~\ref{s_singularities}), the map $\cm$ and its iterates are
not defined everywhere. It is fortunately simple to show that the
singularity set is a $\mu$-null set (namely, a countable union of
smooth curves). Therefore the dynamics is well defined $\mu$-almost
everywhere, which is, in fact, all we need for the study of
statistical properties of the system.

In~\cite{fum} we proved that every dispersing Fermi--Ulam models is
recurrent, that is, $\mu$-almost every point eventually visits a region
of bounded velocity; moreover, we showed that such systems are
(non-uniformly) hyperbolic for large velocities.

We now state the main result of the present work. 
\begin{mthm}
  Dispersing Fermi--Ulam models that are \emph{regular at infinity}
  are ergodic.
\end{mthm}

\emph{Regularity at infinity} is a technical condition which allows
to control the combinatorics of collisions at infinity (see
Section~\ref{SS-VirtualComplexity} for the definition).  For the
moment we note that this property depends only on the parameter
$\Delta$ defined by \eqref{DefDelta}.  We will show in the appendix
that this condition may fail at most for countably many values of
$\Delta$.  In particular all dispersing Fermi--Ulam models with
$|\Delta|>\frac{1}{2}$ are regular at infinity (see Remark
\ref{RkDeltaHalf}).

Consider, as an example, piecewise quadratic motions studied in
\cite{U}.  Thus we assume that
\begin{align*}
  \ell(t)=1+a \left(\{t\}-\frac{1}{2}\right)^2,
\end{align*}
where $\{\cdot\}$ denotes the fractional part\footnote{Here the time scale is fixed
by the requirement that the motion is 1 periodic and spatial scale
 is fixed by the requirement that
 $\displaystyle \ell\left(\frac{1}{2}\right)=1.$}
. 
Here $a$ is a real
number that we assume to be greater than $-4$ so that $\ell(t)>0$ for
all $t$.  In this example we have $\ell''(t) = 2a$, thus the model is
dispersing if and only if $a>0$.  In this case 
one can compute (see~\cite{fum}) that 
\begin{align*}
  |\Delta|(a)=a+\frac{\sqrt{a} (a+4)}{2} \arctan\left(\frac{\sqrt{a}}{2}\right).
\end{align*}
Studying this function we see that $|\Delta|(a)>\ifrac{1}{2}$ for
$a>\ifrac{1}{4}$.  Hence, the model is regular at infinity for all $a>0$
 except, possibly, a countable set
of values of $ a\in \left(0, \ifrac{1}{4}\right).$

The foregoing discussion shows that most dispersing Fermi--Ulam models
are ergodic.  It is possible that, in fact, all dispersing
Fermi--Ulam models are ergodic, but the proof of this would require
new ideas.  On the other hand, the assumption that the Fermi--Ulam
model is dispersing is essential.  For example, for piecewise
quadratic wall motions with one singularity, then non-dispersing
models are not necessarily ergodic (see~\cite{fum}).

Recall that an orbit $\{(\ct_n,w_n)\}_{n\in\bZ}$ where
$(\ct_{n},w_{n}) = \cm^n(\ct_0,w_0)$ is said to be \emph{oscillatory}
if $\limsup w_n=\infty$ and $\liminf w_n<\infty$.
\begin{cor}\label{CrOsc}
  Almost every orbit of a dispersing Fermi--Ulam Model that is regular
  at infinity is oscillatory.
\end{cor}
\begin{figure}[!h]
  \includegraphics{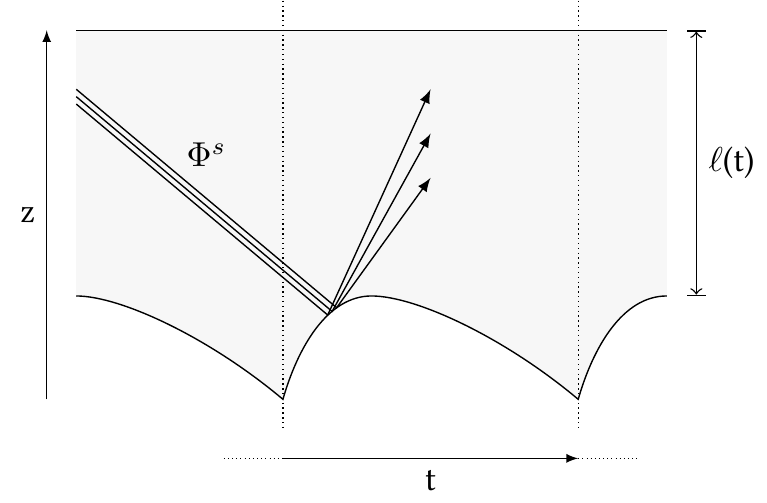}
  \caption{Dynamics of a dispersing Fermi--Ulam Model}
  \label{f_dispersing}
\end{figure}
The core observation made in this paper is that the dynamics of
dispersing Fermi--Ulam Models sports remarkable geometrical
similarities with the dynamics of planar dispersing
billiards\footnote{ This is one reason why we call such models
  dispersing. The other explanation in terms of geometric optics is
  given in Subsection~\ref{SSPSlopes}.}, although with an unusual
reflection law.  Moreover, our phase space $\csp$ is non-compact,
and the smooth invariant measure for $\cm$ is only $\sigma$-finite.
Ergodicity of systems with singularities, preserving a smooth infinite
measure is discussed for example in \cite{Simanyi, LenciI,
  LenciII}.  However, our system is significantly more complicated as
we explain below.

Recall first, that the study of ergodicity of uniformly hyperbolic
systems goes back to Hopf (see~\cite{Hopf}), who analyzed the case
where the stable and unstable foliations are smooth.  The Hopf
argument was extended to smooth uniformly hyperbolic systems\footnote{
  In such systems stable and unstable foliations are only H\"older
  continuous, see~\cite{Anosov}.} by Anosov and Anosov--Sinai
\cite{Anosov, AS}.  Hyperbolic systems with singularities are
discussed in \cite{Sinai, SC, KatStr, Pesin, LW}.  In order to use the
Hopf method (which is recalled in Section~\ref{ScErg}) one needs to
ensure that most points have long stable and unstable manifolds. A
classical way to guarantee this fact is to require that a small
neighborhood of the singularity set has small measure. In our case the
system is non-compact, and an arbitrary small neighborhood of the
singularity set has infinite measure, so a different method has to be
employed.  A more modern approach relies on the so called Growth
Lemma, developed in~\cite{BSC}, see~\cite{ChM} for a detailed
exposition.  The Growth Lemma implies that each unstable curve
intersect many long stable manifolds and vice versa.  The Growth Lemma
provides a significant improvement on the classical estimate on the
sizes of unstable manifolds and it has numerous applications to the
study of statistical properties, including mixing in finite and
infinite measure settings~\cite{C2, C3, ChZ, DN-Mech}, limit
theorems~\cite{C4, DSV}, and averaging~\cite{CD1, CD2, DN-Tube}.
However, in order to prove the Growth Lemma one needs to study the
structure of the singularity set in great detail.  It turns out that
the structure of singularities in dispersing Fermi-Ulam models is
quite complicated. Continuing the analogy with billiards, it
corresponds to billiards with infinite horizon billiards with corner
points.  The Growth Lemma for billiards with corners was established
only recently (see \cite{jmogy} for finite and \cite{BrownNandori} for
infinite horizon case).  Comparing to the aforementioned class of
billiards, an additional difficulty in our model is the lack of
hyperbolicity at infinity.  Indeed, when the velocity is large, the
travel time is short and the expansion deteriorates.  To address this
issue, an accelerated map was studied in~\cite{fum} (see
also~\cite{DF, D-icmp} for related results).  The main contribution of
this paper is to combine the analysis of the high energy regime
studied in~\cite{fum}, with the analysis of low energies (mostly based
on the ideas of~\cite{ChM} and the advances obtained in \cite{jmogy})
in order to prove a Growth Lemma valid for all energies.  
The Growth Lemma also
allows to prove absolute continuity of the stable and unstable
  laminations, which is a crucial ingredient in the proof of
  ergodicity via the Hopf argument.  Absolute continuity is proved in great
  generality for finite measure hyperbolic systems with singularities
  in~\cite{KatStr}, but their results cannot be applied to our infinite
  measure setting, so a different technique has to be employed.

We hope that the methods developed in this paper could be useful for
studying other hyperbolic systems preserving infinite measure (such
as, for example, the systems from \cite{LY, Zh}) and that our Growth
Lemma will be useful in studying more refined statistical properties
of dispersing Fermi--Ulam models.  

Since our analysis has many features in common with the study of
billiards, we will try, wherever possible, to employ the same notation
as in \cite{ChM}.  However, the arguments necessary for our system
require significant modifications in many places, which is,
  ultimately, the reason for the length of this paper. \\

The structure of the paper is as follows. In
Section~\ref{s_hyperbolicity} we describe basic properties of
dispersing Fermi--Ulam Models, including invariant cones and expansion
rates.  Section \ref{s_singularities} discusses the structure of the
singularities of the Poincar\'e map.  Section
\ref{sec:accel-poinc-map} is devoted to the high energy regime.  The
results of \cite{fum} are recalled and extended.  Section
\ref{ScDistortion} studies distortion of the collision map and obtains
regularity estimates on the images of unstable curves. The main
technical tool --the Growth Lemma-- is then proven in
Section~\ref{sec:expansion-estimate}. This lemma is used in
Section~\ref{sec:invariant-manifolds} to study the properties of
stable and unstable laminations which lead to the proof of Ergodicity
via the Hopf argument in Section~\ref{ScErg}.  Possible directions of
further research are discussed in Section~\ref{s_conclusions}.
Appendix~\ref{AppRegInf} contains the proof that for all but,
possibly, countably many values of $\Delta$, the corresponding model
is regular at infinity. The main issue is to show that certain
polynomials are not identically zero by estimating their values in a
perturbative regime.


\vskip3mm
{\bf A remark about our notation for constants.} We will use
the symbol $\Const$ to denote a constant whose value depends uniquely
on $\ell$ (which we assume to be fixed once and for all). The actual
value of $\Const$ can change from an occurrence to the next even on
the same line.
\section{Hyperbolicity}\label{s_hyperbolicity}
In this section we prove existence of invariant stable and unstable
cones for the dynamics and estimate the expansion of tangent vectors.
We begin with an essential property of Hamiltonian dynamics.
\subsection{Involution} \label{SSInvolution}%
Since Fermi--Ulam Models are mechanical systems, there exists a
time-{}reversing involution; on the other hand, since our system is
non-autonomous, we also need to change the time-dependence of the
Hamiltonian function, \ie we need to reverse the motion of the moving
wall.  For any $\ell$, let $\rell(\ct)=\ell(1-\ct)$ denote the
reversed motion, $\bar\exph$ the corresponding extended phase space
and $\rflow{s}:\bar\exph\to\bar\exph$ the flow map corresponding to
the reversed motion of the wall.  Define $\invo:\bR^3\to\bR^3$ so that
$\invo:(t,z,v)\mapsto(-t,z,-v)$.  Clearly, $\invo(\exph)=\bar\exph$;
moreover $\invo$ is an involution (\ie $\invo\circ\invo = 1$) which
anticommutes with the flow, \ie
\begin{align*}
  \invo\circ\flow{-s} = \rflow{s}\circ\invo.
\end{align*}
Notice a trivial but important fact, that $\ell''\geq\cK$ if and
only if ${\bar\ell}''\geq\cK$.

\subsection{Jacobi coordinates}
In billiards, in order to study of hyperbolic properties of the
system, it is convenient to change coordinates in $\exph$ to so-called
\emph{Jacobi coordinates} (see e.g.~\cite{Wojt}).  In our case this
step is not necessary, since, the coordinates $(z,v)$ turn out to be
the Jacobi coordinates of our system.  To fix ideas, let us write the
action of the flow map $\flow{s}$ on the extended phase space $\exph$
as \(\flow{s}:(t,z,v)\mapsto(t+s,z_s,v_s).\) If no collision occurs
between $t$ and $t+s$, then we have
\begin{align}\label{e_freeFlight}
  z_s&=z+s\cdot v& v_s&=v;
\end{align}
differentiating the above yields $\deh z_s = \deh z + s\deh v$ and
$\deh v_s = \deh v$, that is,
\begin{align*}
  \deh\flow{s}|_{(z,v)}=\matrixtt{1}{s}{0}{1}=:U_s.
\end{align*}
Assume now that between $t$ and $t+s$ there is exactly one collision
which occurs with the moving wall; the case of a collision with the
fixed wall is simpler and will be considered in due time as a special
case.  Let $\bar t$ be the time of the collision,
$\bar z=-\ell(\bar t \mod 1)$ be the position of the point mass at the
time of the collision, $\prv$ the pre-collisional velocity and $\pov$
the post-collisional velocity; finally let $\prs=\bar t-t$ and
$\pos=s-\prs=t+s-\bar t$ (see Figure~\ref{f_collision}).
\begin{figure}[!h]
  \def\svgwidth{7cm}
  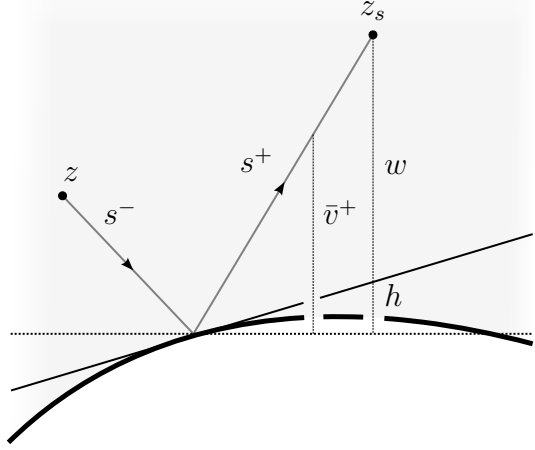
  \caption{Sketch of a collision with the moving wall.}
  \label{f_collision}
\end{figure}
Then:
\begin{align*}
  z &= \bar z - \prs\prv& z_s &= \bar z + \pos\pov\\
  v &= \prv = \slope - w & v_s &= \pov = \slope + w,
\end{align*}
where $\slope(\ct)=-\ell'(\ct)$ denotes the velocity of the moving
wall at time $\ct$ (\ie the slope of the boundary at the point of
collision).  Moreover, let $\curv(\ct)=\ell''(\ct)\geq\cK$ be the
opposite\footnote{ This choice of signs reflects the analogous choice
  which is usually made in the billiard literature.} of the
acceleration of the wall at time $\ct$; then:
\begin{align*}
  \deh\bar t &= \deh\ct &
                          \deh\bar z &= \slope \deh\ct &
                                                         \deh\slope &= -\curv \deh\ct.
\end{align*}
We thus obtain
\begin{subequations}\label{e_jacobiToCollision}
  \begin{align}
    \deh z &=(\slope-\prv)\deh\ct - \prs\deh \prv&
                                                   \deh z_s&=(\slope-\pov)\deh\ct + \pos\deh \pov\\
    \deh v&=-\curv\deh\ct - \deh w &
                                     \deh v_s&=-\curv\deh\ct + \deh w.
  \end{align}
\end{subequations}
We want to study what happens exactly during a collision, therefore we let $\prs,\pos\to
0^+$ and eliminate $\deh\ct$ and $\deh w$, obtaining:
\begin{align*}
  \deh \poz &= -\deh \prz & \deh \pov &= -\cpar\deh \prz- \deh \prv.
\end{align*}
Here $\DS \deh \prz = \lim_{\prs\to0^+}\deh z$ and
$\DS \deh \poz = \lim_{\pos\to0^+}\deh z_s,$ and we defined the
\emph{collision parameter} $\cpar=2\curv/w>0$ following the usual
notation and terminology of billiards.  From the above expression it is
clear that some special care is needed to deal with collisions with
small $w$. If $w=0$ we say that we have a \emph{grazing} collision. Such
collisions give rise to singularities, as will be explained in detail
later.  Notice that collisions with the fixed wall yield the same
formula with $\cpar=0$.

Define now:
\begin{align*}
  L_\cpar&:=\matrixtt{1}{0}{\cpar}{1}.
\end{align*}
Let us denote by $\fft$ the time elapsed before the next collision with
the moving wall (including grazing collisions). We can write the
differential $\deh\flow{\fft}|_{(z,v)}$ as the product
\begin{equation}\label{e_Differential}
  {\deh\flow{\fft}|}_{(z,v)}=  {(-1)}^{n_\fixw+1}
  L_{\cpar}U_{\fft}
\end{equation}
where $n_\fixw$ is the number of collisions with the fixed wall occurring between time $t$
and $t+\fft$, which can be either $0$ or $1$.
\subsection{Invariant cones}
\label{s_invariantCones}
(See \cite[Section~3.8]{ChM}).
Since we are
dealing with matrices acting on $\bR^2$, we will find convenient to deal
with slopes, rather than vectors; slopes in Jacobi coordinates will be
denoted by $\pslope=\delta v/\delta z$ and will be called
\emph{p-slopes}.  A (non-degenerate) matrix acts on slopes as a
(non-degenerate) M\"obius transformation. In particular, let
$J:\reals\setminus\{0\}\to\reals\setminus\{0\}$ denote the inversion
$x\mapsto x\inv$ and let $T_\alpha$ denote the translation
$x\mapsto x+\alpha$, for $\alpha\in\bR$.  Then $U_\fft$ induces the map
$J\circ T_\fft\circ J$, and $L_\cpar$ the map $T_\cpar$, that is:
\begin{align}\label{e_slopeEvolution}
  U_\fft  & :\pslope\mapsto{(\pslope\inv+\fft)}\inv & 
  L_\cpar & :\pslope\mapsto\pslope + \cpar
\end{align}
so we can rewrite \eqref{e_Differential} for p-slopes as follows:
\begin{align}\label{e_evolutionB}
  \pslope\mapsto [T_{\cpar}\circ J\circ T_{\fft}\circ J]\,\pslope.
\end{align}
The above formula immediately shows that the increasing cone
$\{\pslope>0\}$ is forward-invariant\footnote{ In fact $J$ clearly
  preserves such cone; moreover $\fft>0$ by definition and $\cpar>0$ by
  our hypotheses, which implies that also $T_\fft$ and $T_\cpar$
  preserve the increasing cone.}.  By the properties of the involution,
it is also clear that the decreasing cone \{$\pslope<0$\} is invariant
for the time-reversed flow.  It is not difficult to express the
invariant cones in collision coordinates. Namely let $\cslope$ denote
the slope of a vector in collision coordinates, that is
$\cslope = \delta w/\delta \ct$.  Then, using equations
\eqref{e_jacobiToCollision}, we obtain
\begin{equation}\label{e_slopeCollision}
  \cslope = - \curv - \prpslope w = \curv - \popslope w,
\end{equation}
where $\prpslope$ and $\popslope$ denote respectively the
pre-collisional and post-collisional p-slopes.  Thus the cone
$\{\cslope \leq -\cK\}$ (induced by $\prpslope\geq 0)$ is forward
invariant and, correspondingly, $\{\cslope\geq\cK\}$ (induced by
$\popslope\leq0$) is backward invariant.
\begin{mydef}\label{def:invariant-cones}
  Let the \emph{unstable} and \emph{stable cone field} be,
  respectively:
  \begin{align*}
    \coneu_x&=\{(\delta\ct,\delta w)\in\tang_x\csp\st -\infty < \delta w/\delta\ct \leq -\cK\}\\
    \cones_x&=\{(\delta\ct,\delta w)\in\tang_x\csp\st \cK\leq\delta w/\delta\ct < \infty\}.
  \end{align*}
  A curve is said to be an \emph{unstable curve}, or u-curve (resp.\ a
  \emph{stable curve} or s-curve) if the tangent vector at each point
  is contained in $\coneu$ (\resp $\cones$).  A curve (either stable
  or unstable) curve is said to be \emph{forward oriented} if the
  tangent vector at each point has a positive $\ct$-component.
\end{mydef}

  \begin{rmk}
    Observe that in our system unstable curves are decreasing and
    stable curves are increasing.  This, unfortunately, is
    the opposite of the situation that arises in billiards.
  \end{rmk}

Conventionally, we consider curves to be the embeddings an open
intervals, \ie without endpoints.  By our previous arguments,
$\cm_*\coneu_x\subset\coneu_{\cm x}$ and
$\cm\inv_*\cones_x\subset\cones_{\cm\inv x}$.  Moreover
by~\eqref{e_Differential} we gather that a forward-oriented unstable
(\resp stable) curve is sent by $\cm$ (\resp $\cm\inv$) to a
forward-oriented unstable (\resp stable) curve, if the ball has a
collision with the fixed wall between the two collisions with the
moving wall and to a backward-oriented unstable (\resp stable) curve
otherwise.

Further, define the two closed cones\footnote{ In the following
  definitions, with $\delta w/\delta r = \infty$ we allow vectors to be
  vertical.}
\begin{subequations}
  \begin{align}\label{e_definitionClosedCones}
    \conep_x&= \{(\delta\ct,\delta w)\in\tang_x\csp\st
              0 \leq \delta w/\delta\ct \leq \infty\}\\
    \conen_x&= \{(\delta\ct,\delta w)\in\tang_x\csp\st  -\infty \leq \delta
              w/\delta\ct \leq 0\}
  \end{align}
\end{subequations}
and observe that by~\eqref{e_slopeCollision} we have
\begin{equation}\label{PCSlopes}
  \popslope=\frac{\curv-\cslope}{w}, \quad
  \prpslope=\frac{-\curv-\cslope}{w}.
\end{equation}
From the above equations it follows easily that
\begin{align}\label{e_decreasingCone}
  \cm_*\conen_x&\subset \coneu_{\cm x}&
  \cm_*\inv\conep_x&\subset \cones_{\cm\inv x};
\end{align}
in particular, also in $(\ct,w$)-coordinates we have that the decreasing
cone field $\conen_x$ is forward invariant and the increasing cone field
$\conep_x$ is backward invariant.
\subsection{Geometrical interpretation of p-slopes} \label{SSPSlopes}
We have the following geometrical interpretation of
invariant cones in Jacobi coordinates: vectors in the tangent space
correspond to infinitesimal wave fronts; if $\pslope>0$ then the front
is dispersing, \ie{} nearby trajectories tend to get separated when
flowing in positive time.  Correspondingly $\pslope<0$ corresponds to
trajectories which would separate when flowing in negative time, \ie to
trajectories which are focusing in positive time. The case $\pslope=0$
corresponds to flat fronts, whereas the case $\pslope=\infty$
corresponds to a focused front (\ie all trajectories are emitted from
the same point).

\subsection{Expansion}\label{s_Jacobian}
Jacobi coordinates are convenient coordinates on the tangent space to
the collision space $\csp$.  By~\eqref{e_jacobiToCollision} it follows
that
\begin{align*}
  \vectt{\deh z}{\deh v}=\matrixtt{w}{0}{\curv}{-1}\vectt{\deh\ct}{\deh w}, \;\;
  \vectt{\deh\ct}{\deh w}=\matrixtt{w\inv}{0}{\curv w\inv}{-1}\vectt{\deh z}{\deh v}.
\end{align*}

For any $x\in\csp$, let $\fft(x)\ge 0$ denote the time elapsed until the
following (possibly grazing) collision with the moving wall.
Let us consider a vector of p-slope $\pslope^{+} = \pslope$ at $x$;
then~\eqref{e_freeFlight} implies that, during a flight of duration
$\fft$, we have $\deh z_\fft=(1+\fft\pslope)\deh z$ and
$\deh v_\fft = \deh v$. On the other hand, at a collision, we have
$|\deh \poz|=|\deh\prz|$.  Define the metric $|\deh z|$
for (non-vertical) tangent vectors (the so-called $p$-metric). Then we obtain
that, if the p-slope of a vector $v$ is $\pslope$, its expansion by the
collision map in the p-metric is given by
\begin{equation} \label{ExpPMetric}
  \frac{|\deh z_{\fft(x)}|}{|\deh z|}=1+\fft(x)\pslope.
\end{equation}
If $v_{n}\in\coneu_{x_{n}}$ (i.e.\ $\pslope>\cpar_{n}$), since
$\cpar_{n}$ is bounded below by $2\cK/w_{n}$ we obtain the lower bound
\begin{equation}\label{e_expansionLowerBound}
  \frac{|\deh z_{n+1}|}{|\deh z_n|}\geq 1+\frac{2\cK}{w_n}\fft_n
\end{equation}
where $\fft_n=\fft(x_n)$.  Observe that~\eqref{e_expansionLowerBound}
does not ensure any uniformity for the expansion of unstable vectors in
the $p$-metric.
In fact for large relative velocities $\fft\sim w\inv$. Additionally,
$\fft$ can be arbitrarily small also for small relative velocities,
because of the possibility of rapid subsequent collisions with the
moving wall.

We will see later that both these inconveniences can be
circumvented by defining an adapted metric and inducing on a suitable
subset of the collision space (see Proposition~\ref{p_propertiesAdm}).
However, before doing so, it is necessary to study singularities of
our system.

\section{Singularities}\label{s_singularities}
The existence of invariant cones places Fermi--Ulam Models into the
class of {\em hyperbolic systems with singularities}. This class also
contains piecewise expanding maps, dispersing billiards, and bouncing
ball systems (see~\cite{ChM, LW, T2, W-particles} and references
therein).  In hyperbolic maps with singularities, there is a fundamental
competition between expansion of vectors inside the unstable cone and
fracturing caused by singularities.  If fragmentation
prevails, such maps can indeed have poor ergodic properties (see
e.g.~\cite{T1}).  Our goal is to show that this does not happen for
(most) dispersing Fermi--Ulam Models; this will be accomplished with
the proof of the Growth Lemma in Section~\ref{sec:growth-lemma}.

In this section, we collect preliminary information about the geometry
of singularities\footnote{The reader familiar with dynamics of
  dispersing billiards will recognize certain distinctive features of
  the geometry of singularities (see e.g.~\cite[Section~2.10]{ChM}).}
of the collision map $\cm$.
\begin{rmk}
  In the following, if $X\subset\csp$, we will use the notation
  $\intr X$ ({resp.\ $\clo X$, $\partial X$}) to denote the
  topological interior (resp.\ closure, boundary) of the set $X$
  \emph{with respect to the topology on $\reals^2$} (and not with
  respect to the relative topology on $\csp$).
\end{rmk}

\subsection{Local structure}\label{ss_singularitiesLocalStructure}
Let us recall the definition of the collision map: $\cm(\ct,w) = (\ct',w')$
means that a point mass that leaves the moving wall at time $\ct$ with
velocity $w$ relative to the moving wall will have its next collision
with the moving wall at time given $(\bmod \; 1)$ by $\ct'$ and will
leave the moving wall with relative velocity $w'$.  Recall moreover that
$\fft:\csp\to\reals_{\ge0}$ is the (lower semi-continuous) function
which associates to $(\ct,w)$ the time elapsed before the next (possibly
grazing) collision with the moving wall.  If one considers the
\emph{preceding} collision rather than the following one in the above
discussion, we obtain the definition of the inverse map $\cm\inv$.

We define the \emph{singularity set} $\sing0$ to be the
boundary $\partial\csp$, i.e.:
\begin{align*}
  \sing{0} = \partial\csp =\{w = 0\}\cup \{\ct \in \{0,1\}\}.
\end{align*}
$\sing{0}$ is the set of points in the collision space for which the
point mass either just underwent a \emph{grazing} collision (when
$w = 0$), or it just left the moving wall at an instant in which the
motion of the wall is not smooth (when $r\in\{0,1\}$).

Let $x = (\ct,w)\in\csp$; observe that $\fft(x)$ is defined for all
$x\in\csp$.  There are three possibilities: the trajectory leaving the
moving wall at time $\ct$ with relative velocity $w$ may have its next
collision with the moving wall

\begin{enumerate}
\item \label{i_regular} with nonzero relative velocity at an instant
  when the motion of the wall is smooth.  In this case $\cm$ is
  well-defined on $x$ and $\cm(x)\in\intr\csp = \csp\setminus\sing{0}$.
\item \label{i_grazing} with \emph{zero} relative velocity at an
  instant when the motion of the wall is smooth.  In this case $\cm$
  is well-defined, but might\footnote{In fact it will be always be
    discontinuous, except in the case described by
    Lemma~\ref{l_continuityAtxc} } be discontinuous at $x$ (and
  it turns out that $\DS \limsup_{x'\to x}|d\cm| =\infty$).  We have
  \begin{align*}
    \cm(x)\in\{r\in(0,1),\ w = 0\}\subset\sing0;
  \end{align*}
moreover $\fft$ is also
  discontinuous at $x$.
\item \label{i_corner} when the motion of the wall is \emph{not smooth};
  $\fft$ is continuous at $x$, but $\cm(x)$ is \emph{not} defined
  (because the post-collisional velocity is undefined).
\end{enumerate}
We can then define
\begin{align*}
  \singForward{} =\sing0 \cup \{x\in\csp\st \text{items~\ref{i_grazing} and~\ref{i_corner} take place}\}.
\end{align*}
The above also applies to the classification of the \emph{previous}
collision, which leads to the analogous definition of
$\singBackward{}$.  Observe that $\cm$ (\resp $\cm\inv$) is
well-defined and smooth on $x$ if and only if
$x\in\csp\setminus\singForward{}$ (\resp
$x\in\csp\setminus\singBackward{}$).  We let $\sing1 = \singForward{}$
(\resp $\sing{-1} = \singBackward{}$) and for $n > 0$ we define, by
induction:
\begin{align*}
  \sing{n+1}  & =\sing{n}\cup\cm\inv(\sing{n}\setminus\singBackward{}) & 
  \sing{-n-1} & =\sing{-n}\cup\cm(\sing{-n}\setminus\singForward{}).
\end{align*}
Finally, let $\sing{+\infty} = \bigcup_{n\ge0}\sing n$ and
$\sing{-\infty} = \bigcup_{n\le0}\sing n$.  Notice that, for any
$k\in\bZ$, the map $\cm^{k}$ is well-defined and smooth on $x$ if and
only if $x\in\csp\setminus\sing{k}$.

\begin{lem}[Local structure of singularities]\label{l_localSing}
  For $k > 0$ the set $\sing{k}\setminus\sing{0}$ (\resp
  $\sing{-k}\setminus\sing0$) is a union of smooth \emph{stable}
  (\resp \emph{unstable}) curves.  In particular $\sing{k}$ (\resp
  $\sing{-k}$) is a union of smooth curves tangent\footnote{ Here and
    below we say that a curve is tangent to a cone field if the
    tangent to the curve belongs to the cone at every point.} to the
  cone field $\conep$ (\resp $\conen$).
\end{lem}
We will prove the above statement for $\sing{-k}$.  The analogues for
$\sing{k}$ can be obtained using the involution.
Moreover, since the unstable cone is $\cm$-invariant, it suffices to
prove the statement for $\sing{-1} = \singBackward{}$.
\begin{sublem}%
  Let $x\in\singBackward{}\setminus\sing0$, then the p-slope of
  $\singBackward{}$ at $x=(\ct,w)$ is given by
  \begin{subequations}\label{e_slopeSingularity1}
    \begin{align}
      \pslope &= \cpar_0(x)+1/\fft_{-1}(x)>0.\label{e_pslopeSingularity1}
                \intertext{ Equivalently, the slope in collision coordinates is given by }
                \cslope &= -\curv(\ct) - w/\fft_{-1}(x)\leq -\cK.\label{e_cslopeSingularity1}
    \end{align}
  \end{subequations}
\end{sublem}
\begin{proof}
  Observe that each curve in $\singBackward{}$ is formed by trajectories
  for which either $\ct_{-1}=0$, or $w_{-1}=0$.  In the first case, such
  trajectories draw a wave front which is emitted from a single point,
  therefore it is immediate that $\pslope_{-1}^+=\infty$.  We claim that
  also in the second case $\pslope_{-1}^+=\infty$, which then
  immediately implies equations~\eqref{e_slopeSingularity1}
  using~\eqref{e_evolutionB}.  In fact consider two nearby trajectories
  which leave the wall with zero relative velocity at times $\ct$ and
  $\ct'=\ct+\delta\ct$.  Let $v$ and $v'=v+\delta v$ be the
  corresponding outgoing velocities; observe that
  $\delta v\sim\curv\delta \ct$.  On the other hand, the second
  trajectory at time $\ct$ will have height $z'=z+\delta z$,
  where 
  $\delta z\sim\curv\delta\ct^2$; we conclude that
  $\displaystyle \pslope^+_{-1}=\lim_{\delta\ct\to0}\delta v/\delta
  z=\infty$.
\end{proof}
\begin{rmk}
  The corresponding formulae for the slopes of $\singForward{}$ at any
  $x=(\ct,w)\in\singForward{}\setminus\sing0$ are
  \begin{subequations}\label{e_slopeSingularityForward}
    \begin{align}
      \pslope &= -1/\fft_{0}(x) < 0\\
      \cslope &= \curv(\ct) + w/\fft_0(x) > \cK.
    \end{align}
  \end{subequations}
\end{rmk}

\subsection{Global structure.}\label{s_globalSingularities}%
We now begin the description of the global structure\footnote{ The
  structure depends on our simplifying hypotheses on the motion of the
  wall. If $\ell$ had more than one break point, the set $\sing{1}$
  would have a much more complicated structure, although its key
  features will be similar.  Moreover, the structure of $\sing{k}$ for
  $k>1$ would also be essentially similar in the case we have multiple
  breakpoints.} of the singularity sets $\singBoth$.  Let us first
introduce some convenient notation.

Let $\ellu=\max\ell=\ell(0)=\ell(1)$.  Since $\ell$ is strictly
convex, it has a unique critical point (a minimum), which we denote by
$\ctc\in(0,1)$.  Set $\elll=\min\ell=\ell(\ctc)$ and $\xc=(\ctc,0)$.
Recall that $\slope(\ct)=-\ell'(\ct)$ and define
\begin{align*}
  \slopel=\min\slope=\lim_{\ct\to1}\slope(\ct)<0,\; \slopeu=\max\slope=\lim_{\ct\to0}\slope(\ct)>0,
  \;
  \slopeb=\slopeu-\slopel>0.
\end{align*}
We remark that in this new notation, we can write~\eqref{DefDelta} as
\begin{align*}
  \Delta &= -\ellu \slopeb\int_0^1 \ell^{-2}(t) dt.
\end{align*}

Observe that the point $\xc$ is a fixed point for the dynamics: it
corresponds to the configuration in which the point mass stays put at
distance $\elll$ from the fixed wall, and the moving wall hits it with
speed $0$ at times $\ctc+\integers$.  Moreover, points arbitrarily close
to $\xc$ may have arbitrarily long free flight times \ie
\begin{align*}
  \limsup_{x\to\xc}\fft(x) = \infty.
\end{align*}

Next, we identify a special region of the phase space. It is clear
that, if the relative velocity of the point mass at a
collision with the moving wall is sufficiently large, then the particle
will necessarily have to bounce off the fixed wall before colliding
again with the moving wall. On the other hand, if the velocity at a
collision with the moving wall is comparable with the velocity of the
wall itself, then the particle could have two (or a priori more)
consecutive collisions with the moving wall before hitting the fixed
wall.\footnote{ In the case of billiards this corresponds to so-called
  \emph{corner series}.}
\begin{mydef}
  A collision with the moving wall is called a \emph{recollision} if it
  is immediately preceded by another collision with the moving wall; it
  is called a \emph{simple collision} otherwise.  We denote with
  $\reco\subset\csp$ the open set of points corresponding to
  regular\footnote{ That is, we do not take into account points that
    undergo a grazing collision on either the recollision or on the
    previous collision; moreover we do not take into account collisions
    with the singular point $\xc$. } recollisions and let
  $\backReco=\cm\inv\reco$.
\end{mydef}
The following lemma provides a description of the sets $\reco$ and $\backReco$.
\begin{lem}\label{l_forwardReco}
  Let $\singReco^- = \cm([\ctc,1]\times\{0\})$ and
  $\singReco^+ = \cm\inv([0,\ctc]\times\{0\})$.  Then:
  \begin{itemize}
  \item[(a1)] $\singReco^-$ is a connected u-curve that leaves
    $(0,\slopeb)$ with slope $-\infty$ and reaches $\xc$ with slope
    $-\curv({\ctc})$;
  \item[(a2)] $\reco$ is the interior of the curvilinear triangle whose sides
    are the (horizontal) segment $[0,\ctc]\times\{0\}$, the (vertical)
    segment $\{0\}\times[0,\slopeb]$ and $\singReco^-$.
  \item[(b1)] $\singReco^+$ is a connected s-curve that leaves $\xc$ with
    slope $\curv(\ctc)$ and reaches $(1,\slopeb)$ with slope $\infty;$
  \item[(b2)] $\backReco$ is the interior of the curvilinear triangle whose
    sides are the (horizontal) segment $[\ctc,1]\times\{0\}$, the
    (vertical) segment $\{1\}\times[0,\slopeb]$ and $\singReco^+$.
  \end{itemize}
\end{lem}
\begin{figure}[!h]
  \def\svgwidth{4.5cm}
  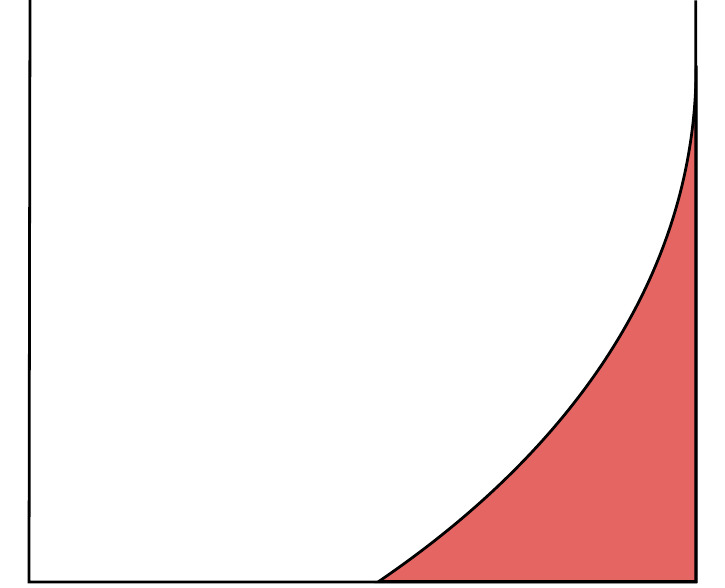
  \caption{The recollision region $\backReco$.}
  \label{f_recollisionRegion}
\end{figure}
\begin{proof}
  We prove part (a). Part (b) follows from part (a) and the properties
  of the involution.  Let $U$ denote the curvilinear triangle in
  $(t,z)$-space bounded by $\Gamma_1$--the wall trajectory for
  $t\in [r_C,1]$, $\Gamma_2$--the wall trajectory for $t\in [1, r_C+1]$
  and $\Gamma_3$--the horizontal segment joining the highest points of
  those trajectories. By our convexity assumption on $\ell$ and
  elementary geometrical considerations, any trajectory $x = (r,0)$ with
  $r\in[\ctc,1]$ stays inside $U$ hence its next collision necessarily
  occurs on the moving wall.  This in turn implies that the u-curve
  $\singReco^- = \cm([\ctc,1]\times \{0\})$ is connected (since it
  cannot be cut by singularities).  It is then trivial to check that
  $\cm(1,0)=(0,\slopeb)$, which implies that $\singReco^-$
  connects $(0,\slopeb)$ with the fixed point $\xc$.  Our statements
  about the tangent slope at $(0,\slopeb)$ and $\xc$ immediately follow
  from~\eqref{e_cslopeSingularity1} observing that
  \begin{align*}
    \lim_{\ct\to 1}\fft((\ct,0)) &= 0&
    \lim_{\ct\to\ctc^+}\fft((\ct,0)) &= 1.
  \end{align*}

  It remains to prove (a2). First, consider a collision that occurs at a
  point $(\ct,w)$ with $\ct \in (\ctc,1]$: the incoming trajectory lies
  above the tangent to $\ell$ at $\ct$, which, in turn, lies above the
  graph of $\ell$ (for $\ct' < \ct$) by convexity of $\ell$.  In
  particular it is above the graph of $\ell$ at time $\ctc$, that is, it
  gets above the maximal height of the wall and its velocity at time
  $\ctc$ is negative.  Hence, necessarily, the preceding collision will
  occur with the fixed wall, proving that
  $\reco\subset[0,\ctc]\times\reals^+$.  It remains to check that any
  point in $[0,\ctc]\times\reals^+$ lying below $\singReco^-$
  corresponds to a recollision, whereas any point lying above
  $\singReco^-$ corresponds to a single collision.  So pick
  $r\in [0, \ctc]$.  By (a1) there is $\ct^*\in [\ctc, 1]$ such that
  $\cF(\ct^*,0)=(r, w^*)\in \singReco^-$.  Let $\Gamma$ be the trajectory
  from $(\ct^*,0)$ to $(r, w^*)$ and $V\subset U$ be the region bounded
  by $\Gamma_1, \Gamma_2,$ and $\Gamma$.  There are two cases.
  \begin{itemize}
  \item[(i)] $w\leq w^*$.  Then the backward trajectory of $(r, w)$ is
    contained in $V$ and so it crosses $\Gamma_1$ before colliding
    with the fixed wall.
  \item[(ii)] $w\geq w^*$.  Then the backward trajectory of $(r, w)$ is
    above $\Gamma$ so if it crossed $\Gamma_1$ this would happen at
    some time $\ct'<\ct^*$.  However by convexity, any orbit starting
    at time $\ct'$ lies strictly above $\Gamma$ so it can not hit the
    moving wall at time $r$.
  \end{itemize}
  This concludes the proof.
\end{proof}
\begin{rmk}\label{r_thereCanBeOnlyOne}
  The above lemma implies that $\clo\backReco\cap\clo\reco=\{\xc\}$, \ie
  the number of consecutive collisions with the moving wall is at most
  $2$ (except for the singular point $\xc$, which is a fixed point of
  the dynamics).
\end{rmk}
\begin{rmk}\label{r_boundFreeFlight}%
  Let $x_0=(r_0,w_0)$; if $x_0\not\in\clo\backReco$, then $\fft(x_0)$
  satisfies the bound:
  \begin{align}\label{e_boundFreeFlight}
    \frac{2\elll}{w_1+\slope(\ct_1)} = \frac{2\elll}{w_0-\slope(\ct_0)} \leq \fft(x_0) \leq
    \frac{2\ellu}{w_0-\slope(\ct_0)} = \frac{2\ellu}{w_1+\slope(\ct_1)}.
  \end{align}
  \eqref{e_boundFreeFlight} follows since $w_0-h(r_0)=w_1+h(r_1)$ is
  the post-collisional absolute velocity of the point mass and
  $\elll\le\ell(\ct) < \ellu$.  Observe moreover that $w_0-h(r_0)>0$,
  otherwise the next collision would certainly be a recollision, since
  the absolute velocity would be non-positive.  On the other hand, if
  $x\in\backReco$, $\fft(x)$ may be arbitrarily small.
\end{rmk}%
We record in the following lemma an observation which will be useful
on several occasions.
\begin{lem}\label{lem:geometry}
  If $x = (\ct,w)$ is so that either $\fft(x) \geq 2$ or
  $\fft_{-1}(x)\geq 2$ then:
  \begin{align*}
    x \in \{w < \Const \tau^{-1/2},\; |\ct-\ctc|< \Const
    \tau^{-1/2} \}.
  \end{align*}
\end{lem}
\begin{proof}
  It suffices to prove the result under the assumption
  $\tau(x)\geq 2$, since the other case follows by applying the
  involution.  Since $\tau(x)\ge 2$, in particular $x\nin\cl\reco$;
  hence by~\eqref{e_boundFreeFlight} we gather
  \begin{equation} \label{AVDNu}
    0<w-h(\ct) \le 2\ellu/\tau.
  \end{equation}
  We also have $\ell(\ct)-\elll = \cO(1/\tau)$, since otherwise
  $(\ct, w)$ would be in the recollision region. Since $\ell$ has a
  critical point at $\ctc$, it follows that
  $|\ct-\ctc|\leq \frac{\bar C}{\sqrt{\tau}}$ giving the second
  inclusion.  It follows that
  $|h(\ct)|\leq \frac{\hat C}{\sqrt{\tau}}$.  Now the first inclusion
  follows from~\eqref{AVDNu}.
\end{proof}

Define $\csps^- = \clo(\csp\setminus\clo\reco)$ and
$\csps^+ = \clo(\csp\setminus\clo\backReco)$.  The curve $\singReco^-$
(\resp $\singReco^+$) is one among the unstable (\resp stable) disjoint
curves whose union form the set $\singBackward{}$ (\resp
$\singForward{}$); the other curves will cut $\csps^-$ (\resp $\csps^+$)
in countably many connected components, as we now describe\footnote{The
  structure of singularities for dispersing Fermi--Ulam Models is remarkably
  similar to the one described in~\cite[Section 4.10]{ChM} for the
  singularity portrait in a neighborhood of a singular point of a
  billiard with infinite horizon.  We refer to the discussion presented
  there for further insights; here we provide a qualitative description
  which however suffices for our purposes.}.  Let us first introduce
some convenient notation: we define the \emph{left boundary}
$\lboundary\csps^\pm = \{(\ct,w)\in\partial\csps^\pm \st
\ct\in[0,\ctc]\}$ and the right boundary
$\rboundary\csps^\pm = \{(\ct,w)\in\partial\csps^\pm \st
\ct\in[\ctc,1]\}$.  \newcommand{\singMap}{\Gamma}
\begin{lem}\label{LmSingF}
  There exist countably many $C^1$-smooth unstable curves
  $\{\singBackward{\indCell}\}_{\indCell = 0}^{\infty}$ with the following
  properties
  \begin{enumerate}
  \item $\singBackward{\indCell}\cap\singBackward{\indCell'} = \emptyset$ if $\indCell\ne\indCell'$.
  \item $\singBackward{} = \singReco^-\cup\bigcup_{\indCell = 0}^{\infty}\singBackward{\indCell}$.
  \item $\singBackward{0}$ is unbounded: its left endpoint approaches
    $(0,\infty)$ and the other endpoint is in $\rboundary\csps^-$.
  \item $\singBackward{\indCell}$ for $\indCell > 0$ is compact and joins
    $\lboundary\csps^-$ to $\rboundary\csps^-$.
  \item $\singBackward{\indCell}$ approaches $\xc$ for
    $\indCell\to\infty$; more precisely:
    \begin{align*}
      \singBackward{\indCell}\subset\{w < \Const \indCell^{-1/2},\;
      |\ct-\ctc|< \Const \indCell^{-1/2} \}.
    \end{align*}
  \item There exists $c > 0$ such that $\singBackward{\indCell}$ is
    tangent to the cone
    \begin{align*}
      \hat{\coneu_{\indCell}} = \{-\curv(\ct)-c\indCell^{-3/2}\le\delta
      w/\delta r\le -\curv(\ct)\}.
    \end{align*}
  \end{enumerate}
  The corresponding statements hold for $\singForward{}$ using the
  involution.
\end{lem}
\begin{proof}
  A point $x'$ can be in $\singBackward{}$ for two different reasons:
  its previous collision with the moving wall $x = (\ct,w)$ may have
  occurred either at an integer time (item~\ref{i_corner} in the
  definition of $\sing0$) or at a non-integer time with a grazing
  collision (item~\ref{i_grazing} in the definition of $\sing0$).  If
  $x'$ is a recollision, then $x'\in\singReco^-$ (and hence
  $\ct\in[\ctc,1]$ and $w = 0$), otherwise we can choose
  $x\in\lboundary\csps^+$.  

  For any $\indCell\in\integers_{\ge0}$ define
  $
  \sing0_\indCell = \{x\in\lboundary\csps^+\st\fft(x)\in[\indCell,\indCell+1]\}.
  $
  Notice that 
  $\cm$ is smooth in the interior of these curves\footnote{ Smoothness
    is obvious unless $(0,0)\in\intr\sing0_\indCell$; even in this case
    it holds true, and follows from arguments identical to the ones
    described in~\cite[after Exercise 4.46]{ChM}}.  We conclude that
  $\cm(\intr\sing0_\indCell)$ is a $C^1$-smooth unstable curve. Define
  \begin{align*}
    \singBackward{\indCell} = \cl\cm(\intr\sing0_\indCell).
  \end{align*}
  Items (a) and (b) then follow by construction.

  Next, it is easy to see that if $w$ is sufficiently large, then the
  trajectory will bounce off the fixed wall and hit back the moving wall
  after a short time $\fft\in(0,1)$; in particular $\sing0_0$ is
  unbounded while $\sing0_\nu$ and $\singBackward\nu$ are bounded for
  $\nu>0$.

  Next, as $w$ increases to $\infty$, the point $\cm (0,w)=(\ct', w')$
  where $\ct'$ is small and $w'$ is large.
  On the other hand when $x\in\sing0_0$ approaches the (only) boundary
  point of $\sing0_0$, the point $\cm x$ will necessarily tend to
  $\rboundary\csps^-$.  This proves item (c).  Item (d) follows from
  analogous arguments.

  Item (e) follows by applying Lemma~\ref{lem:geometry} to an
  arbitrary point in $\singBackward{\indCell}$.
  Finally, item (f) follows from~\eqref{e_cslopeSingularity1}
    and item (e).
\end{proof}

  \begin{lem}[Continuation property]
    For each $n\ne 0$, every curve $S\subset\sing n\setminus \sing0$
    is a part of some monotonic continuous (and piecewise smooth)
    curve $\sing{*}\subset\sing n\setminus \sing 0$ which
    terminates on $\sing0 = \partial\csp$.
  \end{lem}
  \begin{proof}
    It suffices to prove the property for $n > 0$, since the case
    $n < 0$ follows by the properties of the involution.  The
    statement holds for $n = 1$ by Lemma~\ref{LmSingF}; the statement
    then follows by induction by definition of $\sing n$: assume that
    $S\subset\sing {n+1}\setminus\sing n$.  Then, by construction, $S$
    terminates on either $\sing0$ or $\sing n$. However if it
    terminates on $\sing n$, then by inductive hypothesis it can be
    continued as a piecewise smooth curve to $\sing 0$.
  \end{proof}

The curves $\{\singBoth_{\indCell}\}_{\indCell\ge0}$ cut $\csps^\pm$ in
countably many connected components which we denote with
$\{\cell^+_\indCell\}$ (\resp $\{\cell^-_\indCell\}$) and we call
\emph{positive} (\resp \emph{negative}) \emph{cells}.  Indexing is
defined as follows: for $\indCell > 0$ we let $\cell^\pm_\indCell$
denote the component whose boundary contains $\singBoth_{\indCell-1}$
and $\singBoth_{\indCell}$ and let $\cell^\pm_0$ denote the remaining
cell.  The cells $\cell^+_\indCell$ admit also an intrinsic definition
as
\begin{align}\label{e_defCell}
  \cell^+_\indCell &= \intr\{x\in\csps^+ \st \ct(x)+\fft_0(x)\in(\indCell,\indCell+1)\};
\end{align}
observe that each positive cell is indexed by the number of boundaries
of fundamental domains which are crossed by the trajectory between the
current and the next collision.  A similar intrinsic characterization
can be given for the negative cells $\cell^-$.  We summarize in the
following lemma some properties of positive cells that follow from our
above discussion.

\begin{lem}[Properties of positive cells]\label{l_singularityStructure}
  \hskip0pt 
  \begin{enumerate}
  \item \label{i_disjointCells} The cells $\{\cell^+_\indCell\}_{\indCell\ge0}$ are open, connected and pairwise disjoint.
  \item \label{i_partitionCells} We have
    \begin{align*}
      \intr\csps^+\setminus\singForward{}=\bigcup_{\indCell=0}^\infty\cell^+_\indCell.
    \end{align*}
  \item \label{i_adjacentCells}
    $\clo\cell^+_\indCell\cap\clo \cell^+_{\indCell'} = \emptyset$ if
    $|\indCell-\indCell'| > 1$; moreover if
    $\bar x\in\clo\cell^{+}_{\indCell}\cap\clo\cell^{+}_{\indCell+1}$,
    we have either
    \begin{align*}
      \lim_{\cell^{+}_{\indCell}\ni x\to\bar x}\cm x&\in\{1\}\times\reals^+&
      \lim_{\cell^{+}_{\indCell+1}\ni x\to\bar x}\cm x&\in\{0\}\times\reals^+,
    \end{align*}
    or
    \begin{align*}
      \lim_{\cell^{+}_{\indCell}\ni x\to\bar x}\cm x&\in[0,1]\times\{0\}&
      \lim_{\cell^{+}_{\indCell+1}\ni x\to\bar x}\cm x&\in\singReco^{-}.
    \end{align*}

  \item for any $\bar\indCell$ there exists $\eps$ so that the ball of radius $\eps$ centered at $\xc$ does not
    intersect $\bigcup_{\indCell = 0}^{\bar\indCell}   \cell^+_\indCell$.
  \item \label{i_boundedCells} for $\indCell > 1$, we have $\cell^{+}_{\indCell}\subset\{w < \Const \indCell^{-1/2},\;
      |\ct-\ctc|< \Const \indCell^{-1/2} \}$.
  \end{enumerate}
\end{lem}

\begin{rmk}\label{r_singularityStructureBackwards}
  Using the involution, the above lemma also describes (with due
  modifications) the negative cells
  $\cell^-_\indCell = \cm\cell^+_\indCell$.
\end{rmk}

Despite the fact that the singular point $\xc$ is accumulated by
singularities (both forward and backward in time), we have the
following result.
\begin{lem}\label{l_continuityAtxc}
  For every $\eps > 0$ there exists a $\delta > 0$ so that
  \begin{align*}
    \cm(B(\xc,\delta)\setminus\sing1)\subset B(\xc,\eps).
  \end{align*}
\end{lem}
\begin{proof}
  If $x\in B(\xc,\delta)\setminus\sing1$ there are two possibilities;
  either $x\in B(\xc,\delta)\cap\backReco$ or
  $x\in B(\xc,\delta)\cap\cell^+_\indCell$ for some large $\indCell$.
  In the former case $\cm$ is continuous in $\backReco$ and
  $\DS \lim_{\backReco\ni x\to(\ctc,0)}\cm x = \xc$, so we only need to
  check the latter case.  However, if $x\in\cell^+_\indCell$, then, by
  definition $\cm x\in\cell^-_\indCell$ and we conclude the proof
  since the cells $\{\cell^-_\indCell\}$ also accumulate to $\xc$ by
  Lemma~\ref{l_singularityStructure}(e) and
  Remark~\ref{r_singularityStructureBackwards}.
\end{proof}
In view of Lemma~\ref{LmSingF}, a u-curve $W$ can in principle be cut by
singularities of $\cm$ in countably many connected components.\footnote{
  This problem is certainly familiar to the reader acquainted with the
  theory of dispersing billiards with infinite horizon.}  The next lemma
ensures that this may only happen in a neighborhood of the singular
point $\xc$.
\begin{lem}\label{l_finitelyManySectors}
  Let $x\in\csp\setminus\{\xc\}$. For any $l > 0$, the set $\sing{l}$
  cuts a sufficiently small neighborhood of $x$ in finitely many
  connected components.
\end{lem}
\begin{proof}
  Assume that for an arbitrarily small ball $\cU\ni x$ there exists
  $0 < l'\le l$ so that $\cU\setminus\sing {l'-1}$ has finitely many
  connected components and $\cU\setminus\sing {l'}$ has infinitely many.
  We conclude that there exists a connected component $\cU'$ of
  $\cU\setminus\sing {l'-1}$ which is cut by $\sing {l'}$ in infinitely
  many connected components.  By definition $\cm^{l'-1}$ is smooth on
  $\cU'$ and, by our assumption, $\cm^{l'-1}\cU'$ intersects infinitely
  many positive cells $\cell^+$.  We gather that there exists a sequence
  $x_m\in\cU'\cap\cm^{-(l'-1)}\cell^+_{\indCell_m}$, where
  $\indCell_m\to\infty$; by Lemma~\ref{l_singularityStructure} we have
  $\cm^{l'-1}x'_m\to\xc$, which by Lemma~\ref{l_continuityAtxc} implies
  that $x'_n\to\xc$, that is $\xc\in\cl\cU$.  Since $\cU$ can be taken
  to be arbitrarily small, we conclude that $x = \xc$.
\end{proof}
For $l_- \le 0 \le l_+$, define
$\sing{l_-,l_+} = \sing{l_-}\cup\sing{l_+}$: %
then $\csp\setminus\sing{l_-,l_+}$ is given by a (countable) union of
connected components.  A point $x\in\sing{l_-,l_+}$ is said to be a
\emph{multiple point of $\sing{l_-,l_+}$} if it belongs to the closure
of at least three such connected components; we denote the set of
multiple points of $\sing{l_-,l_+}$ by $\mult{l_-,l_+}$.
\begin{lem}\label{l_singularPointNotMultiple}
  The singular point $\xc\nin\mult{l_-,l_+}$ for any
  $l_- \le 0 \le l_+$.
\end{lem}
\begin{proof}
  By Lemma~\ref{l_singularityStructure} we gather that the only
  connected component of $\csp\setminus\sing1$ whose closure meets $\xc$
  is $\backReco$.  This proves our statement for $l_- = 0,\ l_+ = 1$.
  Now consider a connected component $\cc$ of $\csp\setminus\sing{0,2}$;
  by definition there exist
  $\indCell,\indCell'\in\{\recoPrivate,0,1,\cdots\}$ so that
  $\cc = \cell^+_\indCell\cap\cm\inv\cell^+_{\indCell'}$.  If
  $\cl\cc\ni\xc$, then by the above discussion
  $\indCell = \recoPrivate$, which by Remark~\ref{r_thereCanBeOnlyOne}
  implies that $\indCell'\ne\recoPrivate$.  But then we would have
  $\cl\cm\inv\cell^+_{\indCell'}\ni\xc$, which by
  Lemma~\ref{l_continuityAtxc} implies that
  $\cl\cell^+_{\indCell'}\ni\xc$, contradicting
  Lemma~\ref{l_singularityStructure}.  The statement for general $l_-$
  and $l_+$ then follows by applying Lemma~\ref{l_continuityAtxc}.
\end{proof}

\section{Accelerated Poincar\'e map.}\label{sec:accel-poinc-map}
The analysis of Section~\ref{s_hyperbolicity} shows that expansion of the
collision map $\cm$ is small for large energies. That is, the hyperbolicity
of $\cm$ is rather weak in this region.  It is thus convenient to
consider an induced map, obtained by skipping over collisions
that happen in the same fundamental domain for $\ell$.  In this section
we discuss the resulting accelerated map $\cmp$.  In particular, we will
recall the results of~\cite{fum}, where the large energy regime for
piecewise smooth Fermi--Ulam Models was studied in detail.  At the same time, we
will also present some new technical estimates which are needed for the
proof of our Main Theorem.

\subsection{Number of collisions per period.}
Recall the definition of positive and negative $\indCell$-cells given in
the previous section (see \eqref{e_defCell}). Define (see
Figure~\ref{f_inducingRegion}):
\begin{align}\label{e_definitionInducedSpace}
  \cspi = \clo\left(\csp\setminus \clo\cell^-_0\right).
\end{align}%
\begin{figure}[!h]
  \def\svgwidth{5cm} 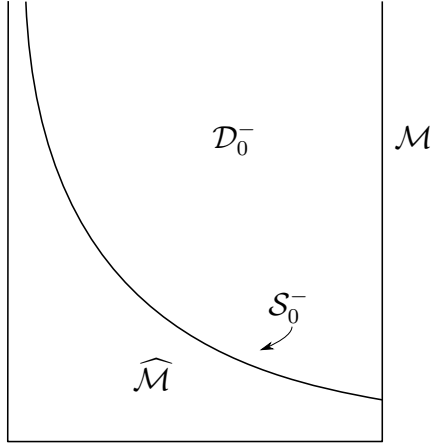
  \caption{The inducing set $\cspi$; note that the geometry can be
    slightly different depending on the properties of $\ell$. In
    fact, it is possible for $\singBackward 0$ to terminate at $\{w = 0\}$
    rather than at $\{\ct = 1\}$.}
  \label{f_inducingRegion}
\end{figure}%
\begin{rmk}\label{r_boundary cspi}
  Observe that $\partial\cspi$ is the union of vertical curves,
  horizontal curves and the unstable curve $\singBackward0$.  In particular,
  each curve in $\partial\cspi$ is compatible with the cone
  field $\conen$.
\end{rmk}

Let $\cE_0 = \intr\csp$ and, for any $n\in\naturals$, define
\begin{align*}
  \cE_{n} &= \{x\in  \csp\setminus\sing{n-1}\st\cm^{k}x\in\cell_0^+ \text{ for any }0\le k < n\}.
\end{align*}
Observe that, by construction, $\cE_n\supset\cE_{n+1}$ and
$\cE_n\supset\cm\cE_{n+1}$; since $\cell_0^+\cap\sing1 = \emptyset$,
we conclude by induction that $\cE_n\cap\sing{n} = \emptyset$.

For any $n > 0$, define $\cEs_n = \cE_{n-1}\setminus\cE_{n}$.
Observe that, if $x\in\cEs_1\setminus\sing{1}$, then $\cm$ is
well defined and smooth at $x$, and moreover $\cm x\in\cspi$; more
generally, for any $k\ge 1$, if $x\in\cEs_{k}\setminus\sing{k}$, then
the map $\cm^k$ is well defined and smooth at $x$, and moreover
$\cm^kx\in\cspi$.  For any $x\in\intr\csp$,
define:
\begin{align*}
  \Np(x) = \sum_{k\ge0} 1_{\cE_k}(x) = \max\{n\ge0\st \cE_n\ni x\}.
\end{align*}
Observe that, if $x\in\cEs_n$, our construction implies that
$\Np(x) = n$.  Finally, let
\begin{align*}
  \singt+ = \sing0\cup \bigcup_{k\ge 0} (\sing{k+1}\cap \cE_k).
\end{align*}
Observe that, for any $k$ we have
$\cEs_{k}\cap\singt+ = \cEs_{k}\cap\sing{k}$ and
$\partial\cEs_{k}\subset\singt+$.  In particular, for any $k > 0$, the
function $x\mapsto\min\{k,\Np(x)\}$ is constant on each connected
component of $\csp\setminus\sing{k}$.  Moreover, by construction,
$\singt+$ is a countable union of $C^1$-smooth stable curves with
\begin{align*}
  \singForward{}\subset\singt+\subset\sing{+\infty}.
\end{align*}
By the above considerations, we conclude that if
$x\in\csp\setminus\singt+$ and  $\Np(x) < \infty$, then
$\cm^{\Np(x)}$ is well-defined and smooth at $x$ and
$\cm^{\Np(x)}x\in\cspi$.  We now proceed to show that $\Np$ is finite
for any $x\in\intr\csp$.
\begin{lem} \label{LmFRet} %
  The sets $(\cEs_n)_{n > 0}$ form a partition$\pmod0$ of $\csp.$
  Moreover
  for any $x = (r,w)\in\intr\csp$:
  \begin{align}\label{e_weakBoundOnNp}
    1\le\Np(x)\le \Const w+ N_\#;
  \end{align}
\end{lem}
\begin{proof}
  We claim that for sufficiently large $n$:
  \begin{align}\label{e_linearBoundOncE}
    \cE_n \subset \{w\ge \Const n-\slopeu\}.
  \end{align}

  Observe that~\eqref{e_linearBoundOncE} implies that
  \begin{align*}
    \bigcap_{k\ge0}\cE_k = \emptyset;
  \end{align*}
  which in particular implies that the sequence $(\cEs_n)_{n > 0}$
  forms a partition$\pmod 0$ of $\csp$. The
  estimate~\eqref{e_weakBoundOnNp} also immediately follows
  from~\eqref{e_linearBoundOncE}.

  We proceed with the proof of our claim. Assume $x\in\cE_n$ and let
  $x_k =(r_k,w_k) = \cm^k x$. By construction, we have for any
  $0\le k < n$ that $x_k\in\cell^+_0$, \ie
  $\ct_{k}+\fft(x_{k}) \in(0,1)$.  By induction, this implies
  $\ct_n = \ct_{0}+\sum_{k=0}^{n-1}\fft(x_k) < 1$.  In particular
  \begin{align*}
    \sum_{k=0}^{n-1}\fft(x_k) < 1.
  \end{align*}
  On the other hand, since $\cell^+_0\cap\cl\backReco = \emptyset$, if
  $(r,w)\in\cell^+_0$, we can use the lower bound
  in~\eqref{e_boundFreeFlight}, which gives
  \begin{equation}
    \label{FTRecall}
    \fft(\ct,w)\geq2\elll/(w-h(\ct)).
  \end{equation}
  Let $v_k = w_k-h(\ct_{k})$ be the absolute velocity after the $k$-th
  collision; notice that since in particular $x_k\nin\backReco$ for
  $0\le k < n$ we have $v_k > 0$; moreover, trivially
  $v_k \le v_0+2k\slopeu$. We conclude that
  \begin{align*}
    1 > \sum_{k = 0}^{n-1}\fft(x_k)\ge \frac\elll\slopeu\sum_{k =
    0}^{n-1}\left[\frac{v_0}{2\slopeu}+k\right]\inv\ge\frac\elll\slopeu\log\left[1+\frac{2\slopeu n}{v_0}\right].
  \end{align*}
  Hence,
  \begin{equation} \label{ColBLV}
    v_0 > \Const n,
  \end{equation}
  which immediately implies~\eqref{e_linearBoundOncE}, since
  $v_0 < w+\slopeu$.
\end{proof}
Define $\singp+ = (\singt+\cap\cspi) \cup \partial\cspi$.
Lemma~\ref{LmFRet} implies that the map
$\cmp:\cspi\setminus\singp+\to\cspi$ given by
\begin{align*}
  \cmp(x) = \cm^{\Np(x)}(x),
\end{align*}
is well defined and smooth.  A completely analogous construction leads
to the definition of a set $\singp{-}$ so that the inverse induced map
$\cmp\inv$ is defined for $x\in\cspi\setminus\singp-$.  In fact we
have that $\cmp$ is a diffeomorphism
$\cmp:\cspi\setminus\singp+\to\cspi\setminus\singp-$.  We can
  also define $\Np_{-}:\cspi\setminus\singp-\to\bZ_{ < 0}$ so that
  $\cmp\inv(x) = \cm^{\Np_{-}(x)}(x)$.  Observe that
  $\Np_{-}(x) =-\Np(\cmp\inv(x))$.

We now proceed to define the singularity set for the map $\cmp^k$ for
any $k\in\integers$.  This is completely analogous to the construction
carried over in Subsection~\ref{ss_singularitiesLocalStructure}; let
$\singp0 = \partial\cspi$, $\singp1 = \singp+$ (\resp
$\singp{-1} = \singp-$) and for any $n > 0$ let
\begin{align*}
  \singp{n+1} &=\singp{n}\cup\cmp\inv(\singp{n}\setminus\singp-)
  &\singp{-n-1} &=\singp{-n}\cup\cmp(\singp{-n}\setminus\singp+).
\end{align*}
Observe that $\cmp^k$ is well defined and smooth at $x$ if and only if
$x\in\cspi\setminus\singp{k}$.  Let furthermore
$\DS \singp{+\infty} = \bigcup_{n\ge0}\singp{n}$ and
$\DS \singp{-\infty} = \bigcup_{n\le0}\singp{n}$.

For any $n \ge 0$, let us define
  $\Np_{n}:\cspi\setminus\singp{n}\to\bN$ by induction as follows.  We
  let $\Np_{0}(x) = 0$ and, for $k \ge 1$, we let
\begin{align*}
  \Np_{k}(x) = \Np_{k-1}(x)+\Np(\cmp^{k -1}x).
\end{align*}
Observe that by construction we have
$\cmp^{n}(x) = \cm^{\Np_{n}(x)}(x)$.  Then define $\singt{n}$ as
follows: $x\in\singt{n}$ if either $x\in\singt{+}$ or
$\cm^{\Np(x)}\in\singp{n-1}$.  Then we can extend the definition of
$\Np_{n}$ to $\csp\setminus\singt{n}$ as follows: if $n = 1$ we let
$\Np_{1}(x) = \Np(x)$; otherwise
$\cm^{\Np(x)}(x)\in\cspi\setminus\singp{n-1}$ and we define
$\Np_{n}(x) = \Np(x)+\Np_{n-1}(\cm^{\Np(x)}x)$.  A similar
construction leads to the definition of $\Np_{-n}$ for $n > 0$.

\begin{rmk}\label{rmk:definition-Npk}
  It follows from our construction that if $x = (r,w)$ is so that
  $\Np_{k}(x)$ is defined, then, denoting once again $x_{j} = \cm^{j}x$:
  \begin{align*}
    \Np_{k}(x) = \min \{n \st \ct+\sum_{j = 0}^{n-1}\tau(x_{j})\ge k\}.
  \end{align*}
\end{rmk}

Let $W$ be an unstable curve, and $n > 0$; let $W'$ be a connected
component of $\cm^{n}W$; then we can define
  \begin{align}\label{eq:definition-hatn}
    \hat n(W') &= \max \{k \st \Np_{k}(x)\le n \text{ for all }
                 x\in \cm^{-n}W'\}.
  \end{align}
We conclude this subsection with the definition of the fundamental
domains
\begin{align}\label{e_fundamentalDomains}
  \Dom_n &= \intr\cspi\cap\cEs_n.
\end{align}
Notice that our previous discussion shows that
\begin{subequations}\label{e_domainSingularities}
  \begin{align}
    \Dom_{n}\cap\sing{n-1} &= \emptyset\\
    \Dom_n\cap\singp+ &= \Dom_n\cap\sing{n}.
  \end{align}
\end{subequations}

\subsection{Dynamics for large energies.} \label{SSCMP}%
In~\cite{fum} we have proved several useful properties that the map
$\cmp$ satisfies for large values of $w$.  We collect them in the
proposition below. Recall the notation $$(\ct_k,w_k) = \cm^{k}(\ct,w).$$
\begin{prp}[Properties of $\cmp$ for large energies]\label{p_largeEnergiesFUM}
  There exists $\largew > 0$ so that, if $(r,w)\in\cspi$, $w\ge\largew$:
  \begin{enumerate}
  \item \label{i_boundOnExcursion} there exists $C_* > 1$ so that for
    any $0 \le k \le \Np(\ct,w)$
    \begin{align}\label{e_boundOnExcursion}
      w_k,w_k-h(\ct_k)\in(C_*\inv
      w, C_*w);
    \end{align} Accordingly, we have\footnote { In fact, the following
      stronger statement holds: the limit of $\frac{\Np(\ct, w)}{w}$ exists
      when $w\to \infty$ and $(\ct, w)\in \cspi$.  However, the weaker
      estimate~\eqref{Per-W} is sufficient for our current purposes.}
    \begin{align}
      \label{Per-W}
      C_{*}\inv{w} &\le \Np(\ct, w)\leq C_* w.
    \end{align}

  \item \label{i_boundOnEndpoint} there exists $\hat C$ so that
    $|w_{\Np(\ct,w)}-w|\le \hat C$.
  \end{enumerate}
\end{prp}
\begin{cor}\label{c_weakBoundOnEndpoint}
  For any $(r,w)\in\cspi\setminus\singp+$, let $(\hat r,\hat w) =
  \cmp(r,w)$; then
  \begin{align*}
    |\hat w-w|\le\const.
  \end{align*}
\end{cor}
\begin{proof}
  The proof immediately follows combining
  Proposition~\ref{p_largeEnergiesFUM}\ref{i_boundOnEndpoint}
  (for large $w$) and~\eqref{e_weakBoundOnNp} (for small $w$).
\end{proof}
In fact, in~\cite{fum} we constructed a normal form for $\cmp$ for
high energies, which we now proceed to describe.  Consider the strip
$\bM = [0,1]\times \bR\ni(\tau,I)$, and for $\Delta\in\bR$ define the
piecewise affine map $\nf_\Delta: \bM\to\bM$ given by the formula
\begin{equation}
  \label{SawTooth}
  \nf_\Delta(\tau,I) = (\bar \tau,\bar I),\text{ where }
  \left\{\begin{array}{rl}
           \bar \tau &= \tau-I\mod 1,\\
           \bar I &= I+\Delta(\bar\tau-1/2).
         \end{array}\right.
     \end{equation}
     The curves $\{\tau = I\mod 1\}$ partition $\bM$ in a countable number of
     fundamental domains that we denote with $(\hDom_n)_{n\in\bZ}$, where the
     index $n$ is so that $\hDom_n\ni(1/2,n)$.  Observe that $\nf_\Delta$ is
     continuous in each fundamental domain. In particular, for
     $n\in\integers$ let $\tsl n:\bM\to\bM$ be the translation map
     \begin{align}\label{e_translationMapDefinition}
       \tsl n &: (\tau,I)\mapsto(\tau,I+n);
     \end{align}
     then $\hDom_n = \tsl n\hDom_0$ and if $x\in\hDom_n$, we have
     $\nf_\Delta = \tsl{n}\circ\nfr_\Delta\circ\tsl{-n}$, where
     $\nfr_\Delta:\reals^2\to\reals^2$ is the affine map given by
     \begin{align*}
       \nfr_\Delta(\tau,I) = (\tilde\tau,\tilde I),\text{ where }
       \left\{\begin{array}{rl}
                \tilde\tau &= \tau - I,\\
                \tilde I &= I +\Delta(\tilde\tau-1/2).
              \end{array}\right.
     \end{align*}
     The relevance of the map $\nf_\Delta$ comes from
     Theorem~\ref{t_normalForm} below.  The theorem is essentially a more
     detailed statement of~\cite[Theorem 1]{fum}.  The reader will have no
     difficulty to check that~\cite[Section II]{fum} indeed provides all that
     is needed to prove Theorem~\ref{t_normalForm}.

     Below the symbol $\cO_k(I^{-1})$ denotes a function whose partial
     derivatives up to order $k$ are $\cO(I^{-1})$.
     \begin{thm}
       \label{t_normalForm}
       There exist $\largew > 0$ and coordinates $(\tau,I)$ on the
       set $\cspi\cap\{w\ge\largew\}$ so that
       \begin{enumerate}
       \item $\const\inv w < I < \const w$; moreover, there exists
         $C > 0$ so that if $(r,w)\in D_n$, and $(r',w')\in D_{n'}$
         and $w' -w > C$, then necessarily $n' > n$.
       \item the singularity lines $\{\ct = 0\}$ and
         $\cm \{\ct = 0\}$ are given in $(\tau, I)$ coordinates by
         $\{\tau = 0\}$ and $\{\tau = 1+\cO_5\left(I^{-1}\right)\}$
         respectively;
       \item if $x\in D_n$ then $\cmp$ in $(\tau,I)$-coordinates is
         a $\cO_5(I\inv)$-\hskip0pt{}perturbation of\/
         $\tsl n\circ\nfr_\Delta\circ\tsl {-n}$ where $\Delta$ is
         given by \eqref{DefDelta}.
       \end{enumerate}
       The coordinates $(\tau, I)$ will be called \emph{adiabatic
         coordinates}.
     \end{thm}

     In particular, the above theorem implies that if $n$ is
     sufficiently large, $\tsl{-n}\Dom_n$ is contained in a
     $\const n\inv$-neighborhood of $\hDom_0$.  We will often drop the
     subscript $\Delta$ from $\nfr$ when this will not cause
     confusion.

     For future reference we include the formulas relating the
     adiabatic coordinates $(\tau, I)$ to the original coordinates
     $(\ct, w).$ Namely we have
     \begin{subequations}
       \label{ChCoordIT}
       \begin{equation} I=w \ell(\ct) +\mathfrak{a}(\ct)+\cO_5(w^{-1}),
       \end{equation}
       \begin{equation} \tau= \theta I+\cO_5(w^{-1}),
       \end{equation}
       \begin{equation} \theta=\int_0^t \ell^{-2}(s)
         ds+\frac{\mathfrak{b}(\ct)}{w}+\cO_5(w^{-2})
       \end{equation}
     \end{subequations} where $\mathfrak{a}$ and $\mathfrak{b}$ are smooth
     functions whose precise value will not be important for us.

     The next result, proven in~\cite{fum}, provides the first major
     step toward the proof of the ergodicity of dispersing Fermi--Ulam
     Models.
     \begin{thm}
       (\cite[Theorem 4]{fum})
       \label{ThRec}
       Dispersing Fermi--Ulam Models are recurrent.
     \end{thm}

     \subsection{Bounds for $p$-slopes.}

     We record in this section several useful estimates.
     \begin{lem}
       \label{LmPSlopeC0}
       There are constants $c_1, c_2 > 0$ such that for any $w_*$
       sufficiently large, any $x = (\ct,w)\in \csp$, if
       $\prpslope\ge 0$ (and in particular for any unstable vector):
       \begin{itemize}
       \item[(a)] If $w\geq w_*$ then $(\prpslope)'\geq \cK/w$.
       \item[(b)] If $w\leq w^*$, then
         \begin{align*}
           (\prpslope)'\geq \frac{c_1}{1+\tau}.
         \end{align*}
         Furthermore, if $x\nin\backReco$, we also have the upper bound
         \begin{align*}
           (\prpslope)'\leq \frac{c_2}{1+\tau}.
         \end{align*}
       \end{itemize}
     \end{lem}
     \begin{proof}
       Assume $w_* > 2\cK$ and so large that $w\geq w_*$ implies that
       $\tau\leq 1$.  In this case,~\eqref{e_slopeEvolution} implies:
       \begin{align*}
         (\prpslope)' &= ((\prpslope+\cpar)\inv+\tau)\inv\ge
                        ((\prpslope+\cpar)\inv+1)\inv\\
                      &\ge (\cpar\inv+1)\inv =  \left(w/2\curv+1\right)\inv\ge
                        \cK w\inv.
       \end{align*}
       This proves item (a).  Next suppose that $w\leq w^*$ somewhere on $W$.
       Then, unless $x\in\backReco$, there is a constant $\delta=\delta(w_*)$
       such that $\tau\geq\delta$.  In order to prove (b), rewrite
       \begin{equation}
         \label{ElFr}
         (\prpslope)'=\frac{1}{\tau}-\frac{1}{\tau(1+\tau(\prpslope+\cpar))}.
       \end{equation}
       Hence
       \begin{align*}
         \frac{1}{\tau}\left(1-\frac{1}{1+\delta \cK}\right)\leq (\prpslope)'\leq \frac{1}{\tau},
       \end{align*}
       which gives both the upper and lower bounds.  If, on the other hand
       $x\in\backReco$, then $\tau\le 1$ and by Lemma~\ref{l_forwardReco} we
       have $w \le \slopeb$; proceeding as in (a), we obtain the lower bound
       provided that $c_1\le (\slopeb/2\cK+1)\inv$.
     \end{proof}

     Recall that
     $\prpslope_k$ denotes the value of $\prpslope$ of the $k$-th iterate
     of the element under consideration.

     \begin{lem}
       \label{LmPSlopeHighEnergy} There are constants $c_3, c_4, \breps$ such
       that the following estimates hold for $w\geq w^*$.

       \begin{enumerate}
       \item i. If $\prpslope\geq \breps$ then $(\prpslope)'\geq \breps$

         \noindent ii. if $\prpslope\leq \breps$ then
         $(\prpslope)'\geq \prpslope+\frac{c_3}{w}. $

       \item i. If $\ifrac{1}{\prpslope} \geq \breps$ then
         $\ifrac{1}{(\prpslope)'}\geq \breps$

         \noindent ii. if $\ifrac{1}{\prpslope}\leq \breps$ then
         $\ifrac{1}{(\prpslope)'}\geq \frac{1}{\prpslope}+\frac{c_3}{w}$.

         \item\label{i_highEnergySqueeze} i. If
         $\breps\leq \prpslope_0\leq \frac{1}{\breps}$ then for any
         $n\leq w,$ $\breps\leq \prpslope_n\leq \frac{1}{\breps}$.

         \noindent ii. If $\prpslope_0\le\breps$ then for $n\leq {w}$, we have
         $\prpslope_n\geq \min(\frac{n c_4}{w}, \breps)$.

         \noindent iii. If $\prpslope_0\geq \frac{1 }{\breps}$ then for
         $n\leq {w}$, we have
         $\prpslope_n\leq \max(\frac{w}{n c_4}, 1/\breps)$.

       \end{enumerate}
     \end{lem}
     \begin{proof}
       In this proof we drop the superscript “$-$” from $\cB$ for
       ease of notation.\\
       \renewcommand{\prpslope}{\cB} 
       (a) We have
       \begin{align*}
         \prpslope'-\prpslope &=
                                \frac{\frac{{2}\kappa}{w}\left(1-\tau \prpslope\right)-\tau
                                \prpslope^2}{1+\tau\left(\prpslope+\frac{2\kappa}{w}\right)}
       \end{align*}
       so (a)ii follows from the fact that
       $\frac{c^{-1}}{w}\le\tau\le \frac{c}{w}$, which in turn follows
       from~\eqref{e_boundFreeFlight}.  Since the function
       $B\mapsto \frac{\cR+B}{1+\tau(\cR+B)}$ is increasing (see
       \eqref{ElFr}) $\prpslope\geq\breps$ implies
       $\prpslope'\geq \frac{\cR+\breps}{1+\tau(\cR+\breps)}\geq\breps $
       where the last inequality relies on the already proven part (a)ii.
       This proves (a)i.

       (b) Let $\beta=1/\prpslope$.  Then
       $\beta'=\tau+\frac{\beta}{1+2\frac{\beta \kappa}{w}}$ whence
       \begin{align*}
         \beta'-\beta &= 
         \tau-\frac{2{\beta^2 \kappa}}{w+2{\beta \kappa}}.
       \end{align*}

       Thus (b)ii follows from the fact that $\tau\ge \frac{c}{w}$.  Since
       the function $\beta\mapsto\tau+\frac{\beta}{1+2\frac{\beta \kappa}{w}}$ is
       increasing, $\beta\geq\breps$ implies
       $\beta'\geq\tau+\frac{\breps}{1+2\frac{\beta \kappa}{w}}\geq \breps$
       where the last step relies on the already proven part (b)ii.  This proves
       (b)i.

       (c) Item i immediately follows from (a)i and (b)i.  By part (a)i we
       can conclude that if $\prpslope_{k}\ge\breps$ for some $0 < k\le n$,
       then necessarily $\prpslope_{n}\ge\breps$.  We can therefore assume
       that $\prpslope_{k} < \breps$ for all $0 < k \le n$.  In this case
       part (a)ii implies that
       $\prpslope_{k+1}\ge\prpslope_{k}+\ifrac{c_3}{w_{k}}$.  Combining
       this with \eqref{e_boundOnExcursion} we
       obtain $\prpslope_{n}\ge\prpslope+\ifrac{n c_3}{w}$, proving
       (c)ii.  The upper bound follows by analogous
       considerations involving $\prpslope\inv$ and part (b).
     \end{proof}

     It is convenient to consider smaller invariant cones, which are
     obtained by iterating the dynamics on $\coneu$ and $\cones$.
     First the cones will be defined on $\cspi$, then they will be
     extended to $\csp$ using the dynamics.  Observe that since such
     cones are defined dynamically and the dynamics is only defined
     almost everywhere, we will only be able to define the cones
     almost everywhere.
     \begin{mydef}\label{def:mature-cones}
       Let $x\in\cspi\setminus\singp+$; define
       \begin{align*}
         \mcones(x)&=\cmp_*\inv|_{\cmp x}\cones;
       \end{align*}
       if $x\in\csp\setminus\singt+$, then $\cm^{\Np(x)}x\in\cspi$,
       and we can define
       \begin{align*}
         \mcones(x) &= \cm^{-\Np(x)}_{*}|_{\cm^{\Np(x)}x}\mcones(\cm^{\Np(x)}x).
       \end{align*}
       Observe that $\mcones(x)$ is defined almost everywhere on
       $\csp$; with a similar procedure we can define $\mconeu(x)$ for
       a.e.\ $x\in\csp$.

       An unstable (\resp stable) curve will be called \emph{mature}
       if it is tangent to $\mconeu$ (\resp $\mcones$). In particular,
       $W\subset\cspi\setminus\singp-$ is a mature unstable curve if
       $\cmp^{-1} W$ is unstable; likewise
       $V\subset\cspi\setminus\singp+$ is a mature stable curve if
       $\cmp V$ is a stable curve.
     \end{mydef}

     Combining Lemma~\ref{LmPSlopeHighEnergy} with
     Theorem~\ref{t_normalForm} and using Lemma~\ref{LmPSlopeC0} we obtain
     the following result.
     \begin{cor} \label{CrHESqueeze} %
       There are constants $\brw, \brb $ such that the following
       holds.  Let $W$ be a mature unstable curve, then
       \begin{enumerate}
         \item\label{i_lower-bound-prpslope} for all $n\ge0$ such that
         $w_n\geq \brw$, or if $x_{n}\in\reco$, we have
         $\prpslope_n\geq\brb$.  \item\label{i_upper-bound-prpslope}
         for all $n\ge0$ such that $x_{n}\nin\reco$ we have
         $\prpslope_{n}\leq\brb\inv$.
       \end{enumerate}
     \end{cor}

     Note that combining Corollary \ref{CrHESqueeze} with
     \eqref{e_slopeCollision} yields that there is a constant $\brC > 1$
     such that for sufficiently large $w$:
     \begin{subequations}\label{MConeNarrow}
       \begin{align}  \mconeu&\subset \left\{-\brC
         w <\frac{\delta w}{\delta \ct}<-\cK-\brC\inv{w}\right\}  \\
         \mcones&\subset \left\{\cK+\brC\inv w <\frac{\delta w}{\delta
                  \ct}<\brC{w}\right\}.
       \end{align}
     \end{subequations}

     In the recollision region $\reco$ Corollary \ref{CrHESqueeze} 
     does not provide an upper bound on $\prpslope$. In fact, in this region $\prpslope$
     may in fact grow arbitrarily large.  However, a simple inspection
     of~\eqref{e_slopeEvolution} shows that for any $L > 0$
     sufficiently large there exists $\delta > 0$ so that if
     $(\prpslope)' > L$ then $w < \delta$ and $\tau < \delta$.  We
     gather that if $(\prpslope)'$ is large, then $x$ lies in a
     neighborhood of the point $(1,0)$.  The analysis in
     Lemma~\ref{l_forwardReco} allows then to conclude that $x'$ lies
     in a neighborhood of $(0,\slopeb)$.  We summarize the above
     observation for future use in the following lemma.
       \begin{lem}\label{lem:vertical-unstable}
         There exists $B > 0$ so that if $W$ is a mature unstable curve
         passing through $x = (\ct,w)$ with pre-collisional p-slope
         $\prpslope$, then either $\prpslope < B$ or $w > B\inv$.
       \end{lem}

     \subsection{The \texorpdfstring{$\admNamePm$}{α±}-metrics}%
     We now proceed to define a pair of convenient metrics on $\csp$, which
     we denote with $|\cdot|\adm$ and $|\cdot|\adms$ and call the
     \emph{$\admName$-metric} and the \emph{$\admNameM$-metric},
     respectively.  Let $\admParam0,\admParam1 > 0$ be small constants
     which will be specified later (see~\eqref{ChooseAlpha}
     and~\eqref{e_chooseAlpha0}). For $x = (\ct,w)$, we define the
     functions
     \begin{align*}
       \admFuncPm(x) = \expo{\admParam0\If_\recoMp(x)}(1+\admParam1\cdot w),
     \end{align*}
     where $\If_\reco$ (\resp $\If_{\backReco}$) is the indicator function
     of $\reco$ (\resp $\backReco$).  For $\deh x\in\tang_x\csp$ we set
     (recall that $\curv(\ct) = \ell''(\ct)$)
     \begin{align*}
       |\deh x|\admPm=\admFuncPm(x)(\curv(\ct) |d\ct|+|dw|).
     \end{align*}
     Note that since $w=\frac{dz}{d\ct}$ we obtain the following relations
     with the Euclidean metric $|d x|^2_{\eum}=d\ct^2+d w^2$ and the
     $p$-metric $|\deh x|\pme$ defined at the beginning of
     Section~\ref{s_Jacobian}.
     \begin{subequations}\label{e_defAdm}
       \begin{align}
         |\deh x|\admPm %
         &= \admFuncPm(x)|\deh x|\pme\frac{\curv(\ct)+|\cslope|}{w} =\label{e_defAdmP} \\%
         &= \admFuncPm(x)|\deh x|\eum\frac{\curv({\ct})+|\cslope|}{\sqrt{1 +\cslope^2}}.\label{e_defAdmE}
       \end{align}
     \end{subequations}

     \begin{lem}
       \label{LmAlp-Euc}
       Let $|\cdot|_{\text{E}(\tau,I)}$ be the Euclidean metric in
       $(\tau,I)$-coordinates on $\cspi$.
       \begin{enumerate}
       \item There exists $c > 0$ so that for any vector
         $dx\in\tang_{x}\cspi$ we have
         \begin{align}\label{e_boundAlphaTau}
           \quad |dx|\admPm \ge
           c|dx|_{\text{E}(\tau,I)}.
         \end{align}

       \item For each $A > 0$ there is a constant $C > 0$ such that if
                  \begin{equation} \label{AVertCone}
             A^{-1}\leq \frac1w \frac{|\delta w |}{|\delta \ct|} \leq A, \quad \delta \ct\ \delta w<0
           \end{equation}
         then
         \begin{equation}
           \label{e_boundAlphaTauCone}
           C\inv
           w|dx|_{\text{E}(\tau,I)}\leq |dx|\admPm\leq C
           w|dx|_{\text{E}(\tau,I)}.
         \end{equation}
       \item  There is a constant $A$ such that each vector in
         $\mconeu$ satisfies \eqref{AVertCone}.  Consequently
         \eqref{e_boundAlphaTauCone} holds on $\mconeu$.
       \end{enumerate}
     \end{lem}

     \begin{proof}
       Without loss of generality, we assume that
       $\max(|\delta w|, |\delta \ct|)=1$.  Using~\eqref{ChCoordIT} we get
       \begin{equation}\label{DeltaI}
         \delta I=\ell \delta w+\left(w \dot\ell +\dot{\mathfrak{a}}\right)\delta \ct+O(w\inv),
       \end{equation}
       \begin{equation}\label{DeltaTau}
         \delta \tau=\theta \delta I+I\delta \theta+O(w\inv)=\theta \delta I+
         \frac{I\delta \ct}{\ell^2}+O(w\inv).
       \end{equation}
       Hence, both terms are $o(w)$, while $|dx|\admPm$ is
       of order $w$; part (a) follows.

       Next, under the assumptions of part (b) we get that
       $|\delta \ct|\leq A/w.$ It follows that both leading terms in
       \eqref{DeltaI} are of order $1$ and, moreover, they have the
       same sign, since $\delta w$ and $\delta\ct$ have different
       signs while $\dot\ell(\ct)$ is negative for small $\ct$ (note
       that since $\tau\in [0,1]$ it follows that
       $\theta=\cO(1/w)$). The foregoing remark also shows that the
       first term in \eqref{DeltaTau} is $\cO(1/w)$ while the second
       term is $\cO(1).$ Part (b) follows.
       It remains to note that \eqref{AVertCone} holds on $\mconeu$ due to
       Corollary \ref{CrHESqueeze}.
     \end{proof}
     The estimate \eqref{e_boundAlphaTauCone} has the following useful
     consequence.  Let
       \begin{equation}
       \label{DefLambdaDelta}
       \Lambda_\Delta=\frac{\cT+\sqrt{\cT^2-4}}{2}, \text{ where }
       \cT=2-\Delta
     \end{equation}
     be the leading eigenvalue of $d\nf $ defined by~\eqref{SawTooth}.
     \begin{cor}
       \label{CrExpAtInf} For each $n$ there are constants
       $\hat{C}, \bar{w}$ such that if $w_k\geq \bar{w}$ for
       $k=0,\cdots, n-1$ and $dx^u\in \mconeu$ then
       \begin{equation}\label{EqExpAtInf} |\cm^{n}_*\deh x^{\unstable} |\adm
         \geq \hat C \Lambda_\Delta^n |\deh x^{\unstable}|\adm.
       \end{equation}
     \end{cor}
     \begin{proof}
       The discussion following \eqref{DeltaI}, \eqref{DeltaTau}
       shows\footnote{ Recall that the leading term in
         \eqref{DeltaTau} is the second one and that $\delta w$ and
         $\delta\ct$ have different signs.} that
       \begin{align*} \mconeu\subset \coneIT:=\{(\delta I, \delta \tau):
         \delta I \delta \tau <0\}.
       \end{align*} It is also straightforward to check that there is a
       constant $\bar C$ such that for $v\in\coneIT$ we have
       \begin{align*} |(F_\Delta^n)_* v|_{\text{E}(\tau,I)}\geq \bar C
         \Lambda_\Delta^n |v|_{\text{E}(\tau,I)}.
       \end{align*}

       Now Theorem~\ref{t_normalForm} gives that for any $n$ and
       sufficiently large $\bar w$ (depending on $n$)
       \begin{align*} |\cm^n_* v|_{\text{E}(\tau,I)}\geq \frac{\bar C}{2}
         \Lambda_\Delta^n |v|_{\text{E}(\tau,I)}
       \end{align*} and~\eqref{EqExpAtInf} follows
       from~\eqref{e_boundAlphaTauCone}.
     \end{proof}

     The $\admNamePm$ metrics are Finsler metrics and they have the
     advantage of being \emph{Lyapunov metrics}, in the
     sense that they are strictly monotone for the (forward or
     backward , respectively) iterations of $\cmp$, as will be proven
     in Proposition~\ref{p_propertiesAdm} below.

   For $x = (\ct,w)\in\csp$, denote $x' = (\ct',w') = \cm x$ and for
   $dx\in\tang_x\csp$ we let $\deh x' = \cm_*\deh x\in\tang_{x'}\csp$.
   Likewise, for $x\in\cspi$, we denote
   $\hat x = (\hat\ct,\hat w) = \cmp(x)$ and for $dx\in\tang_x\cspi$
   we let $\deh \hat x = \cmp_*\deh x\in\tang_{\hat x}\cspi$.
   \begin{prp}\label{p_propertiesAdm}
     The $\admNamePm$-metrics satisfy the following properties:
     \begin{enumerate}
       \item\label{i_admeum} $|\cdot|\admPm$ is (uniformly)
       equivalent to $(1+\admParam1 w)|\cdot|\eum$. In particular
       $|\cdot|\adm$ and $|\cdot|\adms$ are equivalent to each
       other.  \item\label{i_expansionAdm} $\cm$ satisfies the
       following expansion estimate for any $\deh x\in\coneu_x$:
       \begin{subequations}
         \label{e_expansionAdm-two}
         \begin{align}\label{e_expansionAdm-alpha}
           \frac{|\deh x'|\admPm}{|\deh x|\admPm} 
           &\ge \frac{\alpha^{\pm}(x')}{\alpha^\pm(x)}\left(1+\fft\frac{2\cK}{w'}\right)\\
           &\ge e^{-\admParam0}\frac{1+\admParam1 w'}{1+\admParam1 w\phantom'}\left(1+\fft\frac{2\cK}{w'}\right);\label{e_expansionAdm}
         \end{align}
       \end{subequations}
       moreover if $w'$ is sufficiently small, for any $\deh x\in\coneu_x$:
       \begin{align}\label{e_expansionAdm-II}
         \frac{|\deh x'|\admPm}{|\deh x|\admPm} 
         &\ge \frac\const{w'}.
       \end{align}
       Additionally for any sufficiently large $w^{*} > 1$
         there exists $\Lambda^{*} > 1$ so that for any
         $x = (\ct,w)\in\csp\setminus\singt{+}$ with $w\ge w^{*}$,
         $\deh x^{\unstable}\in\coneu_x$ and $0\le n\le \Np(x)$:
       \begin{align}\label{e_boundedExpansion}
         |\cm^{n}_*\deh x^{\unstable}   |\adm& < \Lambda^{*}|\deh x^{\unstable}|\adm.
       \end{align}
     \item\label{i_uniformHyperbolicity} If $\admParam0$ and
     $\admParam1$ are sufficiently small, then the map $\cmp$ is
     uniformly hyperbolic with respect to the $\admNamePm$-metrics and
     the expansion is monotone in the following sense: there exists
     $\Lambda>1$ so that for any $x\in\cspi$,
     $\deh x^{\unstable}\in\coneu_x$ and any
     $\deh x^{\stable}\in\cones_x$:
       \begin{align}\label{e_uniformHyperbolicity}
         |\cmp_*\deh x^{\unstable}   |\adm&>\Lambda|\deh x^{\unstable}|\adm
         &|\cmp_*\inv \deh x^{\stable}|\adms&>\Lambda|\deh x^{\stable}|\adms.
       \end{align}
     \end{enumerate}
   \end{prp}
     \begin{proof}
       Item~\ref{i_admeum} immediately follows from~\eqref{e_defAdmE}.  In
       order to prove the remaining items it is convenient to introduce an
       auxiliary metric, which we denote with $|\cdot|\aum$ and is given by
       the expression:
       \begin{equation} \label{StarMetric}
         |\cdot|\aum = \admFuncPm(x)\inv|\cdot|\admPm =  \curv(\ct) |d\ct|+|dw|.
       \end{equation}
       Recall that by~\eqref{ExpPMetric} and~\eqref{e_slopeEvolution} we
       have
       \begin{align*}
         \frac{|\deh x'|\pme}{|\deh x|\pme} &= 1 + \fft\popslope,&
         (\prpslope)^{\prime}&=\frac{\popslope}{1+\fft \popslope}
       \end{align*}
       where $\fft = \fft(x),$ $\popslope = \popslope(\deh x)$, and
       $(\prpslope)'=\prpslope(dx')$.  Hence, if $\deh x\in\coneu_x$,
       then~\eqref{e_defAdmP} and~\eqref{PCSlopes} give

       \begin{align}\notag
         \frac{|\deh x'|\aum}{|\deh x|\aum} %
         &= (1 + \fft\popslope) \frac{w}{w'}\frac{\curv'-\cslope'}{\curv-\cslope}
           =  \frac{1 + \fft\popslope}{\popslope}
           \frac{2\curv'+(\prpslope)' w'}{w'} = \\ \label{e_expansionAuxMetricA}
         & = \left(1+\frac{2\curv'}{(\prpslope)^{\prime}w'}\right),
       \end{align}
       where for ease of notation we denoted $\curv = \curv(x)$ (\resp
       $\curv' = \curv(x')$) and $\cslope = \cslope(\deh x)$ (\resp
       $\cslope' = \cslope(\deh x')$).  Since $\pslope^{-\prime}\leq 1/\fft$
       we conclude:
       \begin{align}
         \frac{|\deh x'|\aum}{|\deh x|\aum}
         & \ge \label{e_expansionAuxMetricB}%
           1+\fft\frac{2\cK}{w'},
       \end{align}
       from which equations~\eqref{e_expansionAdm-two} immediately
       follow.  In order to prove~\eqref{e_expansionAdm-II}, notice
       that if $w'$ is sufficiently small, then
       Lemma~\ref{lem:vertical-unstable} implies that $(\prpslope)'$
       is bounded from above. Using~\eqref{e_expansionAuxMetricA} then
       immediately implies~\eqref{e_expansionAdm-II}.  
       
       It
         remains to show~\eqref{e_boundedExpansion}. Notice that by
         Proposition~\ref{p_largeEnergiesFUM}\ref{i_boundOnExcursion}
         and Corollary~\ref{CrHESqueeze}(a), we can choose $w^{*}$ so
       that $\prpslope_{n}$ is bounded from below for any
       $0\le n \le \Np(x)$. Using~\eqref{e_expansionAuxMetricA} we thus
       gather that, for some uniform $\Lambda_{1}^{*} > 1$:
       \begin{align*}
         \frac{|dx_{n}|_{*}}{|d x|_{*}} &=%
           \prod_{k = 0}^{n-1}\left(1+C w_{k}\inv \right)\le \Lambda_{1}^{*},
       \end{align*}
       where in the last step we used Lemma~\ref{LmFRet}.  Then
       once again using the definition of $|\cdot|\adm$, we
       obtain~\eqref{e_boundedExpansion} and we conclude the proof of
     item~\ref{i_expansionAdm}.  Observe moreover
     that~\eqref{e_expansionAuxMetricB} gives the trivial bound
       \begin{align*}
         |\deh \hat x|\aum\ge |\deh x'|\aum\ge|\deh x|\aum.
       \end{align*}
       We proceed now to the proof of item~\ref{i_uniformHyperbolicity}. We
       first prove the statement for unstable vectors.  Let
       \begin{align*}
         |dx|\auum = \expo{\admParam0\If_\reco(x)} |dx|\aum.
       \end{align*}
       We now claim that we can choose $\admParam0 > 0$ so that we have
       \begin{align}\label{e_auum}
         \frac{|d\hat x|\auum}{|dx|\auum} \ge \expo{\admParam0}.
       \end{align}
       If the above bound holds, we obtain item~\ref{i_uniformHyperbolicity}. In fact, observe that
       \begin{align*}
         \frac{|\deh \hat x|\adm}{|\deh x|\adm} =%
         \frac{1+\admParam1\hat w}{1+\admParam1 w}\frac{|\deh \hat x|\auum}{|\deh x|\auum}.
       \end{align*}
       Using Corollary~\ref{c_weakBoundOnEndpoint}, we can choose
       $\admParam1 > 0$ so small that
       \begin{align}
         \label{ChooseAlpha}
         \min_{(r,w)\in\cspi}\frac{1+\admParam1\hat w}{1+\admParam1 w} >  \expo{-\admParam0/2}.
       \end{align}
       \eqref{ChooseAlpha} together with~\eqref{e_auum} yields the first estimate
       of~\eqref{e_uniformHyperbolicity} with
       $\Lambda = \expo{\admParam0/2}$.  The corresponding
         estimate for stable vectors is obtained by applying the
         involution, and observing that the involution maps the
         $\admNameM$-metric for $\cm$ to the $\admName$-metric for
         $\cm\inv$.  This concludes the proof
         of~\ref{i_uniformHyperbolicity}.  
         
         It remains to
       prove~\eqref{e_auum}. First of all observe that, by definition
       \begin{align*}
         \frac{|d\hat x|\auum}{|dx|\auum} = \expo{\admParam0 (\If_\reco(\hat
         x) -\If_\reco(x))}\frac{|d\hat x|\aum}{|dx|\aum}.
       \end{align*}
       Notice moreover that if $x\in\backReco$ we have, by definition,
       $\cm x\in\reco\subset\cspi$ which yields $\hat x = x'$.  
       Since $\reco\cap\backReco = \emptyset$, we conclude
       that
       \begin{align*}
         \frac{|d\hat x|\auum}{|dx|\auum} = \expo{\admParam0}\frac{|dx'|\aum}{|dx|\aum}&\ge\expo{\admParam0}
         &\text{for any } x\in\backReco.
       \end{align*}
       On the other hand, if $x\not\in\backReco$ we have
       \begin{align*}
         \frac{|d\hat x|\auum}{|dx|\auum}\ge\expo{-\admParam0}\frac{|d\hat x|\aum}{|dx|\aum}.
       \end{align*}
       It thus suffices to show that we can choose $\admParam0$ so that
       \begin{align}\label{e_aum}
         \frac{|d\hat x|\aum}{|dx|\aum}&\ge\expo{2\admParam0}
         &\text{for any } x\nin\backReco.
       \end{align}
       In order to do so, we combine~\eqref{e_boundFreeFlight}
       and~\eqref{e_expansionAuxMetricB} to obtain
       \begin{align}\label{e_expansionExplicitBound}
         \frac{|\deh x'|\aum}{|\deh x|\aum}
         &\ge 1 + \frac{4\cK\elll}{w'(w'+h(\ct'))}
         &\text{for any } x\nin\backReco.
       \end{align}
       Let us fix $\largew > 0$ sufficiently large to be specified later and consider two cases.
       
       (1) If $w < \largew$,
       by~\eqref{e_expansionExplicitBound} we can find $\Lambda_0 > 1$ such that if $x\nin\backReco$,
       \begin{equation}
         \label{ExpLambda0}
         |\deh \hat x|\aum>\Lambda_0|\deh x|\aum.
       \end{equation}

       (2) Next suppose that $w\ge \largew$ large.  In this case the expansion of
       just one iterate of $\cm$ does not suffice and one needs to take into
       account several iterates.
       \ignore{ as we proceed to explain.  By
         Proposition~\ref{p_largeEnergiesFUM} and~\eqref{e_boundFreeFlight} we
         can choose $C > 1$ so that $C\inv w\inv < \fft_k < Cw\inv$ where
         $\fft_k = \fft(\ct_k,w_k)$ and thus $\Np(\ct,w) > C\inv w$.  For
         $\deh x\in\coneu_{x}$ let us denote by $\prpslope_k$ the
         precollisional $p$-slope of $\deh x_k = \cm^k_*dx$.
         By~\eqref{e_slopeEvolution} we gather that (choosing a larger $C$ if
         necessary)
         \begin{align*}
           [\prpslope_{k+1}]\inv \ge \left(\prpslope_{k}+\frac{2\cK^*}{w_k}\right)\inv+\fft_k \ge [\prpslope_k]\inv\left(1+\frac{C[\prpslope_k]\inv}w\right)\inv+\frac{C\inv}{w}.
         \end{align*}
         A direct computation shows that if $[\prpslope_k]\inv\le1/(2C)$, then
         $$[\prpslope_{k+1}]\inv > [\prpslope_k]\inv + C\inv/(2w). $$
         On the other hand, since the function
         $x\mapsto (x\inv+C/w)\inv+C\inv/w$ is increasing for positive $x$, we conclude that if
         $[\prpslope_k]\inv\ge1/(2C)$, then $[\prpslope_{k+1}]\inv\ge1/(2C)$.  Thus, for $0\le k\le \Np(\ct,w)$
         \begin{align*}
           [\prpslope_k]\inv > \frac1{2C}\min\{1, k/w\}.
         \end{align*}}
       Namely, \eqref{e_expansionAuxMetricA}  and Lemma~\ref{LmPSlopeHighEnergy}\ref{i_highEnergySqueeze} give
       \begin{align}
         \label{ExpLambda1}
         \frac{|\deh\hat x|\aum}{|dx|\aum} &> \frac{|dx_{C\inv w}|\aum}{|dx|\aum} > 1+\frac{C\inv}{w}\sum_{k = 0}^{C\inv w}[\prpslope_k]\inv
         > \Lambda_1
       \end{align}
       for some uniform $\Lambda_1 > 1$.

       Combining \eqref{ExpLambda0} and \eqref{ExpLambda1}
       we obtain~\eqref{e_aum}
       provided that
       \begin{align}\label{e_chooseAlpha0}
         \exp(2\admParam0) &< \min\{\Lambda_0,\Lambda_1\}.
       \end{align}
       This completes the proof of the proposition.
     \end{proof}

     We note the following bound: for any $L > 0$ there exists
     $C\admPm > 1$ so that for any unstable (or stable) curve $W$ such
     that $|W|\eum < L$, and for any $x',x''\in W$:
     \begin{align}\label{eq:dW-over-d-bound}
       C\admPm\inv &\le \frac{d\admPm^{W}(x',x'')}{d\admPm{}(x',x'')}\le C\admPm.
     \end{align}
     In fact, since unstable (\resp stable) curves are decreasing
     (\resp increasing), we have:
     \begin{align*}
       1\le \frac {d\eum^{W}(x',x'')}{d\eum(x',x'')}\le 2.
     \end{align*}
     Thus~\eqref{eq:dW-over-d-bound} follows by the equivalence of
     $d\adm$ with $(1+\alpha_{1}w)d\eum$ proved in
     Proposition~\ref{p_propertiesAdm}\ref{i_admeum} and the bound on
     the length of $W$.
     \begin{rmk}
       From now on, in an attempt to simplify the notation, we drop the
       superscripts $\pm$ from the $\admNamePm$-metric and we will always
       consider $\alpha = \alpha^{+}$.
     \end{rmk}
     \renewcommand{\admFunc}{\alpha} %
     \renewcommand{\admName}{\alpha} %
     \renewcommand{\adm}{_{\alpha}} %

     We now establish some properties of the
     $\admName$-metric which will be useful in the sequel.  Given a
     curve $W$ and two points $x',x''\in W$ we denote with
     $\dadm{W}(x',x'')$ (\resp $\deum{W}(x',x'')$) the
     $\admName$-length (\resp Euclidean length) of the subcurve of $W$
     bounded by $x'$ and $x''$.

     \begin{lem}\label{l_OneItFHat}
       For any $L > 0$ there exists $C > 0$ so that the following
       holds.  Let $n > 0$ and $W\subset\csp\setminus\sing{n}$ be an
       unstable curve.  Let $W_{k} = \cm^{k}W$ and assume that
       $|W_{n}|\eum < L$.  Let $x',x''\in W$ and denote
       $x'_{k} = \cm^{k}x'$ (likewise for $x''$); then:
       \begin{subequations}
         \begin{align}
           \dadm{W}(x'_0,x''_0) &\le C \dadm{W_{n}}(x'_{n},x''_{n})\label{e_adm-dominate}\\
           \sum_{j = 0}^{n}\deum{W_{k}}(x'_k,x''_k)&\le C \dadm{W_{n}}(x'_{n},x''_{n}).
         \end{align}
       \end{subequations}
     \end{lem}
     \begin{proof}

       Since $W\subset\csp\setminus\sing{n}$, we already observed that
       the function $x\mapsto \min(n, \Np(x))$ must be constant on
       $W$.  Let $\Np(W,n)$ denote this constant value.  Let us begin
       by proving an auxiliary result.
       \begin{sublem}\label{sublem:intermediate-step-expansion}
       There exists $C > 0$ such that
         if $n' \le \Np(W,n)$ and $|W_{n'}|\eum < L$,
         then 
         \begin{align}\label{eq:intermediate-step-expansion}
             \dadm{W}(x'_0,x''_0) &\le 
           C \dadm{W_{n'}}(x'_{n'},x''_{n'}).
         \end{align}
       \end{sublem}
       \begin{proof}
         We consider two cases.  Let $x_{0}' = (r_{0}',w_{0}')$ and
         choose $w_*$ sufficiently large. 
      
        (a) Assume $w'_{0}\leq w_*$: Lemma~\ref{LmFRet} gives a
           uniform upper bound on $\Np(x'_{0})$ (hence on $\Np(W,n)$).
           Notice that, even if we do not assume an upper bound on the
           Euclidean length of $W_{0}$, we have for any
           $x = (r,w)\in W_{0}$.
           \begin{align*}
             w\le w_{*}+2\Np\slopeb+L;
           \end{align*}
           Otherwise $\cm^{n'}x = (r_{n'},w_{n'})$ would satisfy
           $w_{n'} > w_{*}+\Np\slopeb+L$, but this is impossible by
           construction, since $w'_{n'}\le w_{*}+n'\slopeb$. and we
           assume $|W_{n'}|\eum < L$.  We now apply
           Proposition~\ref{p_propertiesAdm}\ref{i_expansionAdm} and
           conclude:
           \begin{align*}
             \dadm{W_{0}}(x_0',x_0'')
             &\le e^{n'\alpha_0} (1+\alpha_{1}(w_*+L+2\Np\slopeb))\cdot
               \dadm{W_{n'}}(x'_{n'},x''_{n'})\\\
             &\le C\dadm{W_{n'}}(x'_{n'},x''_{n'}),\end{align*}
           which yields the desired result.
         
         (b) If $w' > w_*$, then
           Proposition~\ref{p_largeEnergiesFUM}\ref{i_boundOnExcursion}
           ensures that $w'_k/w'_0\in (C\inv,C)$ for any
           $0\le k\le \Np(W,n)$. Since $|W_{n'}|\eum < L$, applying
           Proposition~\ref{p_largeEnergiesFUM}\ref{i_boundOnExcursion}
           again (to the inverse map) we conclude that a similar bound
           holds for every $w_{0}$ on $W_{0}$.  Since $w_*$ is chosen
           sufficiently large, $\alpha(x_k) = 1+\alpha_1\cdot w_k$ for
           any $x_k$ on the unstable curve joining $x'_k$ to $x''_k$.
           Iterating~\eqref{e_expansionAdm-alpha} we thus find, for
           unstable vectors tangent to $W$ and $\cm^{n'}W$:
           \begin{align*}
             \frac{|\deh x_{n'}|\adm}{|\deh x_0|\adm} %
             &\geq\frac{\alpha(x_{n'})}{\alpha(x_{0})}.
           \end{align*}
           This yields the desired result, since the ratio is
           uniformly bounded from below (once again since
           $w_{n'}/w_0\in(C\inv,C)$).
      
         We thus proved~\eqref{eq:intermediate-step-expansion}.
       \end{proof}
       In order to obtain~\eqref{e_adm-dominate}, it suffices to
       observe that given $W\subset\csp\setminus\sing n$, we can
       always write $\cm^{n} = \cm^{n_+}\circ\cmp^{l}\circ\cm^{n_-}$
       for some $l \ge 0$, $n_- = \Np(W,n)$ and
       $n_{+} \le \Np(W_{n-n_{+}})$.
       Then~\eqref{e_adm-dominate} follows
       from~\eqref{eq:intermediate-step-expansion} and from the
       uniform hyperbolicity of $\cmp$.

       The proof of the second estimate follows along similar lines. First
       we once again decompose
       $\cm^{n} = \cm^{n_+}\circ\cmp^{l}\circ\cm^{n_-}$ and then
       correspondingly we divide the sum into blocks where each block
       corresponds to one iteration of $\cmp$, or by $\cm^{n_-}$ and
       $\cm^{n_+}$ for the first and last block respectively.

       Let $0\le m < n$ be the starting index of some block and let
       $k\le N(x'_m)$.  We claim that:
       \begin{equation} \label{OneItFHat} \sum_{j=m}^{m+k}
         \deum{W_{j}}(x_j', x_j'')\leq C \dadm{W_{m+k}}(x_{m+k}',
         x_{m+k}'').
       \end{equation}
       In order to prove the claim, we again consider two cases.

  (a) Assume $w'_m\leq w_*.$ Then by
    Proposition~\ref{p_propertiesAdm}\ref{i_admeum} $d_E$ and
    $d_\alpha$ are equivalent for small energies and
    by~\eqref{e_adm-dominate} we obtain
    \begin{align*}
      \sum_{j = m}^{m+k}\deum{W_{j}}(x'_{j},x''_{j})\le C\sum_{j =
      m}^{m+k}\dadm{W_{j}}(x'_{j},x''_{j})\le C k \dadm{W_{m+k}}(x'_{m+k},x''_{m+k})
    \end{align*}
    which proves~\eqref{OneItFHat} since once again $k$ is uniformly bounded.

    (b) If $w_m' > w_*$ there might be many bounces during each period
    of the wall, \ie $k$ is not uniformly bounded.  Then using
    Proposition~\ref{p_propertiesAdm}\ref{i_admeum},~\eqref{e_expansionAdm-alpha},
    Lemma~\ref{LmFRet} and Proposition~\ref{p_largeEnergiesFUM},
    together with~\eqref{e_adm-dominate} we have
    \begin{align*}
      \sum_{j=m}^{m+k} \deum{W_{j}}(x_j', x_j'')
      &\le \brC \left[\sum_{j=m}^{m+k} \frac{\dadm{W_{j}}(x_j', x_j'')}{w_j'}\right]\\
      &\leq \brrC \dadm{W_{m+k}}(x_{m+k}', x_{m+k}'') \frac{\Np(x_m')}{w_m'} \le \brrrC
        \dadm{W_{m+k}}(x_{m+k}', x_{m+k}'').
    \end{align*}
    This proves that~\eqref{OneItFHat} holds also in case (b).

By \eqref{OneItFHat} we can write
  \begin{align*}
    \sum_{j=0}^n \deum{W_{j}}(x_j', x_j'')\leq C \sum_{l' = 0}^{l}
    \dadm{\cmp^{l}W_{n_-}}(\cmp^{l'}x_{n_-}', \cmp^{l'}x_{n_-}'') + C\dadm{W_{n}}(x_{n}', x_{n}'').
  \end{align*}
  By the uniform expansion of the $\admName$-metric shown in
  Proposition~\ref{p_propertiesAdm}(c) the sum on the right hand side
  is a geometric sum, whence:
  \begin{align*}
    \sum_{j=0}^n \deum{W_{j}}(x_j', x_j'')
    &\leq C \dadm{W_{n-n_+}}(x_{n-n_+}', x_{n-n_+}'') + C\dadm{W_{n}}(x_{n}', x_{n}'')
  \end{align*}
  from which we conclude the proof using once
  again~\eqref{e_adm-dominate}.
\end{proof}
Using the properties of the involution and the fact that the
$\admNamePm$-metrics are equivalent to each other, we obtain the
following corollary.

\begin{cor}\label{cor_adm-dominate-stable}
  For any $L > 0$, there exists $C > 0$ so that the following holds.
  Let $n > 0$ and $W\subset\csp\setminus \sing{n}$ be a curve so that
  $\cm^{n}W$ is a stable curve. Let $W_{k} = \cm^{k}W$ and assume that
  $|W_{k}|\eum < L$ for all $0\le k\le n$.  Let $x',x''\in W$ and
  denote $x'_{k} = \cm^{k}x'$ (likewise for $x''$). Then the following
  estimates hold.
  \begin{subequations}
    \begin{align}
      \dadm{W_{n}}(x'_n,x''_n) &\le C \dadm{W}(x'_{0},x''_{0})\label{e_adm-dominate-stable}\\
      \sum_{k = 0}^{n}\deum{W_{k}}(x'_k,x''_k)&\le C \dadm{W}(x'_{0},x''_{0}).
    \end{align}
  \end{subequations}
\end{cor}
As it is clear, \eg from~\eqref{e_expansionAdm-alpha}, the expansion
of unstable curves can be arbitrarily large if the curve is cut by a
grazing singularity.  However, as in the case of billiards
(see~\cite[Exercise 4.50]{ChM}), the divergence of the expansion rate
is integrable, as we show in the following lemma.
\begin{lem}\label{l_globalExpansion}\
  \begin{enumerate}
  \item For any $L > 0$, there exists a constant $C_* > 1$ so that for
    any unstable curve $W$ with $|W|\eum < L$ and any connected
    component $W'\subset \cm W$, we have
    \begin{align}\label{e_globalExpansion}
      |W'|\adm\le C_*|W|\adm^{1/4}
    \end{align}
  \item For any $\delta_{*} > 0$ and $k > 0$ there exists
    $\delta =\delta(\delta_{*}, k)\in(0,\delta_{*})$ so that if $W$ is
    an unstable curve with $|W|\adm\le \delta$, $W'$ is a connected
    subcurve of $\cm^{n}W$ and $\hat n(W) < k$, then
    $|W'|\adm\le \delta_*$.
  \end{enumerate}
  The corresponding estimates for stable manifolds hold true.
\end{lem}
\begin{proof}
  It suffices to prove this result with the $\alpha$-metric replaced by
  the auxiliary metric $|\cdot|_*$ defined by \eqref{StarMetric}.
   Assume first that $w\geq w_*$ on
  $W$, then by~\eqref{e_expansionAuxMetricA} and
  Lemma~\ref{LmPSlopeC0}(a) we conclude that the expansion along $W$ can
  be at most $1+2\curv'w/\curv w'$ which is uniformly bounded from
  above, hence $|W'|_*\leq C |W|_*$.

  Next, assume  that
  there is a point on $W$ so that $w\leq w_*$.  Let $u$ and $u'$ be the
  arclength parameters on $W$ and $W'$ respectively (with respect to
  $|\cdot|_*$-metric). Pick a large $T$ and consider two subcases.
 
 (i) $\tau\leq T$ on $W$: in this case Lemma~\ref{LmPSlopeC0} gives
    a uniform lower bound on $(\prpslope)'$ and
    hence~\eqref{e_expansionAuxMetricA} implies that
    $\left|\frac{d u'}{d u}\right|\leq \frac{\hat c}{w'}$.  Let
    $\tilde w'$ denote the minimal $w'$ on $W'$ and $\tilde u'$
    parametrize the point where the minimum is achieved.  Since
    $|\cslope|\geq\cK$ it follows that
    \begin{align*}
      w' \ge \tilde w'+c|u'-\tilde u'|,
    \end{align*}
    hence, we gather
    \begin{align*}
      \left|\frac{d u'}{d u}\right|\leq \frac{\bar{c}}{|u'-{\tilde
      u'}|}.
    \end{align*}
    Integrating the above estimate we obtain $|W'|_*^2\leq C |W|_*$ as
    needed.
  
  (ii) $\tau\geq T$ somewhere on $W$.  Then there is a (large)
    $\indCell\in\mathbb{N}$ such that
    $r+\tau(W)\subset(\indCell,\indCell+1)$, \ie
    $W'\subset\cell_{\indCell}^{-}$.  In this case Lemma
    \ref{LmPSlopeC0}(b) shows that, on $W'$, $(\prpslope)'$ is of
    order $1/\indCell$; thus repeating the argument from the previous
    subcase we obtain
    \begin{equation}
      \label{QuadLongT}
      |W'|_*^2\leq C \indCell |W|_*.
    \end{equation}
    \ignore{Note that on $W'$ we have $|(v-)'|\leq C/n$ since the
      backward travel time is of order $n$.  Also the height should
      be $\cO(1/n)$ close to the maximal since otherwise $W'$ would
      be in the recollision region. It follows that
      $|\ct_2'-\ct_1'|\leq \frac{\bar C}{\sqrt{n}}$ and hence
      $|W'|_*\leq \frac{\hat C}{\sqrt n}$.  That is} On the other
    hand, by Lemma~\ref{l_singularityStructure}\ref{i_boundedCells}
    and Remark~\ref{r_singularityStructureBackwards}, since
    $W'\subset\cell^{-}_{\indCell}$, we gather
    \begin{equation}\label{WPrSmall}
      |W'|_*^2\leq \frac{\bar{C}^2}{\indCell}.
    \end{equation}
    Multiplying \eqref{QuadLongT} and \eqref{WPrSmall} we obtain the
    result.

  We now prove item (b). Notice that it suffices to prove
      the case $k = 1$, since the general case follows by induction.
      let $\largew$ be sufficiently
      large and consider two possibilities.

      (I) If $W\subset \{w\le\largew\}$, then
      $\Np(x) < N_{*} = C\largew$, and thus $n \le N_{*}$: then the
      conclusion follows from item (a) since
      $|W'|\adm\le C_{*}^{4/3}|W|\adm^{1/4^{N_{*}}}$.

      (II) On the
      other hand, if $W\cap \{w > \largew\}\ne\emptyset$, by choosing
      $\largew$ sufficiently large and $\delta < 1$ we can guarantee
      that~\eqref{e_boundedExpansion} holds for all points in $W$,
      from which our conclusion immediately follows.
 \end{proof}
 \begin{rmk}\label{r_betterGlobalExpansion}
   Inspecting the proof of Lemma~\ref{l_globalExpansion}, we can
   obtain the slightly stronger result that if $W$ is
   unstable (\resp stable) and $W\subset\backReco$, (\resp
   $W\subset\reco$), then $|W'|\adm \le \Const|W|\adm^{1/2}$.
 \end{rmk}
 \begin{lem}\label{l_largeCell}\ %
   \begin{enumerate}
    \item\label{i_almostFinite} For any $\bar\indCell$, there exists
    $\delta =\delta(\bar\indCell) > 0$ so that for any u-curve
    $W\subset\csp$ with $|W|\adm < \delta$, $\cm W$ has at most $3$
    connected components that are \emph{not} contained in
    $\bigcup_{\indCell > \bar\indCell}\cell^-_\indCell$.%
    \item\label{i_almostFinite-large} There exists $\delta > 0$ and
    $w^{*} > 0$ so that if $|W|\adm < \delta$ and
    $W\subset\{w\ge w^{*}\}$, then $W$ intersects at most two
    $\cEs_{n}$'s.
    \item\label{i_almostFinite-induced} For any
    $\bar\indCell$ sufficiently large, there exists
    $\delta =\delta(\bar\indCell) > 0$ and $K > 0$ so that for any
    u-curve $W\subset\cspi$ with $|W|\adm < \delta$, $\cm W$ has at
    most $K$ connected components that are \emph{not} contained in
    $\bigcup_{\indCell > \bar\indCell}\cell^-_\indCell$.%
    \end{enumerate}
\end{lem}
\begin{proof}
  We begin with the proof of item~\ref{i_almostFinite}. Observe that
  by Proposition~\ref{p_propertiesAdm}\ref{i_admeum}, it suffices to
  prove the statement for the Euclidean metric $|\cdot|\eum$.  Let
  $W' = W\setminus\backReco$. By Lemma~\ref{l_forwardReco}(a2) we
  conclude that $W'$ is connected. Since
  $\backReco\cap\sing{+} = \emptyset$, we conclude that
  $\cm(W\cap\backReco)\subset\reco$ is also connected. Therefore it can
  contribute to at most one connected component, which is not in
  $\bigcup_{\indCell > \bar\indCell}\cell^-_\indCell$.  Hence, it
  suffices to prove that there exists $\delta > 0$ so that if
  $|W'|\eum < \delta$, $W'\cap\backReco = \emptyset$, then $\cm W'$
  has at most $2$ connected components that are not contained in
  $\bigcup_{\indCell > \bar\indCell}\cell^-_\indCell$.  This is
  immediate if $\bar\indCell < 2$.  Otherwise there would be a
  sequence of curves $W'_n$ converging to a point which would
  intersect at least three $\cell^{+}_\indCell$, with
  $\indCell\le \bar\indCell$. Hence it would intersect at least two
  $\sing+_\indCell$, with $\indCell\le \bar\indCell$.  Since
  $\sing+_\indCell$ are closed sets, we conclude that two curves
  $\sing+_\indCell$ and $\sing+_{\indCell'}$ must intersect, but this
  is impossible by Lemma~\ref{LmSingF}(a).

  In order to prove item~\ref{i_almostFinite-large}, let us assume
  that $W$ intersects at least three consecutive $\cEs_{n}$'s: let us
  denote them by $\cEs_{n-1}, \cEs_{n}$ and $\cEs_{n+1}$; in
  particular it must be that $W$ intersects both $\sing{n}$ and
  $\sing{n+1}$.  This implies that $\cm^{n+1}W$ will have a component
  $W'$ that joins $\sing0$ to $\sing{-1}$, and thus $|W'|\adm > c$ for
  some uniform $c > 0$ (see~\eqref{e_boundAlphaTau} ).
  However,~\eqref{e_boundedExpansion} guarantees that the expansion of
  $\cm^{n}$ is bounded above by $\Lambda^{*}$.  We conclude that
  $|W|\adm > c/\Lambda^{*}$.  Hence if $|W|\adm < c/\Lambda$, $W$ can
  only intersect $2$ of the $\cEs_{n}$'s.

  We now proceed to the proof of~\ref{i_almostFinite-induced}; fix
  $w_{*} > 0$ sufficiently large.  If
  $W\cap\{w\le w_{*}\}\ne\emptyset$ and $|W|\adm < 1$, then
  Lemma~\ref{LmFRet} allows to conclude that $\Np(x)\le N_{*}$ where
  $N_{*} = Cw_{*}$.  By part (a) there exists $\delta_*$ so that if
  $|W|\adm < \delta_{*}$, then $\cm W$ has at most $3$ connected
  components not contained in
  $\bigcup_{\indCell\ge\bar\indCell}\cell_{\indCell}^{-}$.  Moreover
  by Lemma~\ref{l_globalExpansion}, we can find
  $\delta = \Const\delta_{*}^{4^{N_{*}}}$ so that any connected
  component of $\cm^{n}W$, for $0\le n\le N_{*}$ is not larger than
  $\delta_{*}$.  Finally, observe that if $\bar\indCell$ is
  sufficiently large, then $\cell_{\indCell}^{-}\subset\cspi$ for any
  $\indCell\ge\bar\indCell$.  We can conclude by induction that
  $\cmp W$ has at most $3^{N_{*}}$ components not contained in
  $\bigcup_{\indCell\ge\bar\indCell}\cell_{\indCell}^{-}$, provided
  that $|W|\adm < \delta$.  Assume, on the other hand that
  $W\subset\{w\ge w_{*}\}$.  According to Theorem~\ref{t_normalForm},
  if $|W|_{\text{E}(\tau,I)} < 1/2$, then $W$ lies in at most $2$
  fundamental domains $D_{n}$, and therefore $\cmp W$ has at most $2$
  connected components.  By~\eqref{e_boundAlphaTau}, there exists
  $\delta > 0$ so that if $|W|\adm < \delta$, then
  $|W|_{\text{E}(\tau,I)} < 1/2$.  We conclude that item (b) holds
  even for large $w$.
\end{proof}


Finally, we conclude this section with a useful result about
   singularities (this is the analog of~\cite[Lemma 4.55]{ChM} for our
   system.)
   \begin{lem}\label{l_S-dense}
     The sets $\sing{+\infty}$ and $\sing{-\infty}$ are dense in
     $\csp$.
   \end{lem}
   \begin{proof}
     We prove the lemma for $\sing{+\infty}$ (the statement for
     $\sing{-\infty}$ follows by the properties of the involution).

     Assume by contradiction that $\csp\setminus\sing{+\infty}$
     contains an open ball $B$. Let $x\in B$ and $N = \Np(x)$. Then
     $B' = \cm^{N}B\subset\cspi$ and by invariance of
     $\csp\setminus\sing{+\infty}$ we gather that
     $B'\subset\cspi\setminus\sing{+\infty}\subset\cspi\setminus\singp{+\infty}$.

     We conclude that there exists an unstable curve $W\subset B'$ of
     positive length so that $\cmp^{n}|_W$ is smooth for every
     $n > 0$.  By
     Proposition~\ref{p_propertiesAdm}\ref{i_uniformHyperbolicity} 
     the length of the unstable curve $\cmp^{n}W$ would
     grow arbitrarily large.  Since unstable curves are decreasing, by
     definition of $\cspi$ and of the $\admName$-metric this means
     that for any $w^{*}$, there exists $n^*$ so that
     $\cmp^{n^{*}}W\cap \{w > w^{*}\}\neq\emptyset$.  But by the
     observation below Theorem~\ref{t_normalForm}(a) this means
     (choosing $w^{*}$ sufficiently large) that $\cmp^{n^{*}}W$ will
     intersect nontrivially at least two fundamental domains $D_k$,
     which in turn means that $\cmp^{n^*+1}|_{W}$ is discontinuous,
     which contradicts our assumptions.
   \end{proof}

\section{Distortion estimates}
\label{ScDistortion}
The previous sections dealt with $C^1$ estimates for the dynamics of
Fermi--Ulam Models.  However, it is well known that, in order to obtain good
statistical properties of hyperbolic maps, one needs a higher
regularity than $C^1$ for the purpose of controlling \eg distortion.
The necessary results about higher derivatives of the iterates of
$\cmp$ are presented in this section.
\subsection{Homogeneity strips}\label{s_homo}
In order to control distortion of u-curves, we introduce the so-called
\emph{homogeneity strips} $\homo_k\subset\csp$.  Fix $k_0 \in\bN $
sufficiently large, to be specified later, and define
$$\homo_0 = \{(\ct,w)\in\csp\st w > k_0^{-2}\}.$$
For $k\ge k_0$ define
\begin{align*}
 \homo_k = \{(\ct,w)\in\csp\st w\in({(k+1)^{-2}},k^{-2}]\}.
\end{align*}
By Proposition~\ref{p_propertiesAdm}\ref{i_expansionAdm}, we gather
that 
if $\cm x\in\homo_k$, the
expansion rate along unstable vectors at $x$ for the $\admName$-metric
is bounded below by $\const k^2$.  Moreover, by
Lemma~\ref{l_largeCell}, we can conclude that there exists
$\redCell > 0$ so that $\cell^{\pm}_\indCell \cap \homo_0 = \emptyset$
for any $\indCell > \redCell$.

As it is customary in the theory of billiards, we need to treat the
boundaries of $\homo_k$ as auxiliary (or \emph{secondary})
singularities.  For $k\ge k_{0}$, denote by
$\hhs_{k} = (0,1)\times\{k^{-2}\}$ and put
$\DS \hhs = \bigcup_{k\ge k_{0}}\hhs_{k}$.
Then we let
$\hsing 0 = \sing0 \cup \hhs$ and for any $n > 0$ we let:
\begin{align}\label{e_definitionHSing}
  \hsing{n} &= \sing n\cup \bigcup_{m =
              0}^{n}\cm^{-m}(\hhs\setminus\sing{-m}),&
  \hsing{-n} &= \sing{-n}\cup\bigcup_{m = 0}^{n} \cm^{m}(\hhs\setminus\sing{m}).
\end{align}
\begin{rmk}
\label{RkSCClosed}
Observe that $\cm\hhs$ (\resp $\cm^{-1}\hhs$) is a countable union of
stable (\resp unstable) curves that accumulate on the singular curves
$\sing{-1}\setminus\sing0$ (\resp $\sing{1}\setminus\sing0$).  Each
curve also terminates on $\sing{-1}$ (\resp $\sing{1}$).  In
particular each $\hsing n$ is a closed set.
\end{rmk}

As in Section~\ref{sec:accel-poinc-map}, we now extend
  these definitions  to the induced map.  First, define
  \begin{align*}
    \singt+_{\homo} = \hsing0 \cap \bigcup_{k\ge 0}(\hsing{k+1}\cap\cE_{k}),
  \end{align*}
then let $\hsingp+ = (\singt+_{\homo}\cap\cspi)\cup\partial\cspi$.  By a
similar construction we can define $\hsingp-$. then for any $n > 0$ we let:
\begin{align}\label{e_definitionHSingp}
  \hsingp{n+1} &= \hsingp n\cup\cmp\inv(\hsingp n\setminus\singp-)&
  \hsingp{-n-1} &= \hsingp{-n}\cup\cmp(\hsingp{-n}\setminus\singp+).
\end{align}
The auxiliary singularities will further cut any set into
components, which we call \emph{homogeneous components} (or
\emph{H-components}) An unstable (or stable) curve $W$ is said to be
\emph{weakly homogeneous} if $W$ belongs to only one strip
$\homo_{k}$.
\subsection{Unstable curves.}
In this section we study regularity properties of unstable
curves. By~\eqref{e_slopeCollision}, it suffices to establish the
regularity of the $p$-slope $\prpslope$.  In order to do so, we find
convenient to introduce the following notion: an unstable curve $W$ is
said to be $K$-\admiss{} if $\prpslope$ is $K$-Lipschitz (with respect
to the $\admName$-metric) on $W\setminus\backReco$ and
$(\prpslope)\inv$ is $K$-Lipschitz (with respect to the
$\admName$-metric) on\footnote{ In case that either
  $W\setminus\backReco$ or $W\cap\backReco$ is empty, we assume the
  Lipschitz condition to be trivially satisfied.} $W\cap\backReco$.
Using the involution, we can analogously define the class of stable
$K$-\admiss{} curves.  In this section we focus on properties of
unstable curves. 
Corresponding statements for
stable curves follow using the involution.  Later (in
Section~\ref{sec:invariant-manifolds}), we will use the
properties of stable curves.
\begin{prp} \label{PrBHold} For each $K > 0$ there exists
  $\bar{K} > 0$ such that the following holds.  Let $W$ be a
  weakly homogeneous mature unstable curve that is
  $K$-\admiss{}.  Then, for any $n > 0$, any H-component of $\cm^{n}W$
  is $\bar K$-\admiss{}.
\end{prp}
\begin{proof}
  Recall that for any $x\in W\setminus\sing{n}$ we denote with
  $\prpslope_{n}(x)$ the value of $\prpslope$ of the curve $\cm^{n}W$
  at the point $\cm^{n}x$.  In this proof we drop the superscript
  ``$-$'' in $\prpslope_n$ in order to simplify the notation.  We
  have, using~\eqref{e_slopeEvolution}, that
  $\cB_n=G(\tau_{n-1}, \cB_{n-1}, \cpar_{n-1})$ where
  \begin{align*}
    G(\tau, \cB, \cpar) &= \frac{\cB+\cpar}{1+\tau(\cB+\cpar)}.
  \end{align*}
  A direct computation gives
  \begin{subequations} \label{DeltaG}
    \begin{align}
      G(\tau, \cB', \cpar)-G(\tau, \cB'', \cpar)&=\frac{ (\cB'-\cB'')}{(1+\tau(\cB'+\cpar))(1+\tau(\cB''+\cpar))}, \\
      G(\tau, \cB, \cpar')-G(\tau, \cB, \cpar'')&=\frac{ (\cpar'-\cpar'')}{(1+\tau(\cB+\cpar'))(1+\tau(\cB+\cpar''))},\\
      G(\tau', \cB, \cpar)-G(\tau'', \cB, \cpar)&=\frac{(\cB+\cpar)^2 (\tau'-\tau'')}{(1+\tau'(\cB+\cpar))(1+\tau''(\cB+\cpar))}.\label{e_DeltaG-tau}
    \end{align}
  \end{subequations}
  Let $W_{n}$ be a $H$-component of $\cm^{n}W$ and for
  $0 \le k \le n $ let $W_k = \cm^{k-n}W_n$; let $x',x''\in W_0$ and
  for $0 \le k \le n $ let $x'_k = \cm^{k}x'$ and
  $x''_{k} = \cm^{k}x''$. Observe that by construction $x_k'$ and
  $x_k''$ belong to the same homogeneity strip.  We can further assume
  $W_0$ to be sufficiently short so that $d\eum(x_k', x_k'')\leq 1$ for
  any $0\le k\le n$ (otherwise we can partition $W_{0}$ into smaller
  subcurves which satisfy this requirement).  By
  construction, for any $0\le k < n$, the curve $W_{k}$ is contained in a single
  cell $\cell^{+}_{\indCell}$. In particular each $W_{k}$ is either
  contained or disjoint from $\backReco$.

  Now, for $0\le k < n$ we are going to define $\delta_{k}\ge0$
  as follows.
  Fix a large number
  $w^* > 0$; if $W_{k}\subset \backReco$ we let $\delta_{k}=0$.
  Otherwise, $W_{k}\cap\backReco = \emptyset$ and we let
  $\delta_{k} = \elll/\max\{w^{*},w'_{k}\}$.  Observe that, if $w^{*}$
  is sufficiently large,~\eqref{e_boundFreeFlight} allows to conclude
  that $\delta_{k}$ is a lower bound on $\tau(x)$ among all points $y$
  so that $d_{E}(y,W_{k})\le 1$.  Finally, let
  \begin{align*}
    \Delta_k' &=1+\delta_{k} \left(\cB_{k}'+\frac{\cK}{w_k'}\right),&
    \Delta_k''&=1+\delta_{k} \left(\cB_{k}''+\frac{\cK}{w_k''}\right).
  \end{align*}
Later (in Section~\ref{SSHolonomy}) we will
  consider the case where $x_k'$ and $x_k''$ do not necessarily belong
  to a common unstable curve.  In this case we define $\delta_{k}$
  based on the properties of the curve containing $x_{k}'$.  We thus
 state the next lemma under more general assumptions than
  needed in the current setting.
  \begin{lem} \label{LmB1-2step} %
    Let $W'$ and $W''$ be two mature unstable curves; let $x'\in W'$
    and $x''\in W''$; let $n > 0$ be so that for any $0\le k\le n$ the
    points $x'_k$ and $x''_k$ belong to the same cell
    $\cell_{\indCell}^{-}$, to the same homogeneity strip and
    $d\eum(x'_{k}, x''_{k}) < 1$.  Then the following
    estimates hold for $1 \le k\le n$:
    \begin{enumerate}
    \item If $x_k'\not\in\reco$, then
      \begin{align*}
        \left| \cB_k'-\cB_k''\right|
        &\le\frac{|\cB'_{k-1}-\cB_{k-1}''|}{\Delta_{k-1}' \Delta_{k-1}''}+ C
          \left[d_E(x_{k-1}', x_{k-1}'')+d_E(x_{k}', x_{k}'')\right].
      \end{align*}
    \item If $x_k'\in\reco$, then
      \begin{align} \label{B2Lip} %
        \left|\frac1{\cB_k'}-\frac1{\cB_k''}\right|&\leq C
                                          \left[|\cB_{k-1}'-\cB_{k-1}''|+d_E(x_{k-1}',
                                          x_{k-1}'')+d_E(x_{k}', x_{k}'')\right] .
      \end{align}
      Moreover, if additionally $k \ne n$:
      \begin{align*}
        \left| \cB_{k+1}'-\cB_{k+1}''\right|
        &\le\frac{|\cB'_{k-1}-\cB_{k-1}''|}{\Delta_{k}' \Delta_{k}''}+\\
        &\phantom\leq + C
          \left[ d_E(x_{k-1}', x_{k-1}'')
          + d_E(x_{k}', x_{k}'')
          + d_E(x_{k+1}', x_{k+1}'') \right].
      \end{align*}
    \end{enumerate}
  \end{lem}
  Before giving the proof of the above lemma, let us see how it yields
  Proposition~\ref{PrBHold}.  In our case $W' = W'' = W_0$.  Let us
  first assume that $W_{0}\cap\reco = \emptyset$. We consider two
  possibilities: either $W_{n}\cap\reco = \emptyset$ or
  $W_{n}\subset\reco$.

  In the first case, iterating the estimates of parts (a) and (b) of
  the lemma we get, since $x_n'\not\in\reco$:
  \begin{align}%
    \left|\cB_n'-\cB_n''\right|%
    & \leq \frac{\left|\cB_0'-\cB_0''\right|}{\prod_{j=0}^{n-1}
      \left[\Delta_{j}' \Delta_{j}''\right]}\label{PSlopeRec}
      +C \sum_{j=0}^n d_E(x_j', x_j'')\\
    &\leq \left|\cB_0'-\cB_0''\right|+C \sum_{j=0}^n d_E(x_j', x_j'').\notag\\
    &\leq K d_\alpha(x_0',x_0'') + C \sum_{j=0}^n d_E(x_j', x_j'').\notag\\
    &\leq C(K+1)d\adm(x_n', x_n'').\notag
  \end{align}
  where in the last passage we invoked Lemma~\ref{l_OneItFHat}.

  In the second case, we iterate the estimates of parts (a) and (b)
  until step $n-1$ and use~\eqref{B2Lip} at the last step, which gives:
  \begin{align*}
    \left|\frac1{\cB_k'}-\frac1{\cB_k''}\right|&\leq C \frac{\left|\cB_0'-\cB_0''\right|}{\prod_{j=0}^{n-1}
                                                 \left[\Delta_{j}' \Delta_{j}''\right]}
                                                 +C \sum_{j=0}^n d_E(x_j', x_j'')
  \end{align*}
  from which we conclude as above.

  We now consider the case $W_{0}\subset\reco$. By
  Lemma~\ref{LmB1-2step}(a) 
  \begin{align*}
    |\cB'_{1}-\cB''_{1}|\le \left|\frac1{\cB'_{0}}-\frac1{\cB''_{0}}\right|%
    \frac{\cB'_{0}\cB''_{0}}{\Delta_{0}'\Delta_{0}''}%
    +C \left[d_E(x_{k-1}', x_{k-1}'')+d_E(x_{k}', x_{k}'')\right].
  \end{align*}
  Notice that
  \begin{align*}
    \frac{\cB'_{0}\cB''_{0}}{\Delta_{0}'\Delta_{0}''}\le\frac{\cB'_{0}\cB''_{0}}{(1+\delta_{0}\cB'_{0})(1+\delta_{0}\cB''_{0})}\le\frac1{\delta_{0}^{2}}.
  \end{align*}
  Since $W_{0}\subset\reco$, and $\reco\cap\backReco = \{\xc\}$, we
  conclude that $W_{0}\cap\backReco = \emptyset$ and so $\delta_{0} >
  0$.  In particular we have:
\begin{align*}
  |\cB'_{1}-\cB''_{1}|\le
  +C \left[\left|\frac1{\cB'_{0}}-\frac1{\cB''_{0}}\right|+
  d_E(x_{k-1}', x_{k-1}'')+d_E(x_{k}', x_{k}'')\right].
  \end{align*}
  We then argue as in the other case (for each of the two
  subcases involving $W_{n}$), but starting from $k = 1$ and we obtain
  the result.
\end{proof}
It remains to establish Lemma \ref{LmB1-2step}.
\begin{proof}[Proof of Lemma \ref{LmB1-2step}]
  (a) We have
  \begin{align*}
  \cB_k'-\cB_k''&= \left[G(\tau_{k-1}', \cB_{k-1}', \cpar_{k-1}')-
                  G(\tau_{k-1}', \cB_{k-1}'', \cpar_{k-1}')\right]&\\
                &\phantom = +\left[G(\tau_{k-1}', \cB_{k-1}'', \cpar_{k-1}')-
                  G(\tau_{k-1}', \cB_{k-1}'', \cpar_{k-1}'')\right]&\\
                &\phantom = +\left[G(\tau_{k-1}', \cB_{k-1}'', \cpar_{k-1}'')-
                  G(\tau_{k-1}'', \cB_{k-1}'', \cpar_{k-1}'')\right]
                &=I+\RmII+\RmIII.
\end{align*}
We now estimate each of these three terms separately
using~\eqref{DeltaG}.
\begin{align*}
  |I| &= \frac{ |\cB_{k-1}'-\cB_{k-1}''|}{(1+\tau_{k-1}'(\cB_{k-1}'+\cpar_{k-1}'))(1+\tau_{k-1}'(\cB_{k-1}''+\cpar_{k-1}'))}
        \leq \frac{ |\cB_{k-1}'-\cB_{k-1}''|}{\Delta_{k-1}' \Delta_{k-1}''}.
\end{align*}
Let us now consider the second term. We have
\begin{align*}
  |\RmII| &= \frac{ |\cpar_{k-1}'-\cpar_{k-1}''|}
            {(1+\tau_{k-1}'(\cB_{k-1}''+\cpar_{k-1}'))(1+\tau_{k-1}'(\cB_{k-1}''+\cpar_{k-1}''))} \\
          &\leq \frac{ |\cpar_{k-1}'-\cpar_{k-1}''|}
            {(1+\tau_{k-1}'\cpar_{k-1}')(1+\tau_{k-1}'\cpar_{k-1}'')}.
\end{align*}
The numerator equals
\begin{align*}
  2\left|\frac{\curv_{k-1}' w_{k-1}''-\curv_{k-1}''w_{k-1}'}{w_{k-1}' w_{k-1}''}\right|
  &\leq 2\frac{\curv_{k-1}'|w_{k-1}'-w_{k-1}''|}{w_{k-1}' w_{k-1}''}+
    2\frac{|\curv_{k-1}'-\curv_{k-1}''|}{w_{k-1}''}.
\end{align*}
We split the discussion in two cases:
\begin{enumerate}
\item[(A)]If $|w_{k-1}'|\leq 2$ then we obtain
\begin{align*}
  |\cR'_{k-1}-\cR''_{k-1}| &\leq C \frac{|\ct_{k-1}'-\ct_{k-1}''|+|w_{k-1}'-w_{k-1}''|}{w_{k-1}' w_{k-1}''}.
\end{align*}
Since $\delta_{k-1} > \bar\delta > 0$ (because
$w'_{k-1} < 2 < w^{*}$) 
\begin{align*}
  |\RmII| &\leq \frac{C d_E(x_{k-1}', x_{k-1}'')}
            {\left(1+\frac{2\delta_{k-1} \curv_{k-1}'}{w_{k-1}'}\right)
            \left(1+\frac{2\delta_{k-1} \curv_{k-1}''}{w_{k-1}''}\right)
            w_{k-1}' w_{k-1}''} \leq
            \frac{\brC d_E(x_{k-1}', x_{k-1}'')}
            {\delta_{k-1}^2 \curv_{k-1}' \curv_{k-1}''}\\
          &\leq \brrC d_E(x_{k-1}', x_{k-1}'').
\end{align*}
\item[(B)] Otherwise, if $w_{k-1}'>2$ then we bound the numerator from above by
  $\brC d_E(x_{k-1}', x_{k-1}'')$ and the denominator from below by $1$,
  which also yields $|\RmII|\leq \brrC d_E(x_{k-1}', x_{k-1}'')$.
\end{enumerate}

To estimate $(\RmIII)$, consider two cases.

(A) If $w_{k-1}'\leq w^{*}$ then
\begin{align*}
 |\RmIII| &\le \frac{(\cB_{k-1}''+\cpar_{k-1}'')^2 |\tau_{k-1}'-\tau_{k-1}''|}
             {(1+\delta_{k-1}(\cB_{k-1}''+\cpar_{k-1}''))^2}
             \leq \frac{|\tau_{k-1}''-\tau_{k-1}'|}{\delta_{k-1}^2} \\&
             \leq \frac{|\ct_{k-1}'-\ct_{k-1}''|+|\ct_{k}'-\ct_{k}''|}{\bar\delta^2}
\end{align*}
where in the last step we used the fact that, since $x'_{k}$ and
$x''_{k}$ belong to the same cell $\cell^{-}_{\indCell}$, we have
$|\tau_{k-1}'-\tau_{k-1}''|\le |(r'_k-r'_{k-1})-(r''_{k}-r''_{k-1})|$
and the fact that if $w'_{k-1} < w^{*}$, then
$\delta_{k-1} > \bar\delta$.

(B) If $w_{k-1}'> w^{*}$, then Corollary~\ref{CrHESqueeze}(b) allows us
to estimate the numerator of~\eqref{e_DeltaG-tau} from above by
$C[|\ct_{k-1}'-\ct_{k-1}''|+|\ct_{k}'-\ct_{k}''|]$ and the denominator
by $1$, obtaining:
\begin{align*}
  |\RmIII|\leq C[|\ct_{k-1}'-\ct_{k-1}''|+|\ct_{k}'-\ct_{k}''|].
\end{align*}
Hence, either in case (A) or case (B) we conclude that
\begin{align*}
   |\RmIII|\leq C d_E(x_{k-1}', x_{k-1}'')+d_E(x_{n}', x_{n}''),
 \end{align*}
 which completes the proof of part~(a).

 In order to prove part (b), we begin by estimating
 $|\cB_k'-\cB_k''|$ in terms of $|\cB_{k-1}'-\cB_{k-1}''|.$

 If $x_{k-1}'\in \homo_0$ (and thus $x_{k-1}''\in \homo_0$ by
 assumption) then we have
 \begin{equation}\label{H0Reco}
   |\cB_k'-\cB_k''|\leq
   |\cB_{k-1}'-\cB_{k-1}''|+C\left[d_E(x_{k-1}', x_{k-1}'')+d_E(x_{k}', x_{k}'')
   \right]
 \end{equation}
 because we can bound from below the denominators of $I, \RmII$ and
 $\RmIII$ by 1, and the numerators of $\RmII$ and $\RmIII$ are
 $$O\left(d_E(x_{k-1}', x_{k-1}'')\right)\text{ and }
 O\left(d_E(x_{k-1}', x_{k-1}'')+d_E(x_{k}', x_{k}'')\right)$$
 respectively due to a lower bound on $w_{k-1}' $ and $w_{k-1}''$ and
 the upper bound on $\cB''_{k-1}$ given by Corollary~\ref{CrHESqueeze}
 (since $x_k\in\reco$, we have $x_{k-1}\nin\reco$).
 Combining~\eqref{H0Reco} with the already established part (a) for
 $x_{k+1}'\nin\reco$ we obtain the estimates of part (b) in case
 $x_{k-1}'\in \homo_0$ (note that we have uniform lower bounds on
 $\cB_{k}'$ and $\cB_{k}''$, so that also~\eqref{B2Lip} follows).

 Next, we consider the case $x'_{k-1},x''_{k-1}\in \homo_j$ for some
 $j>0$.  Then $C\inv w_{k-1}'\leq w_{k-1}''\leq C
 w_{k-1}'$. Observe that our assumptions give a uniform upper
   bound on $w'_{k-1}$ and uniform upper bound on $\cB'_{k-1}$.  In
   fact, since $x'_k\in \reco$, it follows that $x'_{k-1}\in\backReco$.
   Thus $\cm\inv x'_{k-1}\nin\backReco$ (this follows from
   Remark~\ref{r_thereCanBeOnlyOne}, because $\xc\nin\homo_j$ for any
   $j$).  Hence the required upper bound on $\cB'_{k-1}$ follows from
   Lemma~\ref{LmPSlopeC0}(b), since we assume $W$ to be
   mature.\ignore{\footnote{ Notice that if we did not assume $W$ to be mature
     but just unstable, the upper bound would hold provided that
     $k > 1$.}}

   Since $\cB'_{k-1}$ is uniformly bounded, assuming $k_0$ in the
   definition of the homogeneity strips to be sufficiently large, we
   have the following estimates

 \begin{align*}
   c \frac{\cR'_{k-1}}{1+\tau'_{k-1} \cR'_{k-1}}
   \leq \cB'_k \leq c^{-1} \frac{\cR'_{k-1}}{1+\tau'_{k-1} \cR'_{k-1}}, \quad
   \frac{c}{w_{k-1}}\leq  \cR'_{k-1}  \leq \frac{c\inv}{w_{k-1}} .
 \end{align*}

 Hence
 \begin{equation}
     \label{BW}
     \frac{\brc}{w_{k-1}+\tau_{k-1}}\leq \cB'_k\leq \frac{\brc^{-1}}{w_{k-1}+\tau_{k-1}}.
   \end{equation}
   Without loss of generality we may assume that
   $\tau_{k-1}'\geq \tau_{k-1}''$.  Then \eqref{BW} shows that
   \begin{equation}
     \label{BLongShort}
     \cB_k'\leq C \cB_k'' .
   \end{equation}
   We now estimate $I, \RmII$ and $\RmIII$ as follows.
   \begin{align*}
  |I| &\leq |\cB_{k-1}'-\cB_{k-1}''|,\\
 |\RmII|&\leq \frac{C d_E(x_{k-1}', x_{k-1}'') (w_{k-1}')^{-2}}{\left(1+\frac{c \tau_{k}'}{w_{k-1}'}\right)^2}
  \leq C (\cB_k')^2 d_E(x_{k-1}', x_{k-1}''),\\
 |\RmIII|&\leq \frac{C \left[d_E(x_{k-1}', x_{k-1}'')+d_E(x_{k}', x_{k}'')\right]
(w_{k-1}' w_{k-1}'')\inv}
  {\left(1+\frac{c \tau_{k-1}'}{w_{k-1}'}\right)\left(1+\frac{c \tau_{k-1}''}{w_{k-1}''}\right)}\\
  &\leq C \cB_k' \cB_k''  \left[d_E(x_{k-1}', x_{k-1}'')+d_E(x_{k}', x_{k}'')\right] .
\end{align*}
Here the second inequality in the estimates of $\RmII$ and $\RmIII$
follow from~\eqref{BW}.

Combining these estimates with \eqref{BLongShort} we conclude
that\footnote{ Observe that \eqref{eq:B2Lip-stronger} holds trivially
  also if $x'_{k-1}\in\homo_{0}$, by~\eqref{H0Reco} and the fact that
  we have a uniform lower bound on $\cB_k'$, as the flight time
  $\tau_{k-1}'$ is bounded (see Lemma~\ref{LmPSlopeC0})}
\begin{align}\label{eq:B2Lip-stronger}
  |\cB_k'-\cB_k''|&\leq |\cB_{k-1}'-\cB_{k-1}''|+C \cB_k' \cB_k'' \left[d_E(x_{k-1}', x_{k-1}'')+d_E(x_{k}', x_{k}'')\right],
\end{align}
which yields~\eqref{B2Lip} since we have a uniform lower
bound on $\prpslope_k$ in the recollision region (see
Lemma~\ref{LmPSlopeC0}).  Combining the above bound with the bound at step
$k+1$ already established in part (a), we conclude
\begin{align*}
   |\cB_{k+1}'-\cB_{k+1}''|&\leq\frac{|\cB_{k-1}'-\cB_{k-1}''|}{\Delta_{k}'\Delta_{k}''}
  +\\&\phantom\leq+   C \frac{\cB_k' \cB_k''}{(1+\delta_k\cB_k')(1+\delta_k\cB_k'')}
  \left[d_E(x_{k-1}', x_{k-1}'')+d_E(x_{k}',
    x_{k}'')\right]\\
  &\phantom\leq+C\left[d_E(x_{k}', x_{k}'')+d_E(x_{k+1}', x_{k+1}'')\right].
\end{align*}
Since
\begin{align*}
  \frac{\cB_k'}{1+\delta_k \cB_k'}\leq \frac{1}{\delta_k},
  &\frac{\cB_k''}{1+\delta_k \cB_k''}\leq \frac{1}{\delta_k}
\end{align*}
part (b) follows, because in the region under consideration, $1/\delta_k$ admits a 
uniform in $k$ upper bound.
\end{proof}
The proof of Lemma~\ref{LmB1-2step} provides some additional useful
information which we record for a future use.
\begin{lem}\label{LmBHoldAdd}\
  \begin{enumerate}
    \item\label{i_lower-bound-on-delta} For any $\brdelta > 0$ there
    is a constant $K(\brdelta)$ such that if $W_n$ is an H-component
    of $\cm^n W$ contained in $\reco$ and if $\tau_{n-1}\geq \brdelta$
    on $W_n$ then $\prpslope_n$ is $K(\brdelta)$ Lipschitz on $W_n$.
  \item\label{i_lower-bound-on-tau} There exist constants $T$ and $K_2$ such that if
    $\tau_{n-1}\geq T$ on $W_n$ then $\prpslope_n|_{W_n}$ is
    ${K_2}/{T^2}$ Lipschitz.
  \end{enumerate}
\end{lem}
\begin{proof}
  Part~\ref{i_lower-bound-on-delta} holds since the assumption that
  $x'_k\not\in \reco$ is only used in Proposition~\ref{PrBHold} to
  obtain a uniform lower bound on the flight time, and such bound is
  now explicitly assumed .

  Moreover, the assumptions in part~\ref{i_lower-bound-on-tau} allow us to
  estimate $\delta^2$ in the denominators of $I,$ $\RmII,$ and
  $\RmIII$ by $T^2$ obtaining
  \begin{align*}
    |\cB_{n}'-\cB_n'' |         %
    &\leq C \frac{|\cB_{n-1}'-\cB_{n-1}''|+d_E(x_{n-1}', x_{n-1}'')+d_E(x_{n}', x_{n}'')}{T^2}\\
    &\leq \brC \frac{|\cB_{n-1}'-\cB_{n-1}''|+d_E(x_{n}', x_{n}'')}{T^2}.
  \end{align*}
  It remains to note that we have a uniform Lipschitz bound on
  $\prpslope_{n-1}$. In fact, if $W_{n-1}\not\subset\reco$ then this
  bound follows from Proposition~\ref{PrBHold}. If
  $W_{n-1}\subset\reco$ then the bound follows from the already
  established part~\ref{i_lower-bound-on-delta}. Indeed, the fact that
  $\tau_n\geq T$ implies (provided that $T$ is sufficiently large)
  that $W_{n-1}$ is close to $x_C$ giving the necessary lower bound on
  $\tau_{n-1}$.
\end{proof}
\begin{cor} \label{CrUMLip} %
  For any $L > 0$ there exists a constant $\hat K > 0$ such that if
  $W\subset\csp\setminus\sing{-\infty}$ is an unstable curve such that
  $|W|\eum < L$ and $\cm^{-n}W$ is unstable for each $n$,
  then
  $W$ is $\hat K$-\admiss{}.  In particular, unstable manifolds are
  $\hat K$-\admiss{}.
\end{cor}
\begin{proof}
  Let $(n_k)_{k = 0}^{\infty}$ be a strictly increasing sequence of
  non-negative numbers such that $\cm^{-n_k} W\not\subset \reco$.  We
  will now show that there exists $K > 0$ so that $\cm^{-n_{0}}W$ is
  $\cB_{-n_{0}}$ is $K$-Lipschitz. This implies that
  $\cm^{-n_{0}}W$ is $K$-admissible, and by Proposition~\ref{PrBHold}
  we could conclude that $W$ is $\hat K$-\admiss{}, with
  $\hat K = \bar K(K)$.

  For any $x',x''\in \cm^{-n_{0}}W$, arguing as in~\eqref{PSlopeRec}
  we obtain that:
  \begin{align} \label{SlopeBack}%
    \left|\cB_{n_{0}}'-\cB_{n_0}''\right|%
    &\leq
      \frac{\left|\cB_{-n_k}'-\cB_{-n_k}''\right|}{\prod_{j=-n_k}^{-n_{0}-1}
      \left[\Delta_j' \Delta_j''\right]} +C \sum_{j=-n_k}^{-n_{0}}
      d_E(x_j', x_j'') .
  \end{align}
  By Lemma~\ref{l_OneItFHat}, the second term of the right hand side
  is smaller than $C d_\alpha(x'_{-n_{0}}, x''_{-n_{0}})$.  On the
  other hand, the first term tends to $0$ as $k\to\infty$, since the
  numerator is bounded above by Corollary~\ref{CrHESqueeze} while
  the denominator tends to infinity due to
  Proposition~\ref{p_propertiesAdm}.
\end{proof}
We now fix $L = 1$ and declare an unstable curve $W$ \emph{\admiss{}}
if $|W|\eum < 1$ and if it is $2\hat K$-admissible, where $\hat K$
is the one given in Corollary~\ref{CrUMLip} for $L = 1$.
\begin{rmk}
  \label{rmk:uniform-bound-euclidean}
  As a matter of fact, it suffices to assume that
  $W\subset\csp\setminus\sing{-\infty}$ is an unstable curve to
  conclude that there exists $L > 0$ so that $|W|\eum < L$ and
  $\cm^{-n}W$ is unstable for each $n$.  We will explain this in
  Section~\ref{SSInvMan}.  For the moment it is convenient to fix
  ideas and set (arbitrarily) $L = 1$.
\end{rmk}

\subsection{Unstable Jacobian.}
Given a mature unstable curve $W$, $n\in\integers$ and $x\in
W\setminus\sing n$, we denote with
\begin{align*}
  \jacW \cm^{n}(x) = \frac{|D_{x}\cm^{n}(dx)|\adm}{|dx|\adm}
\end{align*}
the Jacobian of the restriction of the map $\cm^{n}$ to $W$ at the
point $x$ in the $\admName$-metric (here $dx$ denotes a nonzero
tangent vector to $W$ at $x$).

\begin{lem} \label{LmEXpHolder} %
  Given $L > 0$ there exists $\bar K > 0$ so that for
  any mature \admiss{} unstable curve
  $W\subset\csp\setminus\sing-$ so that $\cm W$ belongs to a single
  H-component and $|W|\adm\le L$
      then $\ln \jacW\cm (x)$ is a H\"older function of
  constant $\bar K$ and exponent ${1/12}$ with respect to the
  $\admName$-metric on $W$.

  Moreover, let $W'$ be a subcurve of $W$ which is mapped by $\cm^{l}$
  to a H-component of $\cm^{l}W$.  If $l \le \Np(x)$ for any $x\in W'$
  then $\ln\jacW\cm^{l}(x)$ is a H\"older function on $W'$ of constant
  $\bar K$ and exponent $1/12$ with respect to the $\admName$-metric
  on $W'$.
\end{lem}
\begin{proof}
  In this proof we again drop the superscript “$-$” from $\cB$
  for the ease of notation.
  \renewcommand{\prpslope}{\cB} 
  In view of~\eqref{StarMetric}
  and~\eqref{e_expansionAuxMetricA}, we have
  \begin{align*}
    \jacW\cm(x) &= \exp\left(\alpha_0(\If_\backReco(x)-\If_{\reco}(x))\right)
                 H(x, \cm x),
  \end{align*}
  where
  \begin{align} \label{DefH2} %
    H(x, \brx) &=
    \frac
    {(\bar{\prpslope} \brw+2\bar{\curv})(1+\alpha_1 \brw)}
    {\bar{\prpslope} \brw(1+\alpha_1 w)}.
  \end{align}
  Observe that the exponential term multiplying $H$ is actually
  constant on $W$, because both $W$ and $\bar W = \cm W$ are contained
  in a single H-component and thus $W$ is either contained in or
  disjoint from $\reco$ or $\backReco$.

  We claim that
  \begin{align} \label{LogH} %
    \ln H &= \ln (\bar{\prpslope} \brw+2\bar{\curv})+\ln (1+\alpha_1 \brw)
            -\ln \bar{\prpslope}-\ln \brw -\ln (1+\alpha_1 w)
  \end{align}
  is uniformly H\"older on $W\times\overline{W}$.

  Suppose first that $\overline{W}\cap\reco = \emptyset$.  Let $(x', \bar{x}')$
  and $(x'', \bar{x}'')$ be two points on $W\times \overline{W}$.
  Note that if $\zeta\ge a > 0$, then $\zeta\mapsto \ln(\zeta)$ is
  Lipschitz with constant $a\inv.$ Therefore
  $\ln (1+\alpha_1 w)$ (and similarly $\ln (1+\alpha_1 \brw)$) is
  uniformly Lipschitz on $W$ (\resp, on $\overline{W}$) with respect to the
  Euclidean metric (and thus to the $\admName$-metric).
  Observe that by the lower bound for large energies in
  Corollary~\ref{CrHESqueeze} (and since $\bar\curv\ge\cK$) we have
that $\bar{\prpslope} \brw+2\bar{\curv}\ge C(\brw +1).$ Hence
  the upper bound of Corollary \ref{CrHESqueeze}
  and the fact that $\bar x'\nin\reco$ give
  \begin{align*}
    |\ln (\bar{\prpslope}''
    \brw''+2\bar{\curv}'')&-\ln(\bar{\prpslope}'
    \brw'+2\bar{\curv}')|\\
    &\le \frac{C}{\bar w'+1}\left|\bar{\prpslope}''
      \brw''-\bar{\prpslope}'
      \brw'\right| + \left|\bar{\curv}'-\bar{\curv}''\right|\\
      &\le C|\bar{\prpslope}''-\bar{\prpslope}'|+Cd\adm(\brx',\brx''),
  \end{align*}
  from which we obtain a uniform Lipschitz estimate on
  $\ln (\bar{\prpslope} \brw+2\bar{\curv})$, using
  Proposition~\ref{PrBHold}.  Next, if $\overline{W}\subset \homo_0$, then
  $\brw > C$ and thus $\ln\brw$ is uniformly Lipschitz. On the other
  hand, if $\overline{W}\subset \homo_k$ for some $k > 0$, then
  $k^3 |\brw'-\brw''|\leq C$, which implies
  $k^2|\brw'-\brw''|\leq {C}{|\brw'-\brw''|}^{1/3}$.  Since
  $\brw > (k+1)^{-2}$, we obtain
  \begin{align*}
    |\ln \brw'-\ln \brw''|\leq C k^2 |\brw'-\brw''|\leq \brC |\brw'-\brw''|^{1/3}.
  \end{align*}

  Finally
  \begin{align*}
    \left|\ln \bar{\prpslope'}-\ln \bar{\prpslope''}\right|=
    \left|\ln \frac{\bar{\prpslope'}}{\bar{\prpslope''}}\right|\leq
    \frac{|\bar{\prpslope'}-\bar{\prpslope''}|}{\bar{\prpslope''}}.
  \end{align*}
  Let $T$ be the constant from
  Lemma~\ref{LmBHoldAdd}\ref{i_lower-bound-on-tau}. If the flight time
  is less than $T$ then we can estimate the numerator by
  $2\hat K d\adm(\brx', \brx'')$ due to Proposition~\ref{PrBHold}
  while the denominator is uniformly bounded from below due to
  Lemma~\ref{LmPSlopeC0} (for small $w'$) and
  Corollary~\ref{CrHESqueeze} (for large $w'$).  On the other hand, if
  the flight time is greater than $T$ then the numerator is less than
  $K_2 d\adm(\brx', \brx'')/T^2$ due to
  Lemma~\ref{LmBHoldAdd}\ref{i_lower-bound-on-tau} while the
  denominator is of order $T\inv$ by Lemma~\ref{LmPSlopeC0}.

  This completes then proof of the fact that $\ln H$ is uniformly
  H\"older on $W\times\overline{W}$ in case $\overline{W}\cap\reco = \emptyset$.  In
  fact, our analysis shows that all terms in~\eqref{LogH} are
  Lipschitz except for $\ln\brw$ which may be ${1/3}$-H\"older.

  The analysis in case $\overline{W}\subset\reco$ is similar except that
  we rewrite
  \begin{align*}
    \frac{\bar{\prpslope} \brw+2\bar{\curv}}{\bar\prpslope}&=
    \brw+2\frac{\bar\curv}{\bar{\prpslope}}.
  \end{align*}
  Then Proposition~\ref{PrBHold} implies that the above expression is
  Lipschitz with respect to the $\admName$-metric.
  Lemma~\ref{lem:vertical-unstable} yields that it is uniformly
  bounded from below, which 
  implies that $\ln(\brw+2{\bar\curv}/{\bar{\prpslope}})$ is Lipschitz
  and therefore that $\ln H$ is ${1/3}$-H\"older also in case
  $\overline{W}\subset\reco$.

  To prove the H\"older continuity of $\ln \jacW\cm$ it remains to note
  that, in view of Lemma~\ref{l_globalExpansion}, the map
  $\cm|_{W}$ is uniformly ${1/4}$--H\"older with respect to the
  $\admName$-metric. \medskip

  We now proceed to the proof of the second  statement.
  First note that if $w$ is bounded, then the H\"older continuity of
  $\ln \jacW\cm^{l}$ follows from the H\"older continuity of
  $\ln \jacW\cm$ since $\Np(x)$ is uniformly bounded.

  In case $w$ is large, that is $w\geq w_*$, then denote with $x_{n} =
  \cm^{n}x$ and observe that:
  \begin{align*}
    \jacW\cm^{l}(x)&=
    \prod_{j=1}^{l} \left(\frac{1+\alpha w_{j}}{1+\alpha w_{j-1}}\right)
    \left(\frac{2\curv_{j} +\prpslope_{j} w_{j}}{ \prpslope_{j} w_{j}}\right)\\
                    &= \left(\frac{1+\alpha w_{l}}{1+\alpha w_{0}}\right)
    \prod_{j=1}^{l} \left(\frac{2\curv_{j} +\prpslope_{j} w_{j}}{ \prpslope_{j} w_{j}}\right)
  \end{align*}
  Once again, $\ln(1+\alpha_1 w_0)-\ln(1+\alpha w_{l})$ is Lipschitz
  on $W'\times \cm^{l}W'$.  Next we show that, in the high energy
  regime, $\cm^{l}$ is uniformly Lipschitz. Indeed, at each step $j<l$,\;
  $\jacW\cm(\cm^{j}x)=1+O(w_0^{-1})$ due
  to~\eqref{e_expansionAuxMetricA} and Corollary~\ref{CrHESqueeze}.
  On the other hand, by \eqref{Per-W}, $l$ is at most of order $w$,
  giving an uniform upper bound on $\jacW\cm^{l}(x)$.  It remains to
  handle the product. Let $x_j'$ and $x_j''$ be two orbits.  Then
  \begin{align*}
    &\left|\sum_{j=1}^{l} \ln\left(\frac{2\curv_{j}' +\prpslope_{j}' w_{j}'}{ \prpslope_{j}' w_{j}'}\right)
      - \ln\left(\frac{2\curv_{j}'' +\prpslope_{j}'' w_{j}''}{\prpslope_{j}'' w_{j}''}\right)\right|\\
    &=\left|\sum_{j=1}^{l}\ln\left(1+2 \frac{\curv_{j}'\prpslope_{j}'' w_{j}''-\curv_{j}'\prpslope_{j}'' w_{j}''}
      {\prpslope_{j}' w_{j}'(2\curv_{j}'' +\prpslope_{j}''
      w_{j}'')} \right)\right|\\ &\le C_1\sum_{j=1}^{l}
                                   \frac{|\curv_{j}'\prpslope_{j}''
                                   w_{j}''-\curv_{j}'\prpslope_{j}''
                                   w_{j}''|}{w_{j}'
                                   w_{j}''}
                                   \leq C_2 \sum_{j=1}^{l} \frac{d_\alpha(x_{j}', x_{j}'')}{w_j'} \\&\le
    \frac{C_3 \Np(x_0')}{w'} d_\alpha(x_{l}', x_{l}'')\leq
    C_4 d_\alpha(x', x'').
    \qedhere
  \end{align*}

\end{proof}
Let $n > 0$, $W\subset\csp$ be a mature unstable curve with the
property that $\cm^{-n} W$ is a mature unstable curve and let
$\tx\in W$ be a reference point on $W$.  Then we can define a
density $\rho_n$ on $W$ as follows:
\begin{align*}
  \rho_n(x)=\frac{\jacW\cm^{-n}(x)}{\jacW\cm^{-n}(\tx)} =   
  \prod_{j=1}^{n} \frac{\jacW\cm(\cm^{-j} x)}{\jacW\cm(\cm^{-j} \tx)}.
\end{align*}

\begin{lem}\label{LmBD}
 
  (a) Given $L > 0$, there is a constant $\tK > 0$ such that the
    following holds. Let $V$ be a mature admissible unstable curve so
    that $W = \cm^{n}V$ belongs to only one H-component and
    $|W|_{\adm} < L$. Then
    \begin{align*}
      \|\ln \rho_n(x)\|_{C^{1/12}(W)}\leq \tK.
    \end{align*}

  (b) Let $W$ be an unstable manifold (that is, $\cm ^{-n} W$ is an
    unstable curve for all $n$) with $|W|_{\adm} < L.$ Then $\rho_n$
    converges when $n\to\infty$ along a sequence of times such that
    $\cm^{-n} W\not\subset \reco$ to a limiting density
    $\rho_\infty$and $\ln\rho_\infty$ is H\"older continuous.
  
    \end{lem}
\begin{rmk}
  In this paper we will only use part (a) of the above lemma. We
  decided to include part (b) as well since the proofs of both items
  are similar and part (b) may be useful for studying statistical
  properties of Fermi--Ulam Models (cf.~\cite[Section 7]{ChM}).
\end{rmk}
\begin{rmk}
   In Remark \ref{rmk:uniform-bound-euclidean}
  we mentioned that the Euclidean length of
  unstable manifolds is uniformly bounded.   Such a bound is
  unavailable for the $\admName$-length, therefore we will not be able
  to drop the bounded $\admName$-length assumption in our
  discussion.
\end{rmk}
\begin{proof}
  The statement would easily follow from Lemma~\ref{LmEXpHolder} if
  $\cm$ were uniformly hyperbolic.  Since this is not the case, we
  need to follow a strategy similar to the argument presented in the
  proof of Lemma~\ref{l_OneItFHat}. we partition the interval
  $[1,\cdots, n]$ into blocks with good hyperbolicity properties.

  First of all, by Lemma~\ref{l_OneItFHat} there exists $C > 1$ so
  that for any $0\le m\le n$, $|\cm^{-m}W|\adm < CL$.  Moreover, since
  $\cm^{-m}W\cap \sing m = \emptyset$, we already observed that the
  function $x\mapsto\min\{m,\Np(x)\}$ is constant on $\cm^{-m}W$.  Let
  $n_0$ be the constant value of $\min\{n,\Np(x)\}$ on $V$, $n_1$ be
  the constant value of $\min\{n,\Np_{1}(x)\}-n_{0}$ and so on, until
  we obtain $n_{0},\cdots, n_{p} > 0$ so that $n_{0}+\cdots+n_{p} = n$
  and for any $0 < l < p$, $n_{0}+n_{1}+\cdots+n_{l} = \Np_{l}(x)$ for
  any $x\in V$.  We can thus rewrite:
  \begin{align*}
    \rho_{n}(x) = \prod_{j = 0}^{p-1}\frac
    {\jacW\cm^{n_{j}}(\cm^{-n+n_{0}+\cdots+n_{j-1}}x)}
    {\jacW\cm^{n_{j}}(\cm^{-n+n_{0}+\cdots+n_{j-1}}\tx)}
  \end{align*}

  Then we can write, for any $x',x''\in W$:
$$      \left|\ln \rho_n(x'')-\ln \rho_n(x')\right|$$
$$=\left|\sum_{j=0}^{p-1}
      \ln{\jacW\cm^{n_{j}}(\cm^{-n+n_{0}+\cdots+n_{j-1}}x'')} -
      \ln{\jacW\cm^{n_{j}}(\cm^{-n+n_{0}+\cdots+n_{j-1}}x')}
      \right| $$

  Then Lemma~\ref{LmEXpHolder} implies:
  \begin{align*}
    \left|\ln \rho_n(x'')-\ln \rho_n(x')\right|
    &\le \Const\sum_{j = 0}^{p-1}
      d\adm(\cm^{-n+n_{0}+\cdots+n_{j-1}}x',\cm^{-n+n_{0}+\cdots+n_{j-1}}x'')^{1/12}\\
    &\le \Const d\adm (x',x'')^{1/12}+\\
    &\phantom\le+\Const\sum_{j = 0}^{p-2} d\adm(\cmp^{-j}\cm^{-n_{p-1}}x',\cmp^{-j}\cm^{-n_{p-1}}x'')^{1/12}+\\
    &\phantom\le+\Const d\adm(\cm^{-n}x',\cm^{-n}x'')^{1/12}.
  \end{align*}
  Proposition~\ref{p_propertiesAdm} and Lemma~\ref{l_OneItFHat} 
  conclude the proof of part (a).
  
        To prove part (b) consider two
  time moments $n_1<n_2$ such that $\cm^{-n_2} W\not\subset\reco$ and
  $\cm^{-n_{1}}|_W = \cmp^{-l_{1}}\cm^{-n^{*}}$ with
  $\cm^{-n^{*}}W\subset \cspi$.  Then:
  \begin{align*}
    \left|\ln\rho_{n_2}(x)-\ln\rho_{n_1}(x)\right|
    &=\left|\ln\rho_{n_2-n_1} (\cmp^{-n_1} x)\right|\\
    &=\left|\ln\rho_{n_2-n_1} (\cmp^{-n_1} x)-\ln\rho_{n_2-n_1} (\cmp^{-n_1} \tx)\right|\\
    &\leq \tK d_\alpha(\cmp^{-n_1} x, \cmp^{-n_1} \tx)^{1/12}\leq C
      \theta^{l_1} (d_\alpha(x, \tx))^{1/12},
  \end{align*}
  where the first inequality relies on Corollary~\ref{CrUMLip}, the
  already established part (a) and the second inequality relies on
  Proposition~\ref{p_propertiesAdm}(c).
\end{proof}
The next bound immediately follows from Lemma \ref{LmBD}.
  \begin{cor}[Distortion bounds]\label{cor:distortion}
    Let $L > 0$; there exists $\Cdist > 0$ so that the
    following holds.  Let $V$ be a mature unstable \admiss{}
    curve, $W_{n}$ be an H-component of $\cm^n V$ so that
    $|W_{n}|\adm < L$  and $V_{n} = \cm^{-n} W_{n}$.  Then, for any
    measurable set $E\subset\csp$:
  \begin{align*}
    e^{-\Cdist|W_{n}|\adm^{1/12}}\frac{\Leb_{W_{n}}(E)}{\Leb_{W_{n}}(W_{n})}
    \le
    \frac{\Leb_{V_{n}}(\cm^{-n}E)}{\Leb_{V_{n}}(V_{n})\vphantom{\hat
    W_{n }}}\le
    e^{\Cdist|W_{n}|\adm^{1/12}} \frac{\Leb_{W_{n}}(E)}{\Leb_{W_{n}}(W_{n})},
  \end{align*}
  where $\Leb_{V}$ denotes Lebesgue measure on the curve $V$ with
  respect to the $\admName$-metric.
\end{cor}

\subsection{Holonomy map}\label{SSHolonomy}
A $C^1$-curve $W$ is called a (homogeneous) stable manifold if
$|\cm^n W|\to 0$ as $n\to\infty$ and for each $n,$ $\cm^n W$ is
contained in one homogeneity strip.  (Homogeneous) unstable manifolds
are defined similarly, with $\cm^{n}$ replaced by $\cm^{-n}$.  At this
point we do not know how often the points have stable and unstable
manifolds, this issue will be addressed in Section~\ref{SSInvMan}.
Below we discuss how the expansion of unstable curves changes when we
move along stable manifolds.  We denote by $\Ws(x)$ the maximal
homogenuous stable manifold passing through the point $x$.  Let
$W_1, W_2$ be two mature unstable curves.  Let
\begin{equation}
\label{DomainRangeHol}
  \Omega_j =\{x\in W_j: \Ws(x)\cap W_{3-j}\neq \emptyset\}.
\end{equation}
Define
\begin{equation}\label{DefHolonomy}
  \hol:\Omega_1\to\Omega_2 \text{ so that } \Ws(x)\cap W_2 = \{\hol(x)\}.
\end{equation}
Observe that $\hol$ commutes with $\cm$ (and
thus with $\cmp$).  We assume that $W_1$ and $W_2$ are close to
each other so that $d_\alpha(x, \hol x)\leq d$ for some small $d > 0$.
Define
\begin{align}\label{Jak}%
  J(x)& =\prod_{j=0}^\infty \frac{\jac{\cmp^{j}W_{2}}\cmp(\cmp^j \hol x)}{\jac{\cmp^{j}W_{1}}\cmp(\cmp^j x)}.
\end{align}
\begin{lem}\label{LmJak}\ %
  \begin{enumerate}
  \item The infinite product~\eqref{Jak} converges. In fact there are
    constants $C>0$, $\theta<1$ such that for any $n > 0$
    \begin{align*}
      \left|J(x)-\prod_{l=0}^{n-1} \frac{\jac{\cmp^{j}W_{2}}\cmp(\cmp^l \hol x)}{\jac{\cmp^{j}W_{1}}\cmp(\cmp^{l} x)}\right| \leq C \theta^n.
    \end{align*}
  \item For any $\breps>0$ there exists $ \brdelta>0$ such that if
    $x'\in W_1,$ $x'' = \hol{}x'\in W_2, $ $d(x', x'')\leq \brdelta$ and
    $|(\prpslope_0)'-(\prpslope_0)''|\leq \brdelta$ then
    \begin{align*}
      \left| \prod_{l=0}^{n-1} \frac{\jac{\cmp^{j}W_{2}}\cmp(\cmp^l x'')}
      {\jac{\cmp^{j}W_{1}}\cmp(\cmp^l x')}-1\right|\leq\breps.
    \end{align*}
  \end{enumerate}
\end{lem}
\begin{rmk}
  In this paper we will not use part (b) of this lemma, but the proof
  follows from similar arguments, and part (b) could be useful in future
  developments.
\end{rmk}
\begin{proof}
  Once again, in this proof we drop the superscript “$-$” from $\cB$
  for the ease of notation.
  \renewcommand{\prpslope}{\cB} 

  For $x'\in W_1$ and $l\ge0$, let us denote
  $x'_l = \cmp^l x'=\cm^{\Np_{l}(x')} x'$ and let $x''=\hol x'$.  With
  this notation we have
  \begin{align*}
    J(x') &=\prod_{l = 0}^{\infty} %
            \frac{\jac{\cmp^{j}W_{2}}\cmp(x_l'')}{\jac{\cmp^{j}W_{1}}\cmp(x'_l)}.
  \end{align*}
  Observe that since $x''\in \Ws(x')$, the points $x'_j$ and $x''_j$
  belong to the same cell $\cell^{-}$ for any $j \ge0 $. In particular
  $x'_j\in\reco$ if and only if $x''_j\in\reco$ (and likewise for
  $\backReco$) and $\Np_{l}(x'') = \Np_{l}(x')$. Let
  $m_{l} = \Np_{l}(x'') = \Np_{l}(x')$.

  Using~\eqref{DefH2} we can then write
  \begin{align*}
    \left|\ln J-\ln \prod_{l=0}^{n-1} \frac{\jac{\cmp^{j}W_{2}}\cmp(x_l'')}{\jac{\cmp^{j}W_{1}}\cmp(x'_l)}\right|
    =\left|\sum_{j=m_n}^{\infty} \left[\ln H(x_j'', x_{j+1}'')-\ln H(x_j', x_{j+1}')\right]\right|.
  \end{align*}
  Inspecting the proof of Lemma~\ref{LmEXpHolder} we obtain the
  following estimate \newcommand{\slopeDifference}{\Xi}
  \begin{align}\notag
    \left|\ln J-\ln \prod_{l=0}^{n-1} \frac{\jac{\cmp^{j}W_{2}}\cmp(x_l'')}{\jac{\cmp^{j}W_{1}}\cmp(x'_l)}\right|
    \label{UBJak}%
    \leq C \sum_{l=n}^{\infty} d_\alpha(x_{m_l}', x_{m_l}'')^{1/12}+
    C \sum_{l=n}^{\infty} \sum_{j=m_l}^{m_{l+1}-1} \slopeDifference_j,
  \end{align}
  where we defined
  \begin{align*}
    \slopeDifference_j %
    &= \begin{cases}
      \displaystyle
      \frac{|\prpslope_j'-\prpslope_j''|}{\min\{1,\prpslope_j''\}}
      &\textrm{if }x'_j\nin\reco, \\[5pt]%
      \displaystyle
      \left|\frac{1}{\prpslope_j'}-\frac{1}{\prpslope_j''}\right|
      &\textrm{otherwise}.
    \end{cases}
  \end{align*}
  Accordingly, we need good bounds on $\slopeDifference_j$.
  Such bounds will be obtained by different arguments depending on
  whether $x'_{j}\nin\reco$ (case A) or $x'_{j}\in\reco$ (case B).

  Let us first consider case A. Observe that since $x'_{j-1}$ and
  $x''_{j-1}$ lie on the same stable manifold, they belong to the same
  cell $\cell^{-}_\indCell$, and $\indCell\sim\tau_{j-1}'$ for large
  $\indCell$. Next, Lemma~\ref{LmPSlopeC0} and Corollary~\ref{CrHESqueeze}  tell us
  that $\prpslope_j''$ can be small only
  if $\indCell$ (and, hence, $\tau_{j-1}''$) is large
  and in this
  case $\prpslope_{j}''$ is of order $\indCell\inv$.
  Applying once more Lemma~\ref{LmPSlopeC0} we get
  $|\prpslope_{j}'-\prpslope_{j}''| < C\indCell^{-2}$ and thus
  \begin{align*}
    \frac{|\prpslope_j'-\prpslope_j''|}{\prpslope_j''}\le
    C\indCell|\prpslope_j'-\prpslope_j''|\le C|\prpslope_j'-\prpslope_j''|^{1/2}.
  \end{align*}
  Hence, regardless if $\prpslope_j$ is small or not,
  it suffices to obtain good bounds for $|\prpslope_j'-\prpslope_j''|$.

  Let $m_l\le j < m_{l+1}$ and let $\tj$ be a number close to
  $m_{(l/2)}$ such that $x_{\tj}'\not\in \reco$.  Set
  $\bar\theta = \Lambda\inv\in(0,1)$.  Since $x_j'\not\in\reco,$
  iterating the estimates of parts (a) and (b) of
  Lemma~\ref{LmB1-2step}, we get
  \begin{align}%
    \notag \left|\cB_j'-\cB_j''\right|%
    & \leq \frac{\left| \cB_{\tj}'-\cB_{\tj}''\right|}
      {\prod_{k = \tj}^{j-1}
      \left[ \Delta_{k}' \Delta_{k}'' \right]
      }
      +C \sum_{k = \tj}^{j-1} d_E(x_k', x_k'')
    \\ &\leq
         \frac{\left|\cB_{\tj}'-\cB_{\tj}''\right|}{\prod_{k=\tj}^{j-1}
         \left[\Delta_k' \Delta_k''\right]} +Cd_\alpha (x_{m_{l/2}}', x_{m_{l/2}}).
    \\ & \le \frac{\left|\cB_{\tj}'-\cB_{\tj}''\right|}{\prod_{k=\tj}^{j-1}
         \left[\Delta_k' \Delta_k''\right]} + C\bar\theta^{l/2} d_\alpha(x', x'').
         \label{StSlope} %
  \end{align}
  where in the second inequality we have invoked
  Corollary~\ref{cor_adm-dominate-stable} and in the last inequality we
  used uniform contraction of stable manifolds by $\cmp$ with respect
  to the $\admName$-metric (which follows from
  Proposition~\ref{p_propertiesAdm}(c). 

  Next, the proof of Proposition \ref{p_propertiesAdm} shows that the
  denominator in the first term of the right hand side
  of~\eqref{StSlope} is $O(\bar\theta^{-l/2})$.  On the other hand by
  Corollary~\ref{CrHESqueeze} (since $x'_{\tj}\nin\reco$), we gather
  \begin{align}
    \label{IntermTrivial}
    |\prpslope_{\tj}'-\prpslope_{\tj}''|=O(1).
  \end{align}
  Accordingly
  $|(\prpslope_{j})'-(\prpslope_{j})''|=O(\bar\theta^{l/2}).$ for
  $m_l \le j < m_{l+1}$.  Plugging this estimate into~\eqref{UBJak}
  and summing over $l\geq n$, we conclude the proof of part (a) in
  case A by
  choosing $\theta = \bar\theta^{1/24}$.

  Next consider case B.
  By~\eqref{B2Lip}, which holds in $\reco$, we have
  \begin{align*}
    \left|\frac{1}{\prpslope_j'}-\frac{1}{\prpslope_j''}\right|
    &\le C|\prpslope_{j-1}'-\prpslope_{j-1}''|+ C d\adm(x'_{j-1},x''_{j-1}).
  \end{align*}
  Since $x'_{j-1}\nin\reco$ we can apply the estimates of case A to
  control $\prpslope_{j-1}'-\prpslope_{j-1}''$ to conclude the proof
  of part (a) in case B.

  The proof of part (b) is similar except that, we replace \eqref{IntermTrivial} by a better
  estimate for $|\prpslope_{\tj}'-\prpslope_{\tj}''|$.  Namely, if $x'_0\nin\reco$, then
  \eqref{SlopeBack} gives
  \begin{equation*}
    |\prpslope_{\tj}'-\prpslope_{\tj}''|
    \leq \frac{\left|\cB_{0}'-\cB_{0}''\right|}{\prod_{k=0}^{\tj-1} \left[\Delta_k' \Delta_k''\right]}
    +C \sum_{l=0}^{\tj} d_\alpha (x_{m_l}', x_{m_l}'')  \leq \brC\brdelta.
  \end{equation*}
  If $x'_0\in\reco$ we obtain a similar bound by
  invoking~\eqref{SlopeBack} up to $j = 1$.

  Accordingly
  $|\prpslope_{j}'-\prpslope_{j}''|=O(\bar\theta^{l/2}\brdelta)$
  for $m_l < j < m_{l+1}$.  Plugging this estimate into~\eqref{UBJak}
  and summing for $l\geq 0$ we obtain part (b).
\end{proof}

\section{Expansion estimate}%
\label{sec:expansion-estimate}
In this section we prove an expansion estimate for unstable curves
which is used in the proof of the so-called Growth Lemma (see
Lemma~\ref{l_growth}).  The section is organized as follows. In
Section~\ref{SS-VirtualComplexity} we define the notion of \emph{regularity
  at infinity}, which appears in the statement of our Main Theorem and
will be used crucially in the proof of the expansion estimate. In
Section~\ref{SSEe} we state the expansion estimate as
Proposition~\ref{p_expansionEstimate}. The proof of this proposition
is divided in two lemmas, which are proved in the final three
subsections of this section.

\subsection{Complexity at infinity.} \label{SS-VirtualComplexity}%
Recall that Theorem~\ref{t_normalForm} states that for large values of
$w$, $\cmp$ is well approximated by the map $\nf_\Delta$ defined
by~\eqref{SawTooth}.  In order to obtain results about the complexity of
the map $\cmp$ near $\infty$, we thus proceed to study the complexity of
the map $\nf_\Delta$.  From now on, we will assume $\Delta$ to be
fixed given by \eqref{DefDelta}.

Recall the definition of fundamental domains $\hDom_n$ given in
Section~\ref{SSCMP}, and define, for any $k > 0$
\begin{equation}\label{DefItin}
  \hDom_{n_0,n_1,\cdots,n_{k-1}} = \bigcap_{j = 0}^{k-1}\cl(\nf^{-j}_{\Delta}\hDom_{n_j}).
\end{equation}
We say that a $k$-tuple $(n_0,n_1,\cdots,n_{k-1})$ is
\emph{$\Delta$-admissible} if
$\hDom_{n_0,n_1,\cdots,n_{k-1}}\ne\emptyset$ and if
$x\in\hDom_{n_0,n_1,\cdots,n_{k-1}}$ we say that
$(n_0,n_1,\cdots,n_{k-1})$ is \emph{a $k$-itinerary} of $x$.  We stress
the fact that the sets $\hDom_{n_0,n_1,\cdots,n_{k-1}}$ are not pairwise
disjoint (their boundaries might overlap), hence some points might have
more than one itinerary.  For $x\in \cl(\hDom_0)$ we define
$ \LC_k(\Delta,x)$ to be the number of possible $k$-itineraries of $x$
that begin with $n_0=0$.
\begin{rmk}
  Observe that $\LC_k(\Delta,x)$ is in general larger than the maximum
  number of singularity lines of order $k$ meeting at the point $x$ (a
  number usually referred to as \emph{complexity}).  In fact, for some
  exceptional values of $\Delta$ (e.g. $\Delta = -1$) we can find $x$ so
  that $\LC_k(\Delta,x) = 2^k$.  On the other hand, for any $\Delta$,
  the number of singularity lines meeting at any point is bounded  above
  by $2k$ (see~\cite[Proof of Theorem 2]{fum} and also~\cite{C3}).
\end{rmk}
We define the \emph{$k$-virtual complexity of $\Delta$ at infinity} as
\begin{align*}
  \LC_k(\Delta)=\max_{x\in \cl(\hDom_0)} \LC_k(\Delta,x).
\end{align*}
\begin{rmk}
  The number $\LC_k(\Delta)$ is crucial in our analysis since it
  controls the number of components in which an arbitrarily small curve
  can be cut not just by $\nf$ but an arbitrarily small perturbation of
  $\nf$.  See Figure~\ref{f_virtualComplexity}: both panes show a
  neighborhood of the point $(1/2,1/2)$.  The left and right pane show
  the singularity portrait (up to $k = 5$ iterates) of
  $\nf_{\Delta = -1}$ and $\nf_{\Delta = -(1+\eps)}$ respectively.  As
  $\eps\to 0$ the nearly parallel lines shown in the right pane slide
  and coalesce at the center.  Observe that the complexity of the center
  in the left pane is $2k$, the complexity of any point in the right
  pane is bounded by $3$, but any short unstable curve passing
  sufficiently near the center is cut by singularities in an exponential
  (in $k$) number of curves provided that $\eps$ is sufficiently small.
  The $k$-\emph{virtual complexity} $\LC_k(\Delta)$ indeed bounds the
  number of such curves. On the other hand, since each point on the orbit of $x$ belongs to at most
  two fundamental domains, it follows that
\begin{equation}
\label{VirtCompTrivial}
\LC_k(\Delta)\leq 2^k.
\end{equation}

\end{rmk}
\begin{figure}[!hb]
  \centering
  \hspace{\stretch{1}}\includegraphics[width = 4cm]{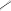}\hspace{\stretch{1}}
  \includegraphics[width = 4cm]{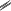}\hspace{\stretch{1}}
  \caption{Comparison of virtual complexity and standard complexity}\label{f_virtualComplexity}
\end{figure}

\begin{mydef}
  A Fermi--Ulam model is  \emph{regular at infinity} if
  \begin{align*}
 \limsup_{k\to\infty} \frac{\LC_k(\Delta)}{\Lambda_\Delta^k}=0
  \end{align*}
  where $\Lambda_\Delta$ is the expansion of the limiting map $\nf_\Delta$ defined by~\eqref{DefLambdaDelta}.

  A model is \emph{superregular at infinity} if there exists a
  constant $C$ so that for any $k\in\bN$ we have
  $\LC_k(\Delta)\le C$.
\end{mydef}

\begin{rmk}
\label{RkDeltaHalf}
We will show in Appendix~\ref{AppRegInf} that for all except possibly
countably many $\Delta$, the map $\nf_\Delta$ is superregular at
infinity.  However, the result of Appendix \ref{AppRegInf} does not
make it easy to check that a given value of $\Delta$ is regular.  On
the other hand~\eqref{VirtCompTrivial} shows that $\nf_\Delta$ is
regular at infinity provided that $\Lambda_\Delta>2$, that is, if
$|\Delta|>\frac{1}{2}$ (see~\eqref{SawTooth}).
\end{rmk}

Recall that the involution defined in Section~\ref{SSInvolution}
conjugates $\cm^{-1}$ to the Poincare map of the time reversed
Fermi--Ulam Model corresponding to $\bar{\ell}(r)=\ell(1-r)$.  Note that the
parameter $\Delta$ defined by~\eqref{DefDelta} is the same for $\ell$
and $\bar{\ell}$.  In particular, the Fermi--Ulam Model is regular at infinity
if and only if the reversed model is regular at infinity.  We
conclude that all results of this section formulated for unstable
curves of $\cm$ are valid also for stable curves of $\cm$ (that are
unstable curves of $\cm\inv$).

\subsection{Expansion estimate}
\label{SSEe}
In order to properly formulate the main result of this section we
need some definitions.  Let $W$ be an unstable curve; then
$\cm W$ is consists of (at most) countable union of connected components.
Any such component may in principle be further cut by
secondary singularities in an (at most) countable number of shorter
curves which we call \emph{\Hcomps{}}.  The same can be said for the
induced map $\cmp$.

We denote by $\{W_{i,n}\}_{i\in\bN}$ (\resp
$\{\hat W_{i,n}\}_{i\in\bN}$) the \Hcomps{} of $\cm^n W$ (\resp
$\cmp^n W$).  Given an \Hcomp{} $\hat W_{i,n}$ of $\cmp^n W$, we can
uniquely define $\Np_{i,n} > 0$ so that
\begin{align*}
  \cmp^n|_{\cmp^{-n} \hat W_{i,n}} = \cm^{\Np_{i,n}}|_{\cmp^{-n} \hat W_{i,n}}.
\end{align*}
Finally, we denote by $\Lambda_{i,n}$ (\resp $\hat\Lambda_{i,n}$) the
minimum expansion, with respect to the $\admName$-metric, of $\cm^n$
(\resp $\cmp^n$) on $\cm^{-n} W_{i,n}$ (\resp
$\cmp^{-n}\hat W_{i,n}$).  Given an unstable curve $W\subset \csp$
(\resp $W\subset\cspi$), and $n > 0$, we define:
\begin{align*}
  \cL_{n}(W) &= \sum_{i}\frac1{\Lambda_{i,n}}&
  \hat \cL_{n}(W) &= \sum_{i}\frac1{\hat\Lambda_{i,n}}.
\end{align*}
Then we let
\begin{align*}
  \cL_{n}(\delta) &= \sup_{W:|W|\adm\le\delta} \cL_{n}(W)&
  \hat \cL_{n}(\delta) &= \sup_{W:|W|\adm\le\delta} \hat \cL_{n}(W)\\
  \cL_{n} &= \liminf_{\delta\to 0}\cL_{n}(\delta)&
  \hat \cL_{n} &= \liminf_{\delta\to 0} \hat \cL_{n}(\delta)
\end{align*}
It follows from the definition that $\cL_{n}$ (\resp
$\hat\cL_{n}$) is a sub-multiplicative sequence, \ie
\begin{align}\label{eq:cL-submultiplicative}
  \cL_{n+m}&\le\cL_{n}\cL_{m} & \hat\cL_{n+m}&\le\hat\cL_{n}\hat\cL_{m}.
\end{align}

\begin{prp}[Expansion estimate]\label{p_expansionEstimate}
There exists $C > 0$ such that
    \begin{align}\label{e_expansionEstimate-trivial}
      \hat \cL_{1} < C.
    \end{align}
  Moreover, if the Fermi--Ulam model is regular at infinity then there
  exists $\bar n > 0$ so that
  \begin{align}\label{e_expansionEstimate}
    \hat \cL_{\bar n} < 1,
  \end{align}
  and there exists $C' > 0$ so that for any $n > 0$ we have
  $\hat \cL_{n} < C'$.
\end{prp}
The rest of this section is devoted to the proof of Proposition~\ref{p_expansionEstimate}.
We will follow the strategy
described in~\cite{jmogy}.
Recall the definition of the homogeneity
strips $\homo_k$ given in Section~\ref{s_homo}.

\begin{mydef}
  Let $W$ be an unstable curve. An \Hcomp{} $W_{i,n}$ (\resp
  $\hat W_{i,n}$) of $\cm^n W$ (\resp $\cmp^n W$) is said to be
  \emph{regular} if for any $0 \le q < n$ (\resp $0\le q < \Np_{i,n}$)
  we have that $\cm^{-q}W_{i,n}\subset \homo_0$ (\resp
  $\cm^{-q}\hat W_{i,n}\subset \homo_0$) and \emph{nearly grazing}
  otherwise.
\end{mydef}
Observe that the notion of regularity depends on the choice of the
constant $k_0$ introduced in Section~\ref{s_homo}; in particular, if
$k_0$ increases, the number of regular \Hcomps{} also increases.
\begin{lem}\label{l_atMostOne1}
  Let $W$ be a u-curve and $N > 0$. Then any connected component of
  $\cm^NW$ (\resp $\cmp^NW$) contains at most one regular \Hcomp{}.
\end{lem}
\begin{proof}
  Let us first prove the statement for connected components of $\cm^NW.$
  We give a proof by induction on $N$.  The statement is true
  if $N = 1$.  Indeed, the intersection of any connected u-curve with $\homo_0$
  is necessarily connected, hence out of the \Hcomps{} in which a
  connected component of $\cm W$ is cut by secondary singularities, at
  most one can be regular.

  Assume now by induction that the statement holds for $N$, and let
  $\tilde W'$ be a connected component of $\cm^{N+1}W$.  Let $\tilde W$
  be the connected component of $\cm^N W$ which contains
  $\cm\inv \tilde W'$.  By inductive hypothesis, either $\tilde W$
  contains no regular \Hcomp{} (and thus so does $\tilde W'$ and the
  statement holds), or it contains only one regular \Hcomp{}, which we
  denote by $W^*\subset\tilde W$.  Then any regular \Hcomp{} of
  $\tilde W'$ has to be contained in the connected u-curve
  $\cm W^*\cap \tilde W'$.  Since at most one of the \Hcomps{} of this
  curve can be contained in $\homo_0$ (and thus can be regular), we
  conclude the proof for $N+1$.

  Finally, the statement for $\cmp^NW$ follows from the statement for
  $\cm^NW$. Namely, suppose that for some $N,$ $\cmp^N W$ contains two
  regular H-components. Denote their preimages by $W'$ and $W''$.  Then
  $\cmp^N W'=\cm^{N'} W'$ and $\cmp^N W''=\cm^{N''} W''$.  Suppose
  without loss of generality that $N'\leq N''$ then $\cm^{N'} W$ has two
  regular H-components giving the contradiction.  \ignore{ Finally,
    because any connected component of $\cmp^NW$ is a connected
    component of $\cm^{\Np}W$ for some $\Np\ge N$ (which possibly
    depends on the component), the statement for $\cmp^NW$ follows from
    the statement for $\cm^NW$.}
\end{proof}
\begin{mydef}\label{d_regularComplexity}
  Given an unstable curve $W$ and $n > 0$, we define the \emph{regular
    $n$-complexity} of $W$ (\resp the \emph{induced regular
    $n$-complexity} of $W$), denoted by $\kreg_n(W)$ (\resp
  $\hkreg_n(W)$) to be the number of regular \Hcomps{} of $\cm^nW$ (\resp
  $\cmp^n W$).  If $n = 0$ we set conventionally
  $\kreg_0(W) = \hkreg_0(W) = 1$.  Finally, define
  \begin{align*}
    \kreg_n(\delta) &= \sup_{W:|W|\adm\le\delta}\kreg_n(W),&
    \hkreg_n(\delta) &= \sup_{W:|W|\adm\le\delta}\hkreg_n(W);\\
    \kreg_n &= \liminf_{\delta\to0}\kreg_n(\delta),&
    \hkreg_n &= \liminf_{\delta\to0}\hkreg_n(\delta).
  \end{align*}
\end{mydef}
\begin{rmk}
  Given an unstable curve $W$, recall the standard definition of
  $n$-\emph{complexity} of $W$ as the number of connected components
  of $\cm^nW$.  Lemma~\ref{l_atMostOne1} implies that regular
  complexity does not exceed standard complexity.  Observe moreover
  that while standard complexity is non-{}decreasing in $n$, regular
  complexity is not necessarily so (\eg the image of a regular
  component of $\cm^nW$ may contain no regular component).  Finally,
  all the above quantities are non-decreasing in $k_0$.
\end{rmk}
For future use, we note that
Lemmata~\ref{l_largeCell}\ref{i_almostFinite} and~\ref{l_atMostOne1}
imply that, provided $k_0$ is sufficiently large and $\delta$ is
sufficiently small, the following trivial estimate holds:
\begin{align}\label{e_trivialEstimateOnRegularComplexity}
  \kreg_n(\delta) &\le 3^n.
\end{align}

Let us now define $\zreg$, $\hzreg$ (\resp $\cL^{*}, \hat \cL^{*}$) as we
did above for $\cL$ and $\hat \cL$, but summing only on regular (\resp nearly grazing components). For instance:
\begin{align*}
  \cL^{*}_{n} = \liminf_{\delta\to0}\sup_{W:|W|\adm\le\delta}{\sum_i}^*\frac1{\Lambda_{i,n}},
\end{align*}
where $\sum^{*}$ denotes that the sum is restricted only to nearly
grazing components.  The following lemmata will allow us to prove
Proposition~\ref{p_expansionEstimate}.
\begin{lem}[Control for nearly grazing components]\label{l_nearlyGrazing}
  For any $N > 0$ and $\eps > 0$, we can choose $k_0$ large enough in
  the definition of homogeneity strips so that
  $\cL_{n}^{*} < \eps\text{ for any $0 < n\le N$}$.
\end{lem}
\begin{lem}[Bound on regular complexity]\label{l_regularComplexityBound}
  If the Fermi--Ulam model is regular at infinity, there exists $\brn$
  such that  if $k_0$ in the definition of homogeneity
  strips is large enough and $\delta$ is sufficiently small then
  \begin{equation}  \label{EqGlobCompl}
    \hzreg_{\brn}(\delta)
    \leq \frac{1}{2}.
  \end{equation}
\end{lem}
The proofs of the two above lemmata are independent of each other.
Lemma~\ref{l_nearlyGrazing} is proved in Section~\ref{SS-NG}, whereas
the proof of Lemma~\ref{l_regularComplexityBound} occupies
Sections~\ref{ss_regularComplexity} and~\ref{SSPointwise}.

Observe that Lemma~\ref{l_nearlyGrazing} allows to prove that
\begin{align}\label{eq:cl1-bounded}
  \cL_{1} < \infty.
\end{align}
In fact, we have $\cL_{1} = \zreg_{1}+ \cL^{*}_{1}$; by
Lemma~\ref{l_nearlyGrazing}, the second term can be made as small as
needed, and by Lemma~\ref{l_largeCell}\ref{i_almostFinite}, provided
that $|W|\adm$ is small enough, the first term is at most
$3\cdot \underline \Lambda\inv$, where $\underline\Lambda$ is a lower
bound for~\eqref{e_expansionAdm}.

Combining
these two results yields\footnote{ The proof given below is similar to
  the one used in~\cite[Main Theorem]{jmogy}.} the proof of the
Expansion Estimate:
\begin{proof}[Proof of Proposition~\ref{p_expansionEstimate}]
  Let $W$ be an unstable curve so that $|W|\adm < \delta$ with
  $\delta > 0$ sufficiently small. Recall that $\Lambda$ is the
  minimal expansion of $\cmp$ in the $\alpha$-metric
  (see~\eqref{e_uniformHyperbolicity}).  Observe that by definition,
  for any $n > 0$
  \begin{align*}
    \hat \cL_{n}(W) = \hzreg_{n}(W)+\hat \cL^{*}_{n}(W).
  \end{align*}


  In view of Lemma \ref{l_regularComplexityBound} it is enough to show
  that, if $\delta$ is sufficiently small, we have
  $\hat \cL^{*}_{n} < 1/2$ for all $0 < n\le\bar n$ where $\bar n$ is
  from Lemma \ref{l_regularComplexityBound}. By
  Proposition~\ref{p_largeEnergiesFUM} there exists
  $\bar w = \bar w(\bar n)$ so that, if $W\subset\{w > \bar w\}$, then
  $\cmp^{n}W$ has no nearly grazing \Hcomps{} for any
  $0 < n\le\bar n$.  Thus, by~\eqref{e_weakBoundOnNp} and
  Proposition~\ref{p_largeEnergiesFUM}\ref{i_boundOnEndpoint}, we
  conclude that there exists a uniform
  $\bar n'\sim \bar n(\bar w+\bar n)$ so that $\Np_{i,n}\le\bar n'$
  for any nearly grazing \Hcomp{} $\hat W_{i, n}$.  Thus
  \begin{align*}
    {\sum_i}^*\frac1{\hat\Lambda_{i,n}} =%
    \sum_{k = 1}^{\bar n'}{\sum_{i:\Np_{i, n} = k}}^* %
    \frac1{\hat\Lambda_{i, n}} \le%
    \sum_{k = 1}^{\bar n'}{\sum_j}^*\frac1{\Lambda_{j,k}}.
  \end{align*}
  Hence, it is sufficient to apply Lemma~\ref{l_nearlyGrazing} with
  $N = \bar n'$ and $\eps = 1/(4\bar n')$ to obtain
  both~\eqref{e_expansionEstimate-trivial} (with $C = K+1/2$)
  and~\eqref{e_expansionEstimate}.

  The uniform bound on $\hat \cL_{n}$ for all $n$ follows since
  $\hat\cL_{m+n}\le\hat\cL_{m}\hat\cL_{n}$.  Namely, let $n = p\bar n+r$,
  where $0\le r < \bar n$. Then
  \begin{align*}
    \hat\cL_{n}\le\hat\cL_{\bar n}^{p}\cdot
    \hat\cL_{1}^{r}&\le C^{\bar n}.\qedhere
  \end{align*}
\end{proof}
\ignore{We conclude this subsection with an auxiliary result, that
  will be useful in Section~\ref{sec:growth-lemma}.
\begin{lem}\label{lem:bounded-complexity}
  Let $W$ be a sufficiently short u-curve and define the number
  $\Np(W) = \sup_{x\in W\setminus\singt{+}}\Np(x)$.  Then, there
  exists $C > 0$ so that $\cL_{n}(W) < C$ for all $n\le \Np(W)$.
\end{lem}

\todo[inline]{Where this Lemma is used?}

\begin{proof}
  First observe that by~\eqref{e_trivialEstimateOnRegularComplexity},
   $\cL_{1}^{\text{reg}} < 3$; then
  Lemma~\ref{l_nearlyGrazing} implies that $\cL_{1}(W) < 4$ provided
  that $W$ is sufficiently short. Moreover, if $W$ is sufficiently
  short and $W\cap\{w\le w^{*}\}$ for some $w^{*}$ large enough, we
  conclude that $\Np(W) < N^{*} = \Const w^{*}$.  In this case the
  lemma follows by choosing $C = 4^{N^{*}}$.  On the other hand, if
  $W\cap\{w > w^{*}\}$, choosing $W$ sufficiently short,
  Lemma~\ref{l_largeCell}\ref{i_almostFinite-large} guarantees that
  $W$ intersects at most two of the sets $\cEs_{n}$ .  This in turn
  implies that $\cm^{n}W$ has at most two connected components for
  $0 \le n\le \Np(W)$; if $w^{*}$ is sufficiently large, we can
  guarantee that such components are H-components; this implies that
  $\cL_{n} \le 2$ for all $0\le n\le \Np(W)$, and concludes the proof.
\end{proof}
}  \subsection{Control for nearly grazing components}
\label{SS-NG}
\begin{proof}[Proof of Lemma~\ref{l_nearlyGrazing}]
  We prove the lemma by induction on $N$.  Let us first assume $N = 1$
  and let $\tilde W'$ be a connected component (rather than an \Hcomp{})
  of $\cm W$. If we restrict to $H$-components contained in $\tilde W'$,
  we obtain
  \begin{align*}
    {\sum_i}^*\frac{1}{\Lambda_{i,1}}                %
    &\le\sum_{k\ge k_0}\const k_0^{-2} = \const k_0\inv. %
  \end{align*}
  Were the number of connected components $\tilde W'_i$ of $\cm W$
  \emph{uniformly bounded}, our claim would thus be proved.  As we
  already observed, this is not the case.  Fix $n_*$
  sufficiently large. Lemma~\ref{l_largeCell}(a) ensures that, except
  for finitely many (i.e. $3$) connected components of $\cm W$, all
  the others will intersect cells $\cell^-_\indCell$ with
  $\indCell\ge n_*$.  Moreover, by Lemma~\ref{l_largeCell}(b),
  $\cell^-_\indCell$ will intersect only homogeneity strips $\homo_k$
  for $k > \const \indCell^{1/4}$.  Denote by $W_{[\indCell,k],1}$ the
  $H$-component of $\cm W$ such that
  $W_{[\indCell,k],1}\subset\homo_k\cap\cell^-_\indCell$.  Then
  using~\eqref{e_expansionAuxMetricB}, estimating the flight time by
  $\indCell$ and the relative velocity by $k^{-2}$ we conclude that
  the expansion of $W_{[\indCell,k],1}$ satisfies
  \begin{align*}
    \Lambda_{[\indCell,k],1} > \const \indCell k^2.
  \end{align*}
  We thus gather that, if $n_*$ is sufficiently large and $W$ is
  sufficiently short, then
  \begin{align*}
    {\sum_{i}}^*\frac1{\Lambda_{i,1}} %
    &\le \const k_0\inv + \sum_{\indCell\ge n_*}
      \sum_{k\ge\const \indCell^{1/4}}\frac1{\Lambda_{[\indCell,k],1}}\\%
    &\le \const k_0\inv + \sum_{\indCell\ge n_*} \const \indCell^{-5/4}\le\const (k_0\inv + n_*^{-1/4}).
  \end{align*}
  The last expression can then be made as small as needed by
  choosing $k_0$ and $n_*$ sufficiently large.  We thus obtained our base
  step: for any $\eps > 0$, if $k_0$ is sufficiently large we have
  \begin{align*}
    \cL_{1} < \eps.
  \end{align*}
  Using the above notation, we assume by inductive hypothesis that for
  any $\eps > 0$ we can choose $k_0$ large enough in the definition of
  homogeneity strips so that $\cL_{n}^{*} < \eps$ and we want to show
  that $\cL_{n+1}^{*} < \eps$.  In order to prove the inductive step,
  observe that for any u-curve $W$, we have the following inductive
  relation summing over the H-components $W_{i,1}$ of $\cm W$:
  \begin{align}\label{e_inductivecL}
    \cL_{n+1}^*(W)              %
    &\le \sum_{i:\:W_{i,1}\text{ is reg.}}\frac1{\Lambda_{i,1}}\cL^*_{n}(W_{i,1}) +
      {\sum_{i}}^*\frac1{\Lambda_{i,1}}\cL_{n}(W_{i,1}).
  \end{align}

  By Proposition~\ref{p_propertiesAdm}\ref{i_expansionAdm},
  there exists $0 < \LambdaMin < 1$ so that
  $\Lambda_{i,n} > \LambdaMin^n$ for any $n > 0$.  Thus, for any
  $\delta$ sufficiently
  small,~\eqref{e_trivialEstimateOnRegularComplexity} and our inductive
  assumption imply the following rough bound on $\cL_{n}(\delta)$:
  \begin{align} \label{RoughL} \cL_n(\delta) \le
    \frac{3^n}{\LambdaMin^n}+\cL_n^*(\delta)\le
    2\frac{3^n}{\LambdaMin^n}.
  \end{align}
  Using~\eqref{e_globalExpansion} we get that if $|W|\adm < \delta$,
  then $|W_{i,1}|\adm < C_*\delta^{1/4}$. Hence by~\eqref{e_inductivecL}
  and using once again~\eqref{e_trivialEstimateOnRegularComplexity}, if
  $|W|\adm < \delta$:
  \begin{align*}
    \cL_{n+1}^*(\delta)              %
    &\le \sum_{i:\:W_{i,1}\text{ is reg.}}\frac1{\Lambda_{i,1}}\cL^*_{n}(C_*\delta^{1/4}) +
      {\sum_{i}}^*\frac1{\Lambda_{i,1}}\cL_{n}(C_*\delta^{1/4})\\
    &\le \frac{3}{\LambdaMin}\cL_{n}^*(C_*\delta^{1/4})+\cL^*_1(W)\cL_{n}(C_*\delta^{1/4}).
  \end{align*}
  Using the inductive hypothesis and~\eqref{RoughL}, taking
  $\liminf_{\delta\to0}$ we gather that $\cL_{n+1}^* < \Const \eps$,
  which concludes the proof of the inductive step.
\end{proof}
\subsection{Control on regular complexity}\label{ss_regularComplexity}
In this section we prove that we can bound the induced regular
complexity $\hkreg_n$, needed to prove Lemma \ref{l_regularComplexityBound},
by means of two other quantities.  One is the virtual complexity introduced
in Subsection~\ref{SS-VirtualComplexity} and the other is 
the pointwise complexity which we now proceed to define.

Let $x\in\csp$ and let $Q_n$ be a connected component of
$\csp\setminus\sing n$ so that $\cl Q_n\ni x$.  We say that $Q_n$ is
\emph{$n$-regular at $x$} if
\begin{align*}
  \lim_{Q_n\ni x'\to x}\cm^{l}x'\in \cl\homo_0\text{ for all $0 < l\le n$};
\end{align*}
otherwise $Q_n$ is said to be \emph{nearly grazing at $x$}.

\begin{mydef}\label{d_regular pointwise complexity}
  Given a point $x\in\csp$ and $n > 0$, we define the \emph{$n$-regular
    complexity at $x$}, denoted with $\ckreg_n(x)$, to be the number of
  components of $\csp\setminus\sing n$ whose closure contain $x$ and
  that are $n$-regular at $x$.  We then define:
  \begin{align*}
    \ckreg_n = \sup_{x\in\csp}\ckreg_n(x).
  \end{align*}%
\end{mydef}
Recall that $\singp{n}$ denotes the singularity set of $\cmp^{n}$ and
let $\hat Q_n$ be a connected component of $\cspi\setminus\singp{n}$.
By the discussion prior to
Lemma~\ref{LmFRet} we conclude that  there exists $\Np_n(\cc_n)$ so that
for any $x\in\hat Q_n$ we have $\cmp^n(x) = \cm^{\Np_n(\cc_n)}(x)$.
Suppose now that $x\in\cl\hat Q_n$; we say that $\hat Q_n$ is
$n$-regular at $x$ if
\begin{align*}
  \lim_{\hat Q_n\ni x'\to x}\cm^{l}x'\in \cl\homo_0\text{ for all $0 < l\le \Np_n(\cc_n)$}.
\end{align*}
Define $\ckregp_n(x)$ to be the number of connected components of
$\cspi\setminus\singp{n}$ whose closure contains $x$ and which are
$n$-regular at $x$.  Set
\begin{align}\label{e_defInducedRegularComplexity}
  \ckregp_n = \sup_{x\in\cspi}\ckregp_n(x).
\end{align}%

If the phase space $\csp$ were compact (as it is in the case of
dispersing billiards) then $\kregp_n$ (see
Definition~\ref{d_regularComplexity}) and $\ckregp_n$ would coincide
(see case (a) in the proof of Lemma \ref{LmCInf-EE} below).  Since 
our the phase space is not compact, we need a more careful
analysis, which we provide below.

\begin{lem}  \label{LmCInf-EE}
  Suppose that for some $\bar n$ we have
  \begin{equation}\label{e_boundAssumption}
    \ckregp_{\bar n}< \frac{\Lambda^{\bar n}}{2} \text{ and }
    \LC_{\bar n}(\Delta)\leq \frac{\Lambda_\Delta^{\bar n}}{4\hat C}
 \end{equation}
 where $\Lambda$ is the minimal expansion in $\alpha$-metric,
 $\Lambda_\Delta$ is the expansion of the limiting map, defined
 by~\eqref{DefLambdaDelta}, and $ \hat C$ is from
 Corollary~\ref{CrExpAtInf}, then~\eqref{EqGlobCompl} holds.
\end{lem}
\begin{proof}
  Assume by contradiction that~\eqref{EqGlobCompl} were false. 
  Then there
  would exist a sequence of unstable curves $(\Wseq{m})_{m}$ so that
  $|\Wseq{m}|\adm\to 0$ as $m\to\infty$ and
  $\hzreg_{n}(\Wseq{m}) > \frac{1}{2}$ for any $m > 0$.  Observe that
    \begin{align}
  \label{LGeneralBound}
    \hzreg_{n}(W)\le \frac{\hkreg_{n}(W)}{\min_{i} \Lambda_{i, n}}.
  \end{align}
  Pick arbitrary points $x^{(m)}\in \Wseq{m}$.  After possibly passing
  to a subsequence we can assume that one of the two possibilities
  below hold.
  \begin{enumerate}
  \item the sequence $x^{(m)}$ is bounded;
  \item the sequence $x^{(m)}$ tends to infinity.
  \end{enumerate}

  We analyze these two cases separately.

  \textbf{Case (a).}
  In this case we estimate the denominator of \eqref{LGeneralBound}
  by $\Lambda^n$ obtaining
  \begin{align}\label{eq:contradiction-small-w}
    \hkreg_{n}(W) > \frac{\Lambda^n}{2}.
  \end{align}

  Since the sequence $x^{(m)}$ is bounded,
  combining~\eqref{e_trivialEstimateOnRegularComplexity}
  with~\eqref{e_weakBoundOnNp} we gather that
  $(\hkreg_{\bar n}(\Wseq{m}))_m$ is also a bounded sequence.  We can
  therefore assume (possibly passing to a subsequence) that
  $\hkreg_{\bar n}(\Wseq{m}) = \Kn $ for all $m$.

  As noted earlier, the set $\cspi\setminus\singp{\bar n}$ is the
  union of a countable number of connected components.
  By Lemmata~\ref{l_singularityStructure} and~\ref{LmFRet}, to each such
  component\footnote{We drop the subscript $\bar n$ as this is fixed
    once and for all and will not cause any confusion} $\cc$ we can
  uniquely associate a $\Np(\cc)$-tuple
  \begin{align*}
    \indTuple(\cc) = (\indCell_0,\indCell_1,\cdots,\indCell_{\Np(\cc)-1})
    \textrm{ where } \indCell_i\in\{\recoPrivate,0,1,\cdots\}
  \end{align*}
  so that
  \begin{align*}
    \cc = \cspi \cap  \bigcap_{l = 0}^{\Np(\cc)-1}\cm^{-l}\cell^+_{\indCell_l}.
  \end{align*}
  For $0\le i < \hkreg_{\bar n}$, denote with $\Wseq{m}_i$ the
  preimage under $\cmp^{\brn}$ of the $i$-th regular \Hcomp{} of
  $\cmp^{\bar n}\Wseq m$.  Let $\ccseq{m}_i$ be so that
  $\Wseq{m}_i\subset\ccseq{m}_i$.  By
  Lemma~\ref{l_atMostOne1} 
  $\ccseq{m}_i\ne\ccseq{m}_j$ if $i\ne j$.

    Since $\cmp^{\bar n}\Wseq{m}_i$ is regular, we must
    have $\indCell_l(\cc) \in\{\recoPrivate,0,1,\cdots,\redCell\}$ for
    all $0\le l < \Np(\cc)$ and some $\redCell > 0$.  Since the
    sequence $(\Wseq{m})_m$ is bounded, we conclude
    by~\eqref{e_weakBoundOnNp} that $(\Np(\ccseq{m}_i))_{m}$ is also a
    bounded sequence.

  Since there are only finitely many $\cc$'s which satisfy these
  requirements, we can always assume (extracting a subsequence if
  necessary) that $\ccseq{m}_i = \ccseq{m'}_i$ for any $m,m'$; for ease
  of notation we will denote such connected components simply by
  $\cc_i$.

  Let us now choose arbitrarily points
  $\xseq{m}_i\in\Wseq{m}_i\subset\cc_i$.  Since $(\xseq{m}_i)_m$ is a
  bounded sequence, we can assume (extracting a subsequence if
  necessary) that $\xseq{m}_i\to\bar x_i$ for some
  $\bar x_i\in\cl\cc_i$.  On the other hand, since $|\Wseq{m}|\adm\to0$
  and $|\cdot|\adm$ is equivalent to the Euclidean norm if $w$ is
  bounded, it must be that $\bar x_i = \bar x_j$ for every
  $0\le i,j < \Kn $.  We call this common limit point
  $\bar x$.
  Since $\cmp^{\bar n}\Wseq{m}_i$ is regular, we conclude that each of
  the $\cc_i$'s is regular at $\bar x$.  We conclude that
  $\Kn \leq \ckregp_{\bar n}(\bar x)\leq \ckregp_{\bar n}$, which
  contradicts~\eqref{eq:contradiction-small-w} by the first estimate
  in~\eqref{e_boundAssumption}.

    \textbf{Case (b).}  In this case we estimate the denominator of
    \eqref{LGeneralBound} using 
    Corollary
    \ref{CrExpAtInf} obtaining
    \begin{align*}
      \hzreg_{n}(W) \leq \frac{\hkreg_{n}(W)}{\hat{C} \Lambda_\Delta^n}.
    \end{align*}

  Observe that if we show
  $\hkreg_{\bar n}(W^{(m)}) \leq 2\LC_{\bar n}(\Delta)$ for all but finitely many
  $m$'s, then~\eqref{EqGlobCompl} follows from the second estimate
  in~\eqref{e_boundAssumption}.  We proceed by contradiction and assume
  (possibly extracting a subsequence) that $|\Wseq{m}|\adm\to 0$,
  $\DS \min_{\Wseq{m}} w\to\infty$, but
  \begin{align*}
    \hkreg_{\bar n}(\Wseq{m})\geq 2\LC_{\bar n}(\Delta)+1 \text{ for
    all $m > 0$}.
  \end{align*}
  Recall the definition (see \eqref{e_fundamentalDomains}) of the
  fundamental domains $\Dom_n = \{x\in\cspi\st \Np(x) = n\}$.  Similarly
  to~\eqref{DefItin}, we define, for $k > 0$:
  \begin{align*}
    \Dom_{n_0,n_1,\cdots,n_{k-1}} &= \bigcap_{j = 0}^{k-1}\cmp^{-j}\Dom_{n_j}.
  \end{align*}
  A $k$-tuple $(n_0,n_1,\cdots,n_{k-1})$ is said to be
  \emph{$\cmp$-admissible} if
  $\Dom_{n_0,n_1,\cdots,n_{k-1}}\ne\emptyset$. If
  $x\in\Dom_{n_0,n_1,\cdots,n_{k-1}}$, we say that
  $(n_0,n_1,\cdots,n_{k-1})$ is the\footnote{ In
    Section~\ref{SS-VirtualComplexity} we gave similar definitions for
    domains given in terms of the normal form.  It must be noted that
    here we do not take the closure in the definition of the
    $\Dom_{n_0,n_1,\cdots,n_{k-1}}$'s, hence we can define \emph{the}
    itinerary (as opposed as \emph{an} itinerary) of a point $x$.  The
    reason for this mismatch is that the $\Dom_n$'s are defined
    dynamically (as opposed to the geometric definition of $\hDom_n$)
    , and thus their boundary carry some dynamical information which
    we want to preserve.} \emph{$k$-itinerary} of $x$.  Define a
  sequence $(N_m)_m$ so that
  $\Wpseq{m} := \Wseq{m}\cap D_{N_m}\ne\emptyset$ and
  $\hkreg_{\bar n}(\Wpseq{m})\ge \LC_{\bar n}(\Delta)+1$. Such a
  sequence exists since any sufficiently short unstable curve
  intersects at most two domains $D_N$.  Passing to the
  $(\tau, I)$-coordinates and taking a subsequence we may assume that
  $T_{-N_m} \Wpseq{m}$ converges to some point $\brx\in \cl(\hDom_0)$,
  where $T_n$ is the translation map defined
  in~\eqref{e_translationMapDefinition}.  The convergence in the
  $\admName$-metric implies convergence in the $(\tau,I)$-Euclidean
  metric by~\eqref{e_boundAlphaTau}.

  Since $\cmp^{\bar n}$ is continuous on the set of points with a given
  itinerary, it follows that there are points
  $\xseq{m}_1, \xseq{m}_2\dots \xseq{m}_{\LC_{\bar n}(\Delta)+1}\in
  \Wpseq{m}$ having different $k$-itineraries.  Possibly by extracting a
  subsequence, we may thus assume that for $1\le l\le\LC_{\bar n}(\Delta)+1$
  \begin{align*}
    \xseq{m}_{l}\in D_{N_m, N_m+n_{1,l}\dots N_m+n_{\bar n-1,l}},
  \end{align*}
  that is, that the itinerary depends on $N_m$ only via the shift by
  $N_m$.  But then, Theorem~\ref{t_normalForm} implies that
  $\brx\in \hDom_{0, n_{1,l}\dots n_{\bar n-1,l}}$ for every $l$,
  therefore $\LC_{\bar n}(\brx)\geq \LC_{\bar n}(\Delta)+1$, which
  contradicts the definition of $\LC_{\bar n}(\Delta)$.
\end{proof}

\subsection{Linear bound on regular complexity.}
\label{SSPointwise}%

In this section we  prove a linear bound for $\ckregp_n$
defined by ~\eqref{e_defInducedRegularComplexity}.

\begin{lem}\label{l_actualInducedComplexityBound}
  For any $n > 0$ we have
  \begin{align}\label{e_actualInducedComplexityBound}
    \ckregp_n < 4n + 2.
  \end{align}
\end{lem}
The induced regular complexity $\ckregp_n$ bounds the number of
connected components of $\cspi\setminus\singp{n}$ that are regular at
any point $x$.  Since such connected components are bounded by $C^{1}$
curves, it is possible to formulate an equivalent infinitesimal
definition, which we now  describe.

For $x\in\csp$, denote by $\tangu_x\csp$ the unit tangent sphere at
$x$.  We identify each element of $\upsilon\in\tangu_{x}\csp$ with the
equivalence class of $C^1$-curves in $\csp$ which emanate from $x$ with
a tangent vector that is a positive multiple of $\upsilon$.  Of course
$\tangu_x\csp$ embeds naturally in $\tang_x\csp$; this embedding defines
a topology on $\tangu_x\csp$.  Observe that if $x\in\intr\csp$, then
$\tangu_x\csp = \mathbb S^1$, but if $x\in\sing0$, then $\tangu_x\csp$
is a closed quarter-sphere if $x = (0,0)$ or $x = (1,0)$ and a closed
half-sphere otherwise.  All such sets will be considered with the
counterclockwise orientation.  Similarly, we define, for any $x\in\cspi$,
the set $\tangu_x\cspi$.

A $C^1$-curve in $\csp$ emanating from $x$ thus naturally induces an
element of $\tangu_x\csp$. In particular if $x\in\sing n$, then the
curves in $\sing n$ cut $\tangu_x\csp$ into a number of connected
components which we call \emph{tangent sectors}.  With a slight abuse of
notation we write $\tangu_x\csp\setminus\sing n$ to denote
$\tangu_{x}\csp\setminus\{\upsilon_1,\cdots,\upsilon_{p}\}$ where the
$\upsilon_i$'s are the unit vectors induced by the curves of $\sing{n}$
which meet at $x$.  Similar considerations apply to $\cspi$~and~$\singp n$.

More generally, given two elements
$\upsilon_-\ne\upsilon_+\in\tangu_x\csp$ let
$\cV=\cV(\upsilon_-, \upsilon_+)$ denote the set of directions lying
between $\upsilon_-$ and $\upsilon_+$ with respect to the
counterclockwise orientation. This set will be called the \emph{tangent
  sector centered at $x$ bounded by $\upsilon_-$ and $\upsilon_+$}.
Conventionally, we also introduce the notion of \emph{empty sector}
$\cV = \emptyset$ and \emph{full sector} $\cV = \tangu_{x}\csp$.  A
curve $\Gamma$ which emanates from $x$ with unit tangent vector
$\upsilon\in\cV$ is said to be \emph{compatible} with $\cV$.

Note that all sufficiently short curves compatible with
$\cV\subset\tangu_{x}\csp\setminus\sing n$ necessarily belong to the
same connected component $Q_n$.  Likewise, all sufficiently short curves
compatible with $\cV\subset\tangu_{x}\cspi\setminus\singp n$ necessarily
belong to the same connected component $\hat Q_n = \hat Q_n(\cV)$. We
denote $\Np_n(\cV) = \Np_n(\cV(\hat Q_n))$.

Let $\cV\subset\tangu_{x}\csp\setminus\sing n$ and $\Gamma$ be a curve
compatible with $\cV$.  By construction we have that
$\DS \lim_{\Gamma\ni x'\to x}\cm^l x'$ is well defined and independent of
$\Gamma$ for any $0\le l\le n$. Let us denote this limit point
$x^l_\cV$.  Likewise, if $\cV\subset\tangu_{x}\cspi\setminus\singp n$,
we can uniquely define $x^{l}_{\cV}$ for any $0\le l\le \Np_n(\cV)$.

Let $\cV\subset\tangu_{x}\csp\setminus\sing n$; we can define for any
$0\le l\le n$ the image sector
$\cV^l\subset\tangu_{x^{l}_{\cV}}\csp\setminus\sing{-l,n-l}$ as follows.
Let $\Gamma$ be a curve compatible with $\cV$.  By construction we have
that $\DS \lim_{\Gamma\ni x'\to x}d\cm^l(x')$ is a well defined linear map
and independent of $\Gamma$ for any $0\le l\le n$.  We denote its action
on $\tangu_x\csp$ by
$\cm_{*,\cV}^{l}:\tangu_x\csp\to\tangu_{x^{l}_{\cV}}\csp$.  Then, with a
small abuse of notation we denote with $\cm_*^{l}\cV$ the sector
$\cm_{*,\cV}^{l}\cV$.  A similar construction yields, for any
$\cV\subset\tangu_{x}\cspi\setminus\singp n$ and any $0\le l\le n$ the
definition of
$\cmp^l_*\cV\subset\tangu_{x^{\Np_l(\cV)}_{\cV}}\cspi\setminus\singp{-l,n-l}$.

A tangent sector $\cV\subset\tangu_{x}\csp\setminus\sing n$ is said to
be \emph{$\cm^n$-regular} if it is non-empty and
$x_\cV^{l}\in\cl\homo_0$ for any $0 < l\le n$.  Otherwise, we say that
the sector is \emph{nearly grazing}.  Likewise, a tangent sector
$\cV\subset\tangu_{x}\cspi\setminus\singp n$ is said to be
\emph{$\cmp^n$-regular} if it is non-empty and $x_\cV^{l}\in\cl\homo_0$
for any $0 < l\le \Np_n(\cV)$.

Of course the above definitions are compatible with the ones given
previously for $Q_n$ and $\hat Q_n$ in the sense that a sector
$\cV\in\tangu_x\csp\setminus\sing n$ is $\cm^{n}$-regular if and only if
the corresponding connected component $Q_n$ is $n$-regular at $x$, and a
sector $\cV\in\tangu_x\cspi\setminus\singp n$ is $\cmp^{n}$-regular if
and only if the corresponding connected component $\hat Q_n$ is
$n$-regular at $x$.
This immediately follows by our construction unless the
connected component joins $x$ with a cusp (\ie the corresponding sector
is empty).  But then we claim that the component must necessarily be
nearly grazing at $x$.  In fact, it is easy to see that if the sector
generated by a connected component $\hat Q_n$ is degenerate, then there
exists $0 < l\le \Np_n(\hat Q_n)$ so that $d\cm^l|_{\hat Q_n}$ is
singular as we approach $x$.  Since $d\cm$ is singular only at
$\{w = 0\}$,  $\hat Q_n$ cannot be a regular at $x$.

In
particular the regular complexity $\ckreg_n$ is the maximum number of
$\cm^{n}$-regular sectors in which $\sing n$ cuts $\tangu_x\csp$ for any
$x\in\csp$.  The corresponding statement holds true for $\ckregp_n$.

\begin{mydef}\label{d_goodSector}
  A tangent sector $\cV(\upsilon_-,\upsilon_+)\subset\tangu_x\csp$ (or
  $\cV(\upsilon_-,\upsilon_+)\subset\tangu_x\cspi$) is said to be
  \emph{good} if
 
  \begin{enumerate}[label = (\roman*)]
    \item\label{i_decreasingVectors}
    $\upsilon_{+},\upsilon_-\in\conen_x$ (recall
    definition~\eqref{e_definitionClosedCones}) and
    \item\label{i_sectorPi} the angle between $\upsilon_{-}$ and
    $\upsilon_+$ does not exceed $\pi$.
  \end{enumerate}
  
  A good tangent sector $\cV(\upsilon_-,\upsilon_+)$ is said to be
  \emph{active} if $\upsilon_-$ and $\upsilon_+$ belong to different
  quadrants, and \emph{inactive} if they belong to the same quadrant.
\end{mydef}
Observe that an active good sector contains either the first or the
third quadrant (in particular, the stable cone); inactive sectors cannot
contain any such quadrants.  In particular, since future singularities
are union of stable curves (Lemma~\ref{l_localSing}), if a good sector
$\cV\subset\tangu_{x}\csp$ (\resp $\cV\subset\tangu_{x}\cspi$) is
inactive, then for any $k > 0$ we have
$\cV\subset\tangu_{x}\csp\setminus\sing k$ (\resp
$\cV\subset\tangu_{x}\cspi\setminus\singp k$).

Good sectors satisfy the following invariance property.
\begin{lem}\label{l_invarianceGoodSectors}
  Let $\cV\subset\tangu_{x}\csp$ be a good sector, and 
  $\cV\setminus\sing 1 = \bigcup_{i = 1}^{s}\cV_{i}$.  Then each image
  sector $\cm_*\cV_{i}$ is good.  Similarly, if
  $\cV\subset\tangu_{x}\cspi$, and
  $\cV\setminus\singp1 = \bigcup_{i = 1}^{s}\cV_{i}$, we have that each
  image sector $\cmp_*\cV_i$ is a good sector.
\end{lem}
\begin{proof}
  First of all observe that the image by a linear map of a sector of
  angle at most $\pi$ is a sector of angle at most $\pi$.  We conclude
  that item~\ref{i_sectorPi} in Definition~\ref{d_goodSector} holds for
  each of the image sectors.

  Let $\upsilon$ be one of the boundary vectors of $\cV_{i}$.  There are
  two possibilities: either $\upsilon$ is one of the boundary vectors of
  $\cV$, or it is induced by $\sing1$.  In the first case,
  $\upsilon\in\conen_{x}$ and thus~\eqref{e_decreasingCone} implies that
  its image
  $\cm_{\cV_{i},*}\upsilon\in\coneu_{\cm_{\cV_{i}}x}\subset\conen_{\cm_{\cV_{i}}x}$.
  In the second case, we have by construction that
  $\cm_{\cV_{i}}\upsilon$ is tangent to some curve in $\sing{-1}$.
  Lemma~\ref{l_localSing} then implies that also in this case
  $\cm_{\cV_{i},*}\upsilon\in\conen_{\cm_{\cV_{i}}x}$, which concludes
  the proof of the first part.  The second part follows from identical
  considerations.
\end{proof}
\begin{rmk}
  The above lemma implies in particular that if
  $\cV\subset\tangu_x\csp\setminus\sing k$ is a good sector, then
  $\cV^{l}$ are also good sectors for any $0\le l\le n$.
\end{rmk}
The linear bound~\eqref{e_actualInducedComplexityBound} will be obtained
by means of the following lemma, whose proof we briefly postpone.
\begin{lem}\label{l_inductiveStep}\
  \begin{enumerate}
  \item \label{i_inductiveStep} %
    Let $x\in\csp\setminus\{\xc\}$. Any active good tangent sector
    $\cV\subset\tangu_x\csp$ is cut by $\sing1$ in at most two
    $\cm$-regular sectors.  The $\cm$-image of at most one of them is active.
  \item \label{i_inducedInductiveStep} %
    Let $x\in\cspi\setminus\{\xc\}$. Any active good tangent sector
    $\cV\subset\tangu_x\cspi$ is cut by $\singp1$ in at most three
    $\cmp$-regular sectors.  The $\cmp$-image of at most one of them is
    active.
  \end{enumerate}
\end{lem}
We can now prove the main result of this subsection.
\begin{proof}[Proof of Lemma~\ref{l_actualInducedComplexityBound}]
  First observe that Lemma~\ref{l_continuityAtxc} implies that if $x$ is
  sufficiently close to $\xc$, then $\cm x$ is also close to $\xc$,
  which implies that $\Np(x) = 1$ and that $\cmp x\nin\homo_0$.  Hence,
  no sector $\cV\subset\tangu_x\cspi$ can be $\cmp$-regular.  We can
  thus assume $x\in\cspi\setminus\{\xc\}$.

  Cutting $\tangu_x\cspi$ along the vertical direction we obtain (up to)
  $2$ good sectors (recall Remark~\ref{r_boundary cspi}); of course both
  such sectors might be active.%

  Let $\cV$ denote one such active sector. We now show inductively that
  for any $k > 0$, the singularity set $\singp{k}$ cuts $\cV$ in at most
  $(2k+1)$ $\cmp^{k}$-regular sectors, and the $\cmp^{k}$-image of at
  most one of them is active.
  Lemma~\ref{l_inductiveStep}\ref{i_inducedInductiveStep} proves our
  claim for $k = 1$.  In order to proceed with our proof, we need to set
  up some notation: for any $k\ge 1$, the singularity set $\singp {k}$
  cuts $\cV$ in a number $s_{k}$ of sectors
  $(\cV^{(k)}_0,\cV^{(k)}_1,\cdots,\cV^{(k)}_{s_{k}-1})$; let $r_k$
  denote the number of such sectors that are $\cmp^{k}$-regular.
  Without loss of generality we can take them to be
  $(\cV^{(k)}_0,\cV^{(k)}_1,\cdots,\cV^{(k)}_{r_k-1})$.

  Assume now, by induction, that our claim holds for $k$; we gather that
  $r_k\le {2}k+1$ and that the image of at most one of the regular
  sectors is active.  If no sector is active, no further cutting is
  allowed, so we are done. Hence we assume that one sector is active and
  without loss of generality we let it be indexed as $\cV^{(k)}_0$.

  Consider now the $\cmp^{k+1}$-regular sectors
  $(\cV^{(k+1)}_0,\cV^{(k+1)}_1,\cdots,\cV^{(k+1)}_{r_{k+1}-1})$
  obtained by cutting $\cV$ by $\singp{k+1}$.  By definition of
  $\cmp^{k+1}$-regularity, for any $0 \le i < r_{k+1}$ there exists
  $0\le j < r_k$ so that $\cV^{(k+1)}_{i}\subset\cV^{(k)}_j$.  However,
  if $\singp{k+1}$ cuts $\cV^{(k)}_j$, then it must be that its
  $\cmp^{k}$-image is cut by $\singp1$, but this is only possible if
  said image is active, \ie if $j = 0$.  Applying
  Lemma~\ref{l_inductiveStep}\ref{i_inducedInductiveStep} to this
  sector, we thus conclude that it can be cut it at most three regular
  sectors and that the image of at most one of them is active.  This in
  turn proves that $r_{k+1} \le r_{k}+2$. This proves our claim for
  $k+1$.

  Since $\tangu_x\cspi$ consists of at most two active sectors we
  conclude that $x$ has at most $2(2n+1)$ regular sectors when cut by
  $\singp{n}$.  Since $x$ was  arbitrarily, this
  proves~\eqref{e_actualInducedComplexityBound}.
\end{proof}

\begin{proof}[Proof of Lemma~\ref{l_inductiveStep}]\ %
  We first show how item~\ref{i_inducedInductiveStep} follows from
  item~\ref{i_inductiveStep}.  Recall that $\cmp$ is the first return
  map of $\cm$ to the set $\cspi$, which is defined
  in~\eqref{e_definitionInducedSpace}. %
  Recall also (see~\eqref{e_domainSingularities}) that
  $\Dom_n\cap\sing{n-1} = \emptyset$ for any $n > 0$, and that
  $\Dom_n\cap\singp1 = \Dom_n\cap (\cm^{-(n-1)}\sing1)$.  
  Since by definition $\bigcup_{n\ge0}\cl\Dom_n = \cspi$ and
  $\cl\Dom_n\cap\cl\Dom_{n'} = \emptyset$ unless $|n-n'|\le 1$,  there
  are two possibilities:
  \begin{enumerate}[label = (\roman*)]
  \item there exists a \emph{unique} $n$ so that $x\in\cl\Dom_n$;
  \item $x\in\cl\Dom_n\cap\cl\Dom_{n+1}$ for some $n$.
  \end{enumerate}
  Assume first that possibility (i) holds. Since
  $\Dom_n\cap\sing{n-1} = \emptyset$, we conclude that $\singp1$ cuts
  $\cV$ in as many (regular) sectors as $\sing1$ cuts $\cm^{n-1}_*\cV$.
  This shows that, in this case, item~\ref{i_inductiveStep} implies
  item~\ref{i_inducedInductiveStep}. %

  Next, suppose that possibility (ii) holds. Also in this case
  $x\nin\sing{n-1}$, so we can define the sector
  $\cV^{*} = \cm^{n-1}_*\cV$.  By item~\ref{i_inductiveStep}, the
  singularity set $\sing1$ cuts $\cV^*$ in at most two $\cm$-regular
  sectors $(\cV^*_0,\cV^{*}_1)$. By Lemma~\ref{l_invarianceGoodSectors}
  the image of both of them is a good
  sector
  and of the image of
  at most one of them (say $\cV^{*}_0$) may be active.  Since
  $x\in\cl\Dom_{n+1}$, some of these sectors may belong to $\Dom_{n+1}$;
  for such sectors we need to consider the cutting by $\sing2$.  If
  $\cV^*_0$ is disjoint from $\Dom_{n+1}$ or its image is not active, we
  are done, since no further cutting can take place.  On the other hand,
  if $\cV^*_0$ belongs to $\Dom_{n+1}$ and its image is active, it might
  be cut by $\sing2$ into further sectors.
  Applying~\ref{i_inductiveStep} to $\cm\cV^*_0$ we gather that $\sing2$
  can cut $\cV^*_0$ into at most two $\cm^2$-regular sectors, the
  $\cm^2$-image of both of them is a good sector and of at most one of
  them is active.  This proves that~\ref{i_inductiveStep}
  implies~\ref{i_inducedInductiveStep} also in case (ii).  Note that we
  have at most two sectors in case (i) and at most three in case (ii).

  It remains to prove item~\ref{i_inductiveStep}.  If $x\not\in\sing1$,
  or $x\in\sing0\setminus\cl(\sing1\setminus\sing0)$, the map $\cm$ is
  smooth in a neighborhood of $x$ and the statement immediately follows.

  We thus assume that $x\in\cl(\sing1\setminus\sing0)$.  Recall (see
  Lemma~\ref{LmSingF}(a-b)) that $x$ can belong to at most one of
  the $\singForward{\indCell}$ and, possibly, to $\singReco^{+}$.

  If $x\in\singReco^+$, then, by Lemma \ref{l_forwardReco}, $\backReco$
  induces a sector which is not $\cm$-regular.  Hence, only cells
  $\cell^+_\indCell$ can induce $\cm$-regular sectors and by
  Lemma~\ref{l_singularityStructure} there are only two possibilities:
  \begin{enumerate}
  \item there exists a \emph{unique} $\indCell$ so that
    $\cell^+_\indCell$ so that $x\in\cl\cell^+_\indCell$.
  \item there exist two consecutive cells $\cell^+_\indCell$ and
    $\cell^{+}_{\indCell+1}$ so that
    $x\in\cl\cell^+_\indCell\cap\cl\cell^+_{\indCell+1}$ (and $x$ does
    not intersect the closure of any other cell.)
  \end{enumerate}
  This already establishes that $\cV$ is cut
  by $\sing1$ in at most two $\cm$-regular sectors.  We now need to
  prove that at most one of their images is an active sector.  Observe
  that if $\cV$ is cut in fewer than two regular sectors, there is
  nothing left to prove.  This is the situation, in particular, in
  case (a).

  \begin{figure}[!h]
    \centering
    \hspace*{-2.5cm} \includegraphics{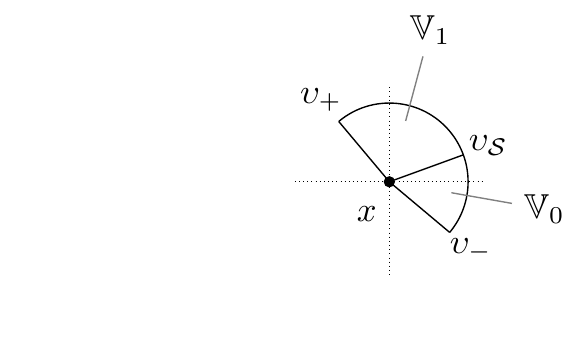}
    \hspace*{-1.5cm} \includegraphics{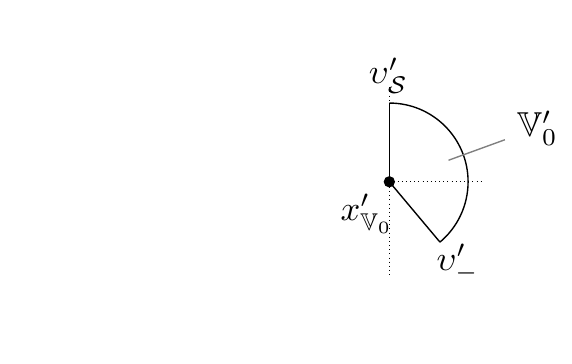}
    \hspace*{-2.5cm} \includegraphics{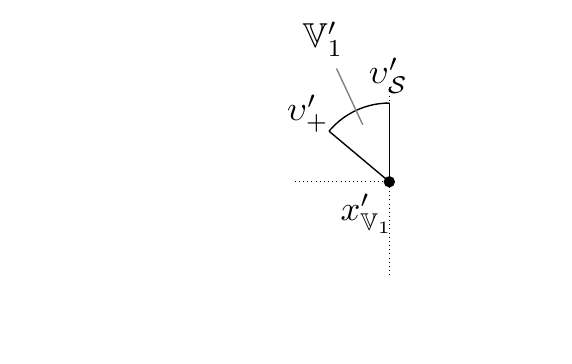}\\
    \hspace*{-2.5cm} \includegraphics{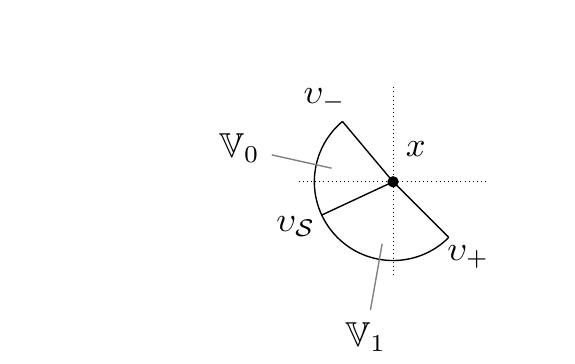}
    \hspace*{-1.5cm} \includegraphics{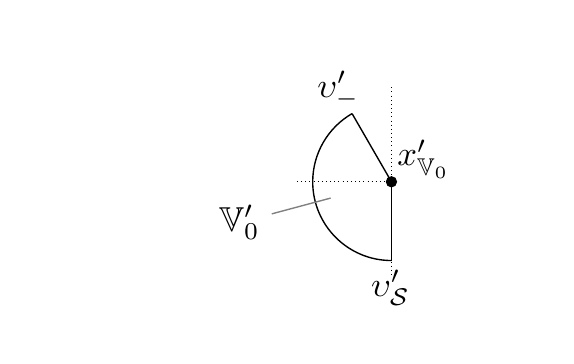}
    \hspace*{-4.5cm} \includegraphics{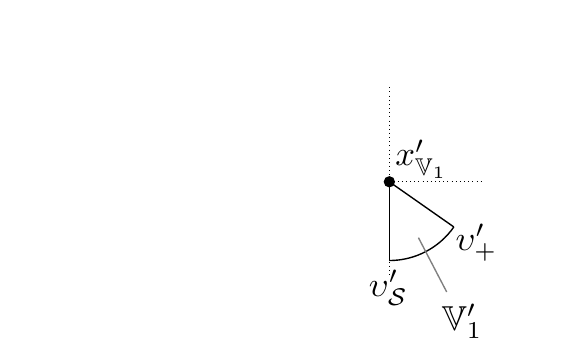}\\
    \hspace*{-2.5cm} \includegraphics{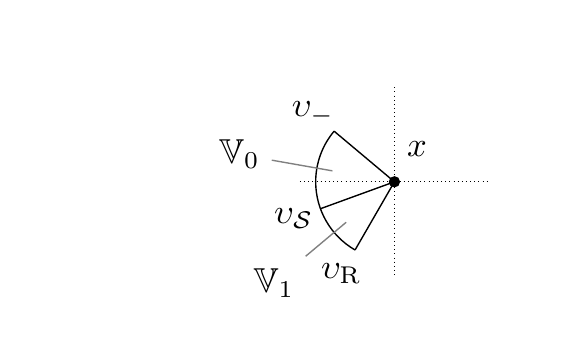}
    \hspace*{-1.5cm} \includegraphics{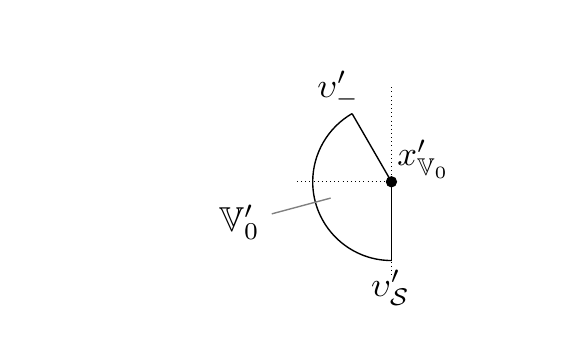}
    \hspace*{-4.5cm} \includegraphics{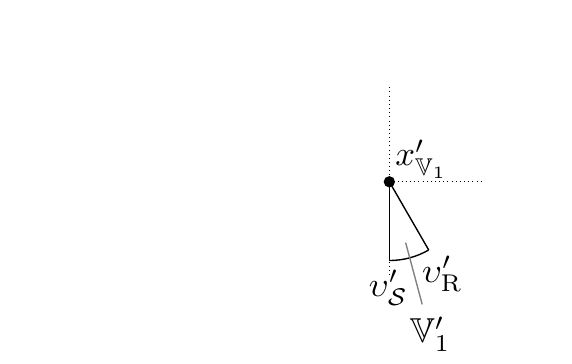}\\
    \caption{The three possible cutting cases for $\cV$ by $\singp {1}$
      in regular sectors (on the left), and their images (on the right)
      by the two differentials $\cm_{*,\cV_0}$ and $\cm_{*,\cV_1}$
      respectively.}
    \label{f_cases}
  \end{figure}
  In case (b), we necessarily have that
  $x\in\singForward\indCell$.  We subdivide the argument into two
  further subcases:
  (i) $x\nin\singReco^{+}$;
  (ii) $x\in\singReco^{+}.$

  In case (i), $\singForward{\indCell}$ cuts $\cV$ in exactly two sectors,
  induced by $\cell^{+}_{\indCell}$ and $\cell^{+}_{\indCell+1}$.
  Notice that these two sectors have a common boundary vector, which is
  induced by $\singForward{\indCell}$: we can then write the two sectors
  as (see Figure~\ref{f_cases}, first and second row)
  $\cV_0 = \cV(\upsilon_-,\upsilon_\sing{})$ and
  $\cV_1 = \cV(\upsilon_{\sing{}},\upsilon_+)$.  We say we are in case
  $i'$ if $\cV$ contains the first quadrant (see first row of
  Figure~\ref{f_cases}) and in case $i''$
  if $\cV$ contains the third
  quadrant (see second row of Figure~\ref{f_cases}).

  Consider first case $i'$. By inspection we gather that $\cV_0$ is
  induced by $\cell^{+}_{\indCell+1}$
  and $\cV_1$ is induced by $\cell^{+}_{\indCell}.$
  Since we assume both sectors to be regular,
  Lemma~\ref{l_singularityStructure}\ref{i_adjacentCells} implies that
  \begin{align*}
    \lim_{y\to x}\cm_{\cV_{0}} x&\in\{0\}\times\reals^+, &
    \lim_{y\to x}\cm_{\cV_{1}} x&\in\{1\}\times\reals^+.
  \end{align*}
  Thus the image $\cm_{*,\cV_0}\upsilon_{\sing{}}$ (\resp
  $\cm_{*,\cV_1}\upsilon_{\sing{}}$) is a vertical vector.  Moreover,
  since $\upsilon_{\sing{}}$ lies in the first quadrant), then both its
  images are vertical vectors pointing upwards.  The other boundary
  vector of each $\cV_{i}$ is one of the original vectors
  $\upsilon_{\pm}$, and thus its image is unstable.  Since
  $\cm_{*,\cV_i}$ is orientation preserving, we conclude that only one
  of the images of $\cV_i$'s can be an active sector (see again
  Figure~\ref{f_cases}, row 1).

  Case $i''$ is completely analogous. In
  this case $\cV_0$ is induced by $\cell^{+}_{\indCell}$ and $\cV_1$ is
  induced by $\cell^{+}_{\indCell+1}$.  Once again, since both sectors are regular,
  we gather by Lemma~\ref{l_singularityStructure}\ref{i_adjacentCells}
  that
  \begin{align*}
    \lim_{y\to x}\cm_{\cV_{0}} x&\in\{1\}\times\reals^+&
    \lim_{y\to x}\cm_{\cV_{1}} x&\in\{0\}\times\reals^+.
  \end{align*}
  Hence the image $\cm_{*,\cV_0}\upsilon_{\sing{}}$ (\resp
  $\cm_{*,\cV_1}\upsilon_{\sing{}}$) is a vertical vector.  Moreover,
  since $\upsilon_{\sing{}}$ lies in the third quadrant, then both its
  images are vertical vectors pointing downwards.  The other boundary
  vector of each $\cV_{i}$ is one of the original vectors
  $\upsilon_{\pm}$, and thus its image is unstable .  Since
  $\cm_{*,\cV_i}$ is orientation preserving, we conclude that only one
  of the images of $\cV_i$'s can be an active sector (see
  Figure~\ref{f_cases}, second row).  This completes the proof in case (i).

  In case (ii), combining Lemma~\ref{LmSingF} (we need the part concerning
  $\singForward{}$!)  with Lemma \ref{l_forwardReco} we gather that $x$
  is the right endpoint of $\singForward\indCell$.  Therefore
  $\singForward\indCell$ will cut $\cV$ only if $\cV$ contains the third
  quadrant.  Thus, if $\cV$ contains the first quadrant, then
  $\singForward\indCell$ does not cut $\cV$.  Thus $\cV$ could only be
  cut by $\singReco^+$ and by an earlier discussion $\cV$ contains at
  most one regular sector, so we are done.

  \ignore{
    only one of the two cells
    (namely $\cell^{+}_{\indCell}$) can induce a sector in $\cV$, and
    therefore we conclude $\cV$ is cut in at most one regular sector and,
    as noted before, nothing else is left to prove.}

  It remains to consider the more difficult case in which $\cV$ contains the third
  quadrant (Figure~\ref{f_cases}, bottom row).
  Since $x$ is the right endpoint of
  $\singForward{\indCell}$, we conclude that the vector induced by
  $\singForward{\indCell}$ must meet with $\singReco^{+}$ on the left.
  Therefore  the two regular sectors are
  $\cV_0 = \cV(\upsilon_-,\upsilon_\sing{})$ and
  $\cV_1 = \cV(\upsilon_{\sing{}},\upsilon_{\recoPrivate})$.  As in case
  $i'',$ we have that $\cV_0$ is induced by $\cell^{+}_{\indCell}$ and
  $\cV_1$ is induced by $\cell^{+}_{\indCell+1}$; the vector
  $\upsilon_{\recoPrivate}$ is induced by $\singReco^{+}$.  Following
  the same reasoning as in case $i''$ above, we conclude that the image
  $\cm_{*,\cV_0}\upsilon_{\sing{}}$ (\resp
  $\cm_{*,\cV_1}\upsilon_{\sing{}}$) is a vertical vector pointing
  downwards.  The image $\cm_{*,\cV_0}\upsilon_{-}$ is of course
  unstable and belongs to the second quadrant.  The image of
  $\upsilon_{\recoPrivate}$ is also in $\conen$ (as it will be induced
  by some curve in $\sing{-1}$) and points downwards.
  Hence, only $\cV_0'$ is active.

  This concludes the argument in case (ii)
  and finishes the proof.
\end{proof}

\section{Invariant manifolds.}\label{sec:invariant-manifolds}
The expansion estimate proved in the previous section
is the
main ingredient for the so-called Growth Lemma (Lemma~\ref{l_growth}).
In turn the Growth Lemma constitutes the backbone for proving ergodicity using
the Hopf argument, as will be done in the next section.  The Hopf
argument relies on existence of a large set of points which have
sufficiently long stable and unstable manifolds.  The present section
contains necessary results about the existence of stable and unstable
manifolds as well as regularity of partition of the phase space into
stable and unstable manifolds.  In this section we always assume that
the Fermi--Ulam model is regular at infinity.
As a notational convention, in an attempt to simplify our notation, in
this section we drop the superscripts from $\dadm{W}(\cdot,\cdot)$, as
they can be unambiguously recovered from the context.
\renewcommand{\dadm}{d\adm}

\subsection{The Growth Lemma}\label{sec:growth-lemma}
In this section we state and prove a version of the Growth Lemma for
our system.  This lemma will allow to obtain, in the next subsection, a
good lower bound on the length of stable and unstable manifolds
passing through most of the points.

Let $W$ be an unstable curve and $x\in W$. $x$ subdivides $W$
into two subcurves. We define $r_W(x)$ as the $\admName$-length of the
shortest of the two subcurves.  The function $r_{W}(x)$ measures, in
an appropriate way, the distance of $x$ to the boundary of $W$.
Observe that if $W$ is weakly homogeneous, we have,
by~\eqref{eq:dW-over-d-bound},\; $r_{W}(x) < \const \dadm(x,\hhs)$.

Observe moreover that
\begin{align}\label{eq:trivial-bound-for-r}
  \Leb_{W}(r_{W}(x) < \eps) = \min\{2\eps,\Leb_{W}(W)\}
\end{align}
(recall that $\Leb_{W}$ denotes Lebesgue measure on the curve $W$
with respect to the $\admName$-metric).

Given an unstable curve $W,$ a point $x\in W$ and $n\ge 0$, we define
$W_{n}(x)$ as follows.  If $x\in\hsing n$ we let
$W_{n}(x) = \emptyset$; otherwise we let $W_{n}(x)$ to be the
H-component of $\cm^{n} W$ that contains $\cm^{n} x$ (recall the
discussion before Proposition~\ref{p_expansionEstimate}).
Then we define $r_{W,n}(x) = r_{W_{n}(x)}(\cm^{n}x)$ (or $0$ if $W_{n}(x) = \emptyset$).

Likewise, given an unstable curve $W$, $x\in W$ and $n\ge 0$, we
define $\widehat{W}_{n}(x)$ and $\hat r_{W,n}$ as follows. Recall the
definition of $\Np_{n}$ given before Remark~\ref{rmk:definition-Npk};
if $\Np_{n}(x)$ is not defined, we let $\widehat{W}_{n}(x) = \emptyset$ and
$\hat r_{W,n}(x) = 0$. Otherwise we let
$\widehat{W}_{n}(x) = W_{\Np_{n}(x)}(x)$ and
$\hat r_{W,n}(x) = r_{W,\Np_{n}(x)}(x)$.
\begin{lem}
  We have $r_{W,0} = \hat r_{W,0} = r_{W}$ and
\begin{align}\label{eq:silly-bound-on-r}
  r_{W,n}(x) < \const \dadm(\cm^{n}x,\hhs).
\end{align}
Moreover, there exists $C > 1$, so that if $\cm^{n}W$ is a single
H-component, then for any $x\in W$:
  \begin{align}\label{eq:r-trivial-bound-connected}
    r_{W,n}(x) > C\inv\Lambda^{\hat n(\cm^{n}W)} r_{W}(x),
  \end{align}
  where $\Lambda$ is the constant appearing
  in~\eqref{e_uniformHyperbolicity}  and $\hat n$ was defined
  in~\eqref{eq:definition-hatn}.
\end{lem}
\begin{proof}
  The first two items follow immediately from the definition and from
  our previous observation.  We thus need to
  prove~\eqref{eq:r-trivial-bound-connected}.  By definition
  $r_{W,n}(x) = |W'_{n}(x)|\adm$, where $W'_{n}(x)$ is shortest
  subcurve of $W_{n}(x)$ joining $x_{n}$ with $\partial W_{n}(x)$.
  Since $\cm^{n}W$ is a single H-component, we conclude that
  $W_{n}(x) = \cm^{n}W$.  Thus $W'_{n}(x)$ connects $x_{n}$ with
  $\partial \cm^{n}W$, and $\cm^{-n}W'_{n}(x)$ connects $x$ with
  $\partial W$. In particular $|\cm^{-n}W'_{n}(x)|\adm\ge r_{W}(x)$.
  Then the proof follows
  from~\eqref{e_adm-dominate},~\eqref{e_uniformHyperbolicity} and the
  definition of $\hat n$.
  \end{proof}

The following is the classical Growth Lemma.
\begin{lem}[Growth Lemma for $\hat r$]\label{l_growth}
  Suppose that the Fermi--Ulam model is regular at infinity.  Then
  there exists $0 < \theta<1$ and $C > 0$ so that for any sufficiently
  short mature \admiss{} unstable curve~$W\subset\csp$, any
  $\eps > 0$ and any $n > 0$
  \begin{align}\label{eq:growth-lemma-induced}
    \Leb_W(\hat r_{W,n}(x) < \eps)\le C\eps\,\Leb_W(W)+ C
    \Leb_W\left(r_W(x)\leq \theta^{n}\eps\right).
  \end{align}
\end{lem}
\begin{proof}
  The proof of the Growth Lemma follows via relatively standard
  arguments (see \cite[Sections 5.9 and 5.10]{ChM}) from the expansion
  estimate (Proposition~\ref{p_expansionEstimate}) and the distortion
  bounds proved in Corollary~\ref{cor:distortion}.

  Recall the definition of $\hat\cL_{n}$ given right before
  Proposition~\ref{p_expansionEstimate}, and let $\bar n$ be the
  number appearing in Proposition~\ref{p_expansionEstimate}. We fix
  $\delta > 0$ to be sufficiently small so that
  $\bar\theta = e^{2\Cdist\delta^{1/12}}\hat\cL_{\bar n} < 1$ (where
  $\Cdist$ is the constant appearing in
  Corollary~\ref{cor:distortion}) and that
  Lemma~\ref{l_globalExpansion}(b) holds with $k = \bar n$ and
  $\delta_{*} = 1$.

  Let us first assume that $W\subset \cspi$ and that
  $|W|\adm < \delta$.  Then we claim that there exists $\bar C > 0$ so
  that for any $\eps > 0$:
  \begin{align}\label{eq:growth-lemma-inductive-step}
    \Leb_{W}(\hat r_{W,\bar n}(x) < \eps) <  \bar C\eps\Leb W + \Leb_{W}(r_{W}(x) <                                          e^{-\Cdist\delta^{1/12}}\bar\theta\eps).
  \end{align}
  As we observed in Corollary~\ref{cor:distortion}, our distortion
  bounds on unstable curves depend on their length. In this proof we
  will need very fine distortion bounds, and it will then be necessary
  to work only with sufficiently short unstable curves.  This entails
  a partitioning scheme for H-components that we now proceed to
  describe. Let $\{W_{i}\}$ denote the set of H-components of $\cmp^{\bar n}W$.
  We partition each $W_{i}$ into a number
  \begin{align*}
    k_i = \left\lfloor \frac{|W_{i}|\adm}{\delta}\right\rfloor+1
  \end{align*}
  of subcurves of equal $\admName$-length (smaller than $\delta$) that
  we denote with $W_{ij}$.  Observe that if $|W_{i}|\adm < \delta$,
  $k_{i} = 1$, and no shortening takes place.  We call such subcurves
  \emph{shortened H-components of $\cmp^{\bar n}W$}.  We will shorten
  the H-components inductively every $\bar n$ steps of the induced map
  $\cmp$. By our choice of $\delta$, this guarantees that at each
  intermediate step, no H-component will have $\admName$-length
  exceeding $1$.  Given $x\in W$, we will then denote with
  $\hat W'_{n}(x)$ the shortened H-component of $\cmp^{n}W$ whose
  interior contains $\cmp^{n}x$ (or $\emptyset$ if some image of $x$
  lies on an endpoint of a shortened subcurve).  We then define
  $\hat r'_{W,n}(x) = r_{\hat W'_{n}(x)}(\cmp^{n}x)$.  Observe
  that $\hat r'_{W,\bar n} < \hat r_{W,\bar n}$, so that
  proving~\eqref{eq:growth-lemma-inductive-step} for $\hat r'_{W,n}$
  will imply~\eqref{eq:growth-lemma-inductive-step} for
  $\hat r_{W,n}$.  Let $B_{ij}\subset W_{ij}$ be the
  $\eps$-neighborhood (in the $\admName$-metric) of the boundary of
  each $W_{ij}$; in particular $\Leb_{W_{ij}}(B_{ij}) = 2\eps$.  Then
   \begin{align*}
    \Leb_{W}(\hat r'_{W,\bar n}(x) < \eps) = \sum_{ij}\Leb_{W}(\cmp^{-\bar n}B_{ij}).
  \end{align*}
 By the distortion estimates of
  Corollary~\ref{cor:distortion}
  \begin{align*}
    \sum_{ij}\Leb_{W}(\cmp^{-\bar n}B_{ij}) &\le 
    e^{\Cdist\delta^{1/12}}\sum_{ij}\Leb_{W}(\cmp^{-\bar n}W_{ij})
    \frac {\Leb_{W_{ij}}(B_{ij})}{\Leb_{W_{ij}}(W_{ij})}. \\
    &\le 2e^{\Cdist\delta^{1/12}}\eps \sum_{ij}
      k_{i}\frac{\Leb_{W}(\cmp^{-\bar n}W_{ij})}{\Leb_{W_{i}}(W_{i})} \\
    &\le 2\delta\inv e^{\Cdist\delta^{1/12}}\eps
      \sum_{ij}\Leb_{W}(\cmp^{-\bar n}W_{ij})+ \\
    &\phantom\le+2
      e^{\Cdist\delta^{1/12}}\eps\sum_{ij}\frac{\Leb_{W}(\cmp^{-\bar
      n}W_{ij})}{\Leb_{W_{i}}(W_{i})} \\
    &\le
      \bar C\eps\Leb_{W}(W)+2e^{\Cdist\delta^{1/12}}\eps\sum_{i}\frac{\Leb_{W}(\cmp^{-\bar
      n} W_{i})}{\Leb_{W_{i}}(W_{i})} \\
    &\le \bar C\eps\Leb_{W}(W)+2e^{\Cdist\delta^{1/12}}\eps\hat\cL_{\bar n},
  \end{align*}
  where we defined $\bar C = 2\delta\inv e^{\Cdist\delta^{1/12}}$.
  Using~\eqref{eq:trivial-bound-for-r}, the fact that the left hand
  side is always bounded above by $\Leb_{W}$, and our definition of
  $\bar\theta$, we conclude that
  \begin{align}\label{eq:growth-lemma-inductive-step-r'}
    \Leb_{W}(\hat r'_{W,\bar n}(x) < \eps) <  \bar C\eps\Leb W + \Leb_{W}(r_{W}(x) <                                          e^{-\Cdist\delta^{1/12}}\bar\theta\eps).
  \end{align}
  which, as noted earlier, 
  implies~\eqref{eq:growth-lemma-inductive-step}.

  We now proceed to show that for any $k > 0$:
  \begin{align}\label{eq:growth-lemma-induced-estimate}
    \Leb_{W}(\hat r'_{W,k\bar n}(x) < \eps) &\le
    e^{\Cdist\delta^{1/12}}\frac{1-\bar\theta^{k}}{1-\bar\theta}\cdot
    \bar C\eps\Leb_{W}(W)+ \\
    &\phantom =
    +\Leb_{W}(r_{W}(x) \le\bar \theta^{k}\eps).\notag
  \end{align}
  For $k = 1$  \eqref{eq:growth-lemma-induced-estimate} follows from
  \eqref{eq:growth-lemma-inductive-step-r'}.
    Let
  us assume by induction that~\eqref{eq:growth-lemma-induced-estimate}
  holds for $k$ and prove it for $k+1$.  Let $W'$ be a shortened
  H-component of $\cmp^{k\bar n}$. Notice that by construction
  $W'\subset\cspi$ and $|W'|\adm < \delta$.  Then,
  applying~\eqref{eq:growth-lemma-inductive-step} to $W'$ we gather:
  \begin{align*}
    \Leb_{W'}(\hat r'_{W',\bar n}(y) < \eps) &\le \bar C\eps\Leb_{W'}(W') +
                                               \Leb_{W'}(\hat
                                               r_{W'}(y)\le
                                               e^{-\Cdist\delta^{1/12}}\bar\theta\eps).
  \end{align*}
  Let $W'' = \cmp^{-k\bar n}$, then by Corollary~\ref{cor:distortion},
  we conclude that:
  \begin{align*}
    \Leb_{W''}(\hat r'_{W'',(k+1)\bar n}(x) < \eps) &
    \le e^{\Cdist\delta^{1/12}}\bar C\eps\Leb_{W''}(W'')\\
    &\phantom = +
    e^{\Cdist\delta^{1/12}}\Leb_{W''}%
    (\hat r'_{W'',k\bar n}(x) < e^{-\Cdist\delta^{1/12}}\bar \theta\eps).
  \end{align*}
  Summing over all $W''$'s and applying the inductive hypothesis yields:
 $$    \Leb_{W}(\hat r'_{W,(k+1)\bar n}(x) < \eps) \le $$
    $$e^{\Cdist\delta^{1/12}}\bar C\eps\Leb_{W}(W)+
   e^{\Cdist\delta^{1/12}}\Leb_{W}
      (\hat r'_{W,k\bar n}(x) < e^{-\Cdist\delta^{1/12}}\bar\theta\eps) \le $$
 $$   e^{\Cdist\delta^{1/12}}\bar C\frac{1-\bar\theta^{k+1}}{1-\bar\theta}\eps\Leb_{W}(W)+
  e^{\Cdist\delta^{1/12}}\Leb_{W}
      (r_{W}(x) < e^{-\Cdist\delta^{1/12}}\bar\theta^{k+1}\eps),
$$
  which yields~\eqref{eq:growth-lemma-induced-estimate} for $k+1$.
  Hence we can write:
  \begin{align}\label{eq:growth-lemma-final-estimate}
    \Leb_{W}(\hat r'_{W,k\bar n}(x) < \eps)\le
    C\eps\Leb_{W}(W)+\Leb_{W}(r_{W}(x) \le \bar\theta^{k}\eps).
  \end{align}
  where $C = \ifrac{\bar C e^{\Cdist\delta^{1/12}}}{(1-\theta)}$.

  We now extend this estimate to iterates that are not
  multiples of $\bar n$.  We begin by obtaining a bound on
  $\Leb_{W}(\hat r'_{W,s}(x) < \eps)$ for $s < \bar n$.  Notice that
  no partitioning into short curves occurs before step $\bar n$,
  therefore if $\{W_{i}\}$ denotes the set of $H$-component of
  $\cmp^{s}W$, we have
  \begin{align*}
    \Leb_{W}(\hat r'_{W,s}(x) < \eps)  = \Leb_{W}(\hat r_{W,s}(x) < \eps) = \sum_{i}\Leb_{W}(\cmp^{-s}B_{i}),
  \end{align*}
  where $B_{i}$ is a $\eps$-neighborhood of the boundary of $W_{i}$.
  Then we proceed as before. Since $|W|\adm < \delta$, we are guaranteed
  that $|W'|\adm < 1$. Thus, applying the distortion bounds
  in Corollary~\ref{cor:distortion}, we gather:
  \begin{align*}
    \sum_{i}\Leb_{W}(\cmp^{-s}B_{i})&\le
    2e^{\Cdist}\eps \sum_{i}\frac{\Leb_{W}(\cmp^{-s}W_{i})}{\Leb_{W_{i}}(W_{i})}\le
    2 e^{\Cdist}\eps\hat\cL_{s}.
  \end{align*}
  Applying once again~\eqref{eq:trivial-bound-for-r}, and observing
  that by Proposition~\ref{p_expansionEstimate} we have that
  $\hat\cL_{s}$ is bounded uniformly in $s$, yields:
  \begin{align}\label{eq:growth-lemma-transient-bound}
    \Leb_{W}(\hat r_{W,s}(x) < \eps)\le \Leb_{W}(r_{W}(x) < \Const\eps).
  \end{align}
  Now, for any $m > 0$, we write $m = k\bar n+s$, with
  $0 \le s < \bar n$. Applying~\eqref{eq:growth-lemma-transient-bound}
  to each shortened component $W'$ of $\cmp^{k\bar n}W$ yields:
  \begin{align*}
    \Leb_{W'}(\hat r_{W',s}(y) < \eps)\le \Leb_{W'}(r_{W '}(x) < \Const\eps).
  \end{align*}
  Taking $W'' = \cmp^{-k\bar n}W'\subset W$, and applying the
  distortion bounds:
  \begin{align*}
    \Leb_{W''}(\hat r_{W'',k\bar n+s}(x) < \eps)&
    \le e^{\Cdist\delta^{1/12}}\Leb_{W''}(r'_{W'',k\bar n}(x) < \Const\eps).
  \end{align*}
  Now summing over all $W''$ and
  applying~\eqref{eq:growth-lemma-final-estimate}, we finally conclude
  that
  \begin{align*}
    \Leb_{W}(\hat r_{W,k\bar n+s}(x) < \eps)&
    \le  e^{\Cdist\delta^{1/12}}C\eps\Leb_{W}(W)+\Leb_{W}(r_{W}(x)\le \Const\bar \theta^{k}\eps).
  \end{align*}
  Choosing $\theta = \bar\theta^{1/\bar n}$ and $C = \Const\theta\inv$
  yields~\eqref{eq:growth-lemma-induced} under the assumption
  $W\subset \cspi$ and $|W|\adm < \delta$.

  Now, observe that, given an unstable curve $W$, for any $x\in W$,
  $\hat W_{1}(x)$ is either $\emptyset$ or it is a curve
  $W'\subset\cspi$.  By Lemma~\ref{l_globalExpansion}, it is possible
  to assume $W$ so short that each $W'$ is such that
  $|W'|\adm < \delta$.  By applying once again the distortion
  argument, we deduce that~\eqref{eq:growth-lemma-induced} holds in
  the general case, by suitably increasing the constants.
\end{proof}

We are now going to complement the Growth Lemma above (which involves
iterates of $W$ by $\cmp$) with some estimates on the length of the
iterates of unstable curves by $\cm$.  More precisely, let $W$ be an
unstable curve and $x\in W$: we define:
\begin{align*}
  \bar r_{W}(x) = \min_{0\le n < \Np(x)}r_{W,n}(x),
\end{align*}
with the convention that if $\Np(x)$ is undefined, then $\bar r_{W}(x)
= 0$.
\begin{lem}[Transient growth control]\label{lem:transient-growth-control}
  There exists $C > 0$ so that for any sufficiently short mature
  \admiss{} unstable curve $W$:
  \begin{align*}
    \Leb_{W}(\bar r_{W}(x) < \eps) < \Leb_{W}(r_{W}(x) < C\eps).
  \end{align*}
\end{lem}
\begin{proof}
  The proof follows from distortion arguments similar to the ones
  given in the proof of the Growth Lemma.  Assume that
  $|W|\adm < \delta$.  Fix $\largew > 0$ sufficiently large. Assume
  first that $W\subset\{w\le \largew\}$.  Then there exists $N_{*} =
  \Const \largew$ so that $\Np(x) < N_{*}$ for any $x\in W$.  Thus:
  \begin{align*}
    \Leb_{W}(\bar r_{W}(x) < \eps) 
    &\le \sum_{n = 0}^{N_{*}-1}\Leb_{W}(r_{W,n}(x) < \eps).
  \end{align*}
  We proceed to obtain a bound on $\Leb_{W}(r_{W,n}(x) < \eps)$.  Let
  us fix $n > 0$ and let $\{W_{i}\}$ denote the set of H-components of
  $\cm^{n}W$; let $B_{i}$ be an $\eps$-neighborhood of the boundary of
  $W_{i}$. Then
  \begin{align*}
    \Leb_{W}(r_{W,n}(x) < \eps) = \sum_{i}\Leb_{W}(\cm^{-n}B_{i}).
  \end{align*}
  Assuming $|W|\adm < \delta$, we are guaranteed that each component
  $W_{i}$ satisfies $|W_{i}|\adm < 1$. Hence by our distortion bounds
  (Corollary~\ref{cor:distortion})
  \begin{align*}
    \sum_{i}\Leb_{W}(\cm^{-n}B_{i})&\le 2e^{\Cdist}\eps
                                     \sum_{i}\frac{\Leb_{W}(\cm^{-n}W_{i})}{\Leb_{W}(W_{i})} \\
    &\le 2e^{\Cdist}\eps \cL_{n}(W) \\
     &\le \Leb_{W}(r_{W}(x) < e^{\Cdist}\cL_{n}\eps).
  \end{align*}

   By \eqref{eq:cL-submultiplicative},  $\cL_{n}\le\cL_{1}^{n}.$
  Thus $\cL_{n}\le\max\{1,\cL_{1}^{N_{*}}\}$, which is bounded
  by~\eqref{eq:cl1-bounded}.  This concludes the proof of the lemma in
  the case of low energies.

  Let us assume, on the other hand, that
  $W\cap \{w > \largew\}\ne\emptyset$.  Then if $\largew$ is
  sufficiently large and $\delta$ sufficiently small, by
  Lemma~\ref{l_largeCell}\ref{i_almostFinite-large}, $W$ intersects at
  most two cells $\cEs_{n}$. Such cells partition $W$ in (at most) two
  subcurves $W_{1}$ and $W_{2}$ so that $\Np(x) = N_{*}$ for all
  $x\in W_{1}$ and $\Np(x) = N_{*}+1$ for all $x\in W_{2}$, for some
  $N_{*} > 0$.  Note that
  \begin{align*}
  \Leb_{W}(\bar r_{W}(x) < \eps)\le\Leb_{W}(\bar r_{W_{1}}(x) < \eps)
    + \Leb_{W}(\bar r_{W_{2}}(x) < \eps).
  \end{align*}

  Let us consider $\bar r_{W_{1}}(x)$; by construction
  $W_{1}\subset\cEs_{N_{*}}$.  Since
  $\cEs_{N_{*}}\cap\sing{N_{*}-1} = \emptyset$, we gather that
  $\cm^{n}W_{1}$ is connected for any $0\le n <
  N_{*}$. Thus,~\eqref{eq:r-trivial-bound-connected} ensures that
  $r_{W_{1},n}(x)\le C\inv r_{W_{1}}(x)$ for any $n < N_{*}$, and
  therefore $\bar r_{W_{1}}(x) < C\inv r_{W_{1}}(x)$.  By the same
  token we conclude $\bar r_{W_{2}}(x) < C\inv r_{W_{2}}(x)$. Hence
  \begin{align*}
    \Leb_{W}(\bar r_{W}(x) < \eps) \le \Leb_{W}(r_{W} < 2C\eps)
  \end{align*}
  which concludes the proof of the lemma.
\end{proof}
In order to obtain bounds on the length of stable and unstable
manifolds, we will need some results similar to the ones presented
above, but for slightly different functions $r$. We now proceed to define
them and link their properties to the ones of the functions $r$ that
have been investigated above.

Recall the properties of the singularity sets $\singBoth$
outlined in Lemma~\ref{LmSingF} and define, for $N \ge 0$:
\begin{align*}
  \singForward{(N)} &= \sing{0}\cup\singForward\recoPrivate\cup\bigcup_{\indCell = 0}^{N}\singForward{\indCell}.
\end{align*}
For $x\in W$ let us define $r_W(x,\singForward{(N)})$ as follows. If
$x\in\singForward{(N)}$ we set $r_W(x,\singForward{(N)}) = 0$.
Otherwise $\singForward{(N)}$ cuts $W$ into finitely many subcurves.
Let $W'$ be the subcurve that contains $x$ and
$r_W(x,\singForward{(N)}) = r_{W'}(x)$. Observe that necessarily
$r_{W}(x,\singForward{(N)})\le r_{W}(x)$.  Finally define\footnote{
  The motivation for this definition will become clear to the reader
  in the proof of Lemma~\ref{lem:kolya-patch}}
\begin{align*}
  r_{W}^{*}(x) = \inf_{N > 0}\{N^{3/2}r_{W}(x,\singForward{(N)})\}.
\end{align*}

Notice that $r^{*}_{W}(x)\le r_{W}(x)$, and it could, in principle, be
much smaller than $r_{W}$. However, the measure of points where this
possibility occurs is under control thanks to the following bound.
\begin{lem}\label{lem:r-star-vs-r}
  There exists $C > 0$ so that for any unstable curve $W$
  \begin{align*}
    \Leb_{W}(r^{*}_{W}(x) < \eps)\le \Leb_W(r_{W}(x) < C\eps).
  \end{align*}
\end{lem}
\begin{proof}
  By Lemma{~\ref{LmSingF}}, we conclude that the set
  $\{r_{W}^{*}(x) < \eps\}$ is contained in the union of
  \begin{itemize}
  \item 2 intervals of $\admName$-length $\eps$ at the boundary of $W$
  \item an interval of $\admName$-length $2\eps$ centered at each point
    of $W\cap(\singForward{\recoPrivate}\cup\singForward{0})$;
  \item an interval of $\admName$-length $2\indCell^{-3/2}\eps$
    centered at each point of $W\cap\singForward{\indCell}$ for
    $\indCell > 0$.
  \end{itemize}
  Hence
  \begin{align*}
    \Leb_{W}(r_{W}^{*}(x) < \eps) < 2\eps (1+2+\sum_{\indCell > 0}\indCell^{-3/2}) < 2\Const\eps.
  \end{align*}
  Since by definition $\Leb_{W}(r_{W}^{*}(x) < \eps) \le \Leb_{W}(W)$, we
  conclude that
  \begin{align*}
    \Leb_{W}(r_{W}^{*}(x) < \eps) & \le \Leb_{W}(r_{W}(x) < C\eps) .\qedhere
  \end{align*}
\end{proof}
Using the above lemma, it is possible to obtain a Growth Lemma and
transient growth control for $r^{*}$.
Let $W$ be an unstable curve and $x\in W$. For $n\ge0$ we define
$r^{*}_{W,n}(x)$ as follows; if $x\in\hsing n$ we let $r^{*}_{W,n}(x)
= 0$; otherwise $W_{n}(x)\ne\emptyset$ and we set
\begin{align*}
  r^{*}_{W,n}(x) = r^{*}_{W_{n}(x)}(\cm^{n}x).
\end{align*}
Likewise, given $n\ge 0$, if $\Np_{n}(x)$ is not defined, we let
$\hat r^{*}_{W,n}(x) = 0$. Otherwise we define
\begin{equation}
\label{DefHatR}
\hat r^{*}_{W,n}(x) = r^{*}_{W,\Np_{n}(x)}(x).
\end{equation}
  Finally,
let $x\in W$. If $\Np(x)$ is undefined, we let $\bar r^{*}_{W}(x)
= 0$. Otherwise let
\begin{align*}
  \bar r^{*}_{W}(x) &= \min_{0\le n < \Np(x)}r^{*}_{W,n}(x).
\end{align*}
We now prove for $\bar r^{*}$ the same bound that was proved in
Lemma~\ref{lem:transient-growth-control}.
\begin{lem}\label{lem:transient-growth-control-star}
  There exists $C > 0$ so that for any sufficiently short mature
  \admiss{} unstable curve $W\subset \csp$
  \begin{align}\label{eq:transient-growth-r-star}
    \Leb_{W}(\bar r^{*}_{W}(x) < \eps) &< \Leb_{W}(r_{W}(x) < C\eps).
  \end{align}
\end{lem}
\begin{proof}
  Assume $|W|\adm < \delta$ and fix $\largew > 0$ sufficiently large.
  Assume first that $W\subset\{w\le\largew\}$.  Then there exists $N_{*}
  = \Const\largew$ so that $\Np(x) < N_{*}$ for any $x\in W$.  Thus:
  \begin{align*}
    \Leb_{W}( \bar r^{*}_{W}(x) < \eps)\le \sum_{n = 0}^{N_{*}-1}
    \Leb_{W}(r^{*}_{W,n}(x) < \eps).
  \end{align*}
  Lemma~\ref{lem:r-star-vs-r} then implies that
  \begin{align*}
    \Leb_{W}(\bar r_{W}^{*}(x) < \eps)\le \sum_{n = 0}^{N_{*}-1}\Leb_{W}(r_{W,n}(x) < C\eps).
  \end{align*}
 Now arguing as in the proof of
  Lemma~\ref{lem:transient-growth-control}, we conclude
  that~\eqref{eq:transient-growth-r-star} holds in this first case.

  Assume now that $W\cap\{w > \largew\}\ne\emptyset$, then if $\delta$
  is sufficiently small and $\largew$ sufficiently large, we conclude
  by Lemma~\ref{p_largeEnergiesFUM}\ref{i_boundOnExcursion} that for
  any $x\in W$ and any $0\le n < \Np(x)$,
  $\cm^{n}x\in\{w\ge \largew/2\}$.  First of all notice that
  Lemma~\ref{LmSingF} and the construction of $\cEs_{n}$ guarantees
  that $\cEs_{n}\cap\sing+ = \emptyset$ unless $n = 1$.  By Lemma~\ref{LmSingF}(d)
  $\sing+_{\indCell}$ is compact for $\indCell > 0$.
    Therefore for large enough $\largew$,
  the only possible curve of $\sing+$ that intersects with
  $\cEs_{0}\cap\{w\ge\largew/2\}$ is $\sing+_{0}$, but
  $\sing+_{0}\subset\partial\cEs_{0}$; we conclude that
  $\cEs_{0}\cap\{w\ge\largew/2\} \cap \sing+ = \emptyset$.  We thus
  proceed as in the proof of Lemma~\ref{lem:transient-growth-control}:
  If $\largew$ is sufficiently large and $\delta$ sufficiently small,
  by Lemma~\ref{l_largeCell}\ref{i_almostFinite-large}, $W$ intersects
  at most two cells $\cEs_{n}$; such cells partition $W$ in (at most)
  two subcurves $W_{1}$ and $W_{2}$.  Then
  \begin{align*}
    \Leb_{W}(\bar r^{*}_{W}(x) < \eps)&\le
    \Leb_{W}(\bar r^{*}_{W_{1}}(x) < \eps)+
    \Leb_{W}(\bar r^{*}_{W_{2}}(x) < \eps).
  \end{align*}
  Notice that $\cm^{n}W_{i}$ will belong to only one cell
  $\cEs_{\indCell}$ for any $n$ involved in the definition of
  $\bar r^{*}_{W_{i}}$.  By the argument above, we gather that $\bar
  r^{*}_{W_{i}} = \bar r_{W_{i}}$. Now we conclude arguing as
  in the proof of Lemma~\ref{lem:transient-growth-control}.
\end{proof}

\subsection{Size    of invariant manifolds.}\label{SSInvMan}
Recall that a stable curve $W$ is  a homogeneous stable
manifold if $|\cm^{n}W|\adm\to 0$ as $n\to\infty$ and $\cm^{n}W$
belongs to a single homogeneity strip for any $n\ge 0$. Recall also
the corresponding definition for unstable manifolds.  Given
$x\in \csp$, we denote with $\Ws(x)$ (\resp $\Wu(x)$) the maximal
homogeneous stable (\resp unstable) manifold containing $x$ (or
$\emptyset$ if such manifold does not exists).  Conventionally we
consider such curves without the endpoints.  We now give a convenient
characterization of $\Ws(x)$ and $\Wu(x)$. The construction closely
follows~\cite[Section 4.11]{ChM}, and we refer the reader to that
section for additional details.  For
$x\in\csp\setminus\hsing{-\infty}$, we denote with $Q_{-n}(x)$ the
connected component of the open set $\csp\setminus\hsing{-n}$ that
contains $x$.  Naturally, $Q_{-n}(x)\supset Q_{-(n+1)}(x)$ for any
$n$. Moreover $\overline{Q_{-n}}(x)$ is compact for any $n$
sufficiently large, possibly depending on $x$.\footnote{ This holds
  since, for $n$ sufficiently large (\eg
  $n > \Np(x)+\Np(\cm^{\Np(x)}(x))$), the set $\cm^{\Np(x)}Q_n(x)$ is
  contained in some fundamental domain $\Dom_{m}$, and such sets are
  bounded (see \eg~\eqref{Per-W}).}  Let
$\widetilde\Wu(x) = \bigcap_{n\ge1} \overline{Q_{-n}(x)}$.  Using
compactness of $\overline{Q_{-n}(x)}$ and Lemma~\ref{l_S-dense} one
can show that $\widetilde\Wu(x)$ is a compact unstable curve. It then
follows that $\Wu(x)$ is equal to $\widetilde\Wu(x)$ minus the
endpoints.  A completely similar construction can be carried over for
$\Ws(x)$. 

If $\Wu(x) = \emptyset$ we define $\ru(x) = 0$. Otherwise, $x$
subdivides $\Wu(x)$ in two subcurves; we denote with $\ru(x)$ the
$\admName$-length of the shortest of such subcurves.  Define $\rs(x)$
similarly.

We now obtain lower bounds for $\rs$ and $\ru$. In order to
do so we  introduce some notation.  Given $x\in\csp$, define
the functions $E^{\pm}:\csp\to\reals$ so that if
$x\in\homo_{k}\cap\cell^{\pm}_{\indCell}$, then
$E^{\pm}(x) = (\indCell+1)(k^{2}+1)$. More precisely
\begin{align*}
  E^{\pm}(x) &= \sum_{k}(k^{2}+1)\chi_{\homo_{k}\cap\cell^{\pm}_{\recoPrivate}}+\sum_{k,\indCell}(k^{2}+1)(\indCell+1)\chi_{\homo_{k}\cap\cell^{\pm}_{\indCell}}(x),
\end{align*}
where $\chi$ denotes the indicator function of the set written as its
subscript.
\begin{lem}
  $E^{\pm}$ controls the contraction and expansion of stable and
  unstable vectors by $d\cm$ as follows:
  \begin{subequations}
    \label{eq:function-E}
    \begin{align}
    \label{FunE1}
      C\inv E^{-}(\cm x) &< \frac {\|d\cm
                           v^{\unstable}\|}{\|v^{\unstable}\|} < C
                           E^{-}(\cm x)& \fa
                                         x\in\csp\setminus\singForward{}, v^{\unstable}\in\coneu_{x}\\
                                         \label{FunE2}
      C\inv E^{+}(\cm\inv x) &< \frac {\|d\cm\inv
                               v^{\stable}\|}{\|v^{\stable}\|} < C
                               E^{+}(\cm\inv x)& \fa
                                                 x\in\csp\setminus\singBackward{}, v^{\stable}\in\cones_{x}.
    \end{align}
  \end{subequations}
\end{lem}
\begin{proof}
  Of course it suffices to show\eqref{FunE1}, then \eqref{FunE2}
 follows from the properties of the involution.  If $\cm x\in\reco$,
  then the lower bound follows from~\eqref{e_expansionAdm-II} and the
  upper bound follows from~\eqref{e_expansionAuxMetricA} and
  Corollary~\ref{CrHESqueeze}\ref{i_lower-bound-prpslope}.  On the
  other hand, suppose $\cm x\nin\reco$. If $w$ is large, then our estimates
  follow from Corollary~\ref{CrHESqueeze}
  and~\eqref{e_expansionAuxMetricA}. If $w$ is small,
  Lemma~\ref{LmPSlopeC0}(b) and~\eqref{e_expansionAuxMetricA} yield
  the desired estimate.
\end{proof}

Given $x\in\csp$ and $n\in\bZ$, we denote with $\ds(x,\hsing n)$
(\resp $\du(x,\hsing n)$) the length (in the $\admName$-metric) of the
shortest\footnote{ The existence of such a curve follows from the fact
  that the stable (\resp unstable) cone is closed and that the
  singularity set is closed. } stable (\resp unstable) curve which
connects $x$ with $\hsing n$.

For $x\in\csp$, let $\Lambda^{\unstable}_{n}(x)$ be the minimal
expansion of unstable vectors by $d\cm^{n}|_{x}$. Similarly, let
$\Lambda^{\stable}_{n}(x)$ be the minimal expansion of stable vectors
by $d\cm^{-n}|_{x}$.  Notice that there exists $\underline\Lambda >
0$: so that for any $n > 0$ and $x\in\csp$
\begin{align}\label{eq:lambda-lowerbar}
  \Lambda_{n}^{\stable}(x) > \underline\Lambda, \quad
  \Lambda_{n}^{\unstable}(x) > \underline\Lambda.
\end{align}

Moreover, by definition, for any $0 < m < n$:
\begin{align*}
  \Lambda^{\unstable}_{n}(x)&\ge\Lambda^{\unstable}_{m}(x)\Lambda^{\unstable}_{n-m}(\cm^{m}x)&
  \Lambda^{\stable}_{n}(x)&\ge\Lambda^{\stable}_{m}(x)\Lambda^{\stable}_{n-m}(\cm^{-m}x).
\end{align*}
Hence  by~\eqref{eq:function-E}, there exists $c > 0$ so that for any
$n\ge 1$
\begin{subequations}
\begin{align}\label{eq:lambda-E}
  \Lambda^{\unstable}_{n}(x)&\ge c E^{-}(\cm
                              x)\Lambda^{\unstable}_{n-1}(\cm x),\\
  \Lambda^{\stable}_{n}(x)&\ge c E^{+}(\cm\inv
                            x)\Lambda^{\stable}_{n-1}(\cm\inv x).
\end{align}
\end{subequations}

\begin{lem} \label{PrSize}%
  For any $L > 0$ there exists a constant $c > 0$ such that
  \begin{align*}
    \rs(x)&\geq \min\{L,c \inf_{n > 0} \Lambda_{n}^{\stable}(\cm^{n}x) \ds(\cm^{n}x, \hsing{-1})\},\\
    \ru(x)&\geq \min\{L,c \inf_{n > 0} \Lambda_{n}^{\unstable}(\cm^{-n}x) \du(\cm^{-n}x, \hsing{1})\}.
  \end{align*}
\end{lem}
\begin{proof}
  The proof of the lemma is a combination of the arguments given
  in~\cite[Lemma 4.67, (4.61), Exercise 5.19 and (5.58)]{ChM}.

  Let us
  prove the statement for $\ru$ (the statement for $\rs$ follows as
  usual by the properties of the involution).  We may further assume
  that $x\in\csp\setminus\hsing{-\infty}$ (otherwise the right hand
  side of the inequality is $0$ and the statement holds trivially). As
  before, for any $n$, we let $Q_{-n}(x)$ be the connected component
  of $\csp\setminus\hsing{-n}$ containing the point $x$; clearly
  $Q_n(\cm^{-n}x) = \cm^{-n}(Q_{-n}(x))$ is the connected component of
  $\csp\setminus\hsing{n}$ containing the point $\cm^{-n}x$.

  Let $n^{*}$ be so that $\overline{Q_{-n^{*}}(x)}$ is compact. 
  Choose $w^{*}$ so that
  $\overline{Q_{-n^{*}}(x)}\subset\{w\le w^{*}\}$.  Let us now fix
  $\eps > 0$ and choose $n > n^{*}$ so that $\overline{Q_{-n}(x)}$ is
  contained in an Euclidean $\eps/w^{*}$-neighborhood of $\Wu(x)$.

By construction
  $\cm^{-n}\Wu(x)\subset Q_{n}(\cm^{-n}x)$.  Let $W'_{-n}$ be an
  arbitrary continuation as a mature unstable curve of
  $\cm^{-n}\Wu(x)$ to $\partial Q_{n}(\cm^{-n}x)$.  We further assume
  that $W'_{-n}$ is $\hat K$-\admiss{}\footnote{By
    Corollary~\ref{CrUMLip}, $\cm^{-n}\Wu(x)$ is $\hat K$-\admiss{}
    and we can choose our
    continuation to satisfy this requirement}.  Then
  $W' = \cm^{n}(W'_{-n})$ is an unstable continuation of $\Wu(x)$ that
  terminates on $\partial Q_{-n}(x)$. It is divided by the point $x$
  into two subcurves; denote with $W$ the shortest one (in the
  $\admName$-metric).  By our construction
  and~\eqref{e_defAdmE} we gather that $\ru(x)\ge |W|\adm-\Const
  \eps$.  Since $\eps$ is arbitrary, it suffices to show that
  \begin{align*}
    |W|\adm
    &\geq \min\{L,c \inf_{n > 0} \Lambda_{n}^{\unstable}(\cm^{-n}x) \du(\cm^{-n}x, \hsing{1})\}.
  \end{align*}

  The above bound trivially holds if $|W|\adm \ge L$. Let us thus
  assume that $|W|\adm < L$ and for $0\le m\le n$ let
  $W_{-m} = \cm^{-m}W$.  Since $W_{-n}$ terminates on $\hsing{n}$,
  there exists $m\in[1,n]$ so that $W_{-m}$ joins $\cm^{-m}x$ with
  $\hsing{1}$.  We thus gather
$$    |W|\adm= \frac{|W|\adm}{|W_{-m}|\adm}|W_{-m}|\adm
           \ge\Const \Lambda_{m}^{\unstable}(\cm^{-m}x) |W_{-m}|\adm$$
$$           \ge\Const \Lambda_{m}^{\unstable}(\cm^{-m}x)\du(\cm^{-m}x,\hsing1)
$$
  where we used distortion estimates obtained in
  Corollary~\ref{cor:distortion}.
\end{proof}
The statement we are about to prove below
(Lemma~\ref{lem:kolya-patch}) is the analog of~\cite[Exercise
5.69]{ChM}, but there are some differences which are due to two
separate issues.  First of all the statement of that exercise is
incorrect: the strategy presented in~\cite[Section 5.5]{ChM} has a gap
and needs to be corrected (see~\cite{kolya-patch} for a proposed
solution).  Secondly, the argument would need a non-trivial adaptation
to our specific case because of the nature of our singularities
(presence of corner points, non-compactness).  We thus proceed to give
in detail the statement and the proof of what is needed for our
analysis.  In order to simplify our notation we denote, as usual,
$x_{n} = \cm^{n}x$.
\begin{lem}\label{lem:kolya-patch}
  There exists a constant $C > 0$ so that
  \begin{subequations}
    \begin{enumerate}
    \item for any mature unstable curve $W\subset\csp$, any $n\ge 2$ and
      any $x\in W\setminus\sing n$:
      \begin{align}
        \Lambda^{\stable}_{n}(x_{n})\ds(x_{n},\hsing{-1}) \ge C\min\{
        &\Lambda^{\stable}_{n}(x_{n})r_{W,n}(x),\\\notag
        &\Lambda^{\stable}_{n-1}(x_{n-1})r^{*}_{W,n-1}(x),\\\notag
        &\Lambda^{\stable}_{n-2}(x_{n-2})r_{W,n-2}(x)\}.
      \end{align}
    \item for any unstable curve $W\subset\csp$ that is the image of a
      mature unstable curve and any
      $x\in W\setminus\sing 1$:
      \begin{align}
        \Lambda^{\stable}_{1}(x_{1})\ds(x_{1},\hsing{-1}) \ge C\min\{
        &\Lambda^{\stable}_{1}(x_{1})r_{W,1}(x), r^{*}_{W}(x), r_{W}(x)^{4}\}.
      \end{align}
    \end{enumerate}
  \end{subequations}
\end{lem}
\begin{proof}
  Recall that $\hsing{-1}$ is a closed set (see Remark \ref{RkSCClosed}).
  In particular
  $\ds(x_{n},\hsing{-1})$ is attained as $|V|\adm$, where $V$ is a
  stable curve which joins $x_{n}$ to some point $z\in\hsing{-1}$.  By
  definition (see~\eqref{e_definitionHSing}) we have:
  \begin{align*}
    \hsing{-1} = \hhs \cup \cm(\hhs\setminus\singForward{}) \cup \singBackward{}.
  \end{align*}
  Hence there are three possibilities:
  \begin{enumerate}
  \item $z\in\hhs$;
  \item $z\in\cm(\hhs\setminus\singForward{})$;
  \item $z\in\singBackward{}$.
  \end{enumerate}
We begin with case (a). By definition it holds that
  $|V|\adm\ge \dadm(x_{n},\hhs)$. Using~\eqref{eq:silly-bound-on-r} we
  thus conclude that $\ds(x_n,\hsing{-1})\ge c r_{W,n}(x)$.

In cases (b) and (c) we consider $V' = \cm\inv V$.
  Then $V'$ is a weakly homogeneous stable curve and,
  by~\eqref{eq:function-E}:
  \begin{align*}
    |V|\adm\ge \frac{c |V'|\adm}{E^{+}(x_{n-1})}.
  \end{align*}
  In case (b), $V'$ links $x_{n-1}$ to some point $z'\in\hhs$,
  therefore $|V'|\adm\ge\dadm(x_{n-1},\hhs)$ and
  using~\eqref{eq:lambda-E} we gather that
\begin{align*}
    \Lambda^{\stable}_{n}(x_{n})\ds(x_{n},\hsing{-1})&\ge c\Lambda^{\stable}_{n-1}(x_{n-1})\dadm(x_{n-1},\hhs).
  \end{align*}
  Using again~\eqref{eq:silly-bound-on-r} we thus conclude that
  \begin{align*}
    \Lambda^{\stable}_{n}(x_{n})\ds(x_{n},\hsing{-1})&\ge c\Lambda^{\stable}_{n-1}(x_{n-1})r_{W,n-1}(x).
  \end{align*}

  Finally, we consider case (c): then $V'$ is a stable curve linking
  $x_{n-1}$ to some\footnote{ Note that $\cm\inv$ is undefined on
    $\singBackward{}$ so we cannot quite say that $z' = \cm\inv z$}
  point $z'\in\singForward{}$.  We consider two possibilities:
  \begin{enumerate}
  \item[$(c')$] $x_{n-1}\in\reco$ and $z'\in\{0\}\times[0,\slopeb]$;
  \item[$(c'')$] otherwise.
  \end{enumerate}
  In case $(c')$, observe that since $V'$ is a stable curve, it is
  increasing, and the assumptions in $(c')$ imply that
  $V'\subset\reco$ (see Lemma~\ref{l_forwardReco}).  We have now to
  deal separately with the case $n = 1$ and $n > 1$. If $n > 1$, 
  consider $V'' = \cm\inv V'$. Observe that $V''\subset\backReco$
  is a stable (once again, increasing) curve, which joins
  $x_{n-2}\in\backReco$ to $z''\in\{\ct = 1\}$.  The expansion of
  $d\cm\inv$ along $V'$ is bounded above\footnote{ Remarkably, the
    geometry still allows us to obtain an upper bound on expansion
    despite the fact that $V''$ is not, a priori, weakly homogeneous}
  by $cE^{+}(x_{n-2})$ (since $x_{n-2}\in\backReco$ and it is the
  lowest point on $V''$).  We conclude that
  \begin{align*}
    |V'|\adm\ge c\frac{|V''|\adm}{E^{+}(x_{n-2})}.
  \end{align*}
  Hence, $|V''|\adm\ge \dadm(x_{n-2},\{\ct = 1\})$.  Now $x_{n-2}$
  cuts $W_{n-2}(x)$ into two subcurves; let $W'_{n-2}(x)$ denote the
  subcurve to the right of $x_{n-2}$; then by definition
  $|W'_{n-2}(x)|\ge r_{W,n-2}(x)$.  Notice that
  $W'_{n-2}(x)\subset\backReco$, thus
  $W'_{n-2}(x)\cap\reco = \emptyset$; Corollary~\ref{CrHESqueeze} then
  implies that we have uniform transversality of $W'_{n-2}(x)$ with any
  vertical line, which allows to conclude that
  \begin{align*}
    \dadm{}(x_{n-2},\{\ct = 1\}) &\ge c|W'_{n-2}(x)|\adm\ge c r_{W,n-2}(x).
  \end{align*}
 Hence in case $(c')$ and if $n > 1$:
  \begin{align*}
    \Lambda^{\stable}_{n}(x_{n})\ds(x_{n},\hsing{-1}) &\ge C
        \Lambda^{\stable}_{n-2}(x_{n-2})r_{W,n-2}(x)\}.
  \end{align*}
  Otherwise if $n = 1$, we need to modify the above argument as
  follows.  Applying Lemma~\ref{l_globalExpansion} (and
  Remark~\ref{r_betterGlobalExpansion}) to the stable curve
  $V'\subset\reco$ we conclude that
  \begin{align*}
    |V'|\adm \ge c |V''|\adm^{2};
  \end{align*}
  Then arguing as before (with $W_{n-2}$ replaced by $\cm\inv W$, that
  is guaranteed to be a mature unstable curve by our assumption), we
  conclude that $|V''|\adm\ge c r_{\cm\inv W}(x_{-1})$.  Applying once
  again Lemma~\ref{l_globalExpansion} (and
  Remark~\ref{r_betterGlobalExpansion} to $\cm\inv W$), we conclude
  that $r_{W}(x)\le \Const r_{\cm\inv W}(x_{-1})^{1/2}$, from which we
  finally conclude that
  \begin{align*}
    \rs(x) = |V'|\adm\ge Cr_{W}(x)^{4}.
  \end{align*}

  We now estimate $|V'|\adm$ in case $(c'').$ We claim
  that
  \begin{align}\label{eq:vprime-lower-bound}
    |V'|\adm \ge C\inf_{N > 0}N^{3/2}d(x_{n-1},\singForward{(N)}).
  \end{align}
  The above holds trivially if $z'\in\singForward{(1)}$. Otherwise,
  there exists $\indCell > 1$ so that $z'\in\singForward{\indCell}$.
  This implies that $V'\subset\cell^{+}_{\indCell'}$ where either
  $\indCell' = \indCell$ or $\indCell' = \indCell+ 1$.  Since
  $\cell^{+}_{\indCell'}$ is bounded if $\indCell' > 1$ (see
  Lemma~\ref{l_singularityStructure}\ref{i_boundedCells}), $V '$
  lies in a region where $w$ is bounded and so the $\admName$-metric
  and the Euclidean metric are equivalent.

  Moreover, the angle between $V'$ and $\singForward\indCell$ is
  bounded above by $C\indCell^{-3/2}$ (see the proof of Lemma
  \ref{lem:geometry}).
  Thus
  $\dadm(x_{n-1},\singForward{\indCell})\le C
  \indCell^{-3/2}|V'|\adm$.  Since
  $\dadm(x_{n-1},\singForward{(\indCell)}) \le
  \dadm(x_{n-1},\singForward{\indCell})$, we obtain
  \eqref{eq:vprime-lower-bound}.

  By Lemma~\ref{l_localSing} $\singForward{(N)}$
  is a union of curves compatible with the cone $\conep$. Moreover,
  since we are in case $(c'')$, $x_{n-1}\nin\reco$ (and thus
  $W_{n-1}(x)\cup\reco = \emptyset$).  Hence by
  Corollary~\ref{CrHESqueeze}, $W_{n-1}$ is uniformly transversal to
  any curve in $\conep$ and we conclude that
  $$\dadm(x_{n-1},\singForward{(N)})\ge \Const
  r_{W_{n-1}(x)}(x_{n-1},\singForward{(N)}).$$ This yields
  \begin{align*}
    |V'|\adm \ge C\inf_{N > 0}N^{3/2}r_{W_{n-1}(x)}(x_{n-1},\singForward{(N)})
    = C r^{*}_{W,n-1}(x).
  \end{align*} Therefore
  \begin{align*}
    \Lambda^{\stable}_{n}(x_{n})\ds(x_{n},\hsing{-1})&\ge C\Lambda^{\stable}_{n-1}(x_{n-1}) r^{*}_{W,n-1}(x)
  \end{align*}
concluding the proof.
\end{proof}
Using the two results bounds above it is possible to obtain lower
bounds on the length of stable (\resp unstable) manifolds passing
through most points on any given unstable (\resp stable) mature
admissible curve.  This is done in the following
corollary, which is the analog to~\cite[Theorems 5.66--5.67, Section
5.12]{ChM}.
\begin{cor}\label{CorSize}
 
  (a) There exists $C > 0$ so that for any \admiss{} mature unstable curve $W\subset\csp$ and
    $\eps > 0$ with the property that for every $x\in W$ we have
    $d\adm(\cm x,\hsing{-1}) > C\eps$, then
    \begin{align*}
      \Leb_{W}(\rs(x)\le\eps) < \Const \eps.
    \end{align*}
  
  (a')
    for any \admiss{} mature unstable curve $W\subset\csp$ that is the
    image of a mature unstable curve and any $\eps > 0$:
    \begin{align*}
      \Leb_{W}(\rs(x)\le\eps) < \Const \eps^{1/4}.
    \end{align*}
  
  (b) for any $\eta > 0$ there exists $k > 0$ so that for any
    \admiss{} mature unstable curve $W\subset\csp$ and $\eps > 0$ with
    the property that for every $x\in W$ we have
    $d_{\alpha}(\cm^{n}x,\hsing{-1}) > \eps$ for any
    $0\le n \le \Np_{k}(x)$; then
    \begin{equation}
      \label{ShortStable} \Leb_W(\rs(x)\leq \eps)\leq\eta\eps.
    \end{equation}
 
 (c) There exists $C > 0$ so that
  for any \admiss{} mature stable curve $W\subset\csp$ and
    $\eps > 0$ with the property that for every $x\in W$ we have
    $d\adm(\cm\inv x,\hsing{1}) > C\eps$, then
    \begin{align*}
      \Leb_{W}(\ru(x)\le\eps) < \Const \eps.
    \end{align*}
  
  (c')
    for any \admiss{} mature stable curve $W\subset\csp$ that is the
    pre-image of a mature stable curve and any $\eps > 0$:
    \begin{align*}
      \Leb_{W}(\ru(x)\le\eps) < \Const \eps^{1/4}.
    \end{align*}
  
  (d) for any $\eta > 0$ there exists $k > 0$ so that for any
    \admiss{} mature stable curve $W\subset\csp$ and $\eps > 0$ with
    the property that for every $x\in W$ we have
    $d_{\alpha}(\cm^{-n}x,\hsing{1}) > \eps$, for any
    $\Np_{-k}(x) < n \le 0$; then
    \begin{align*}
      \Leb_W(\ru(x)\leq \eps)&\leq\eta \eps.
    \end{align*}
  \end{cor}
\begin{proof}
  We prove parts (a), (a') and (b). Parts (c), (c') and
  (d) follow by identical arguments by considering $\cm^{-1}.$
  Combining Lemmata ~\ref{lem:kolya-patch} and ~\ref{PrSize} (with
  $L = 1$) with the estimate $r_{W,n}(x)\ge r_{W,n}^{*}(x)$ we obtain
  \begin{align}\label{eq:lower-bound-rs-temp}
    \rs(x)\ge \min\{1,c\Lambda_{1}^{\stable}(\cm x)\ds(\cm x,\hsing{-1}),C
    \inf_{n\ge0}\Lambda_{n}^{\stable}(\cm^{n}x)r_{W,n}^{*}(x)\}.
  \end{align}
Define
  $C = c\inv\underline\Lambda\inv$
  (recall~\eqref{eq:lambda-lowerbar})  to ensure 
  that if
  $\dadm{}(\cm x,\hsing{-1}) > C\eps$, then
  $c\Lambda_{1}^{\stable}(\cm x)\ds(\cm x,\hsing{-1}) > \eps$.  Then,
  under the assumptions of (a), assuming $\eps < 1$, the only
  possibility for $\rs(x) \le \eps$ is that the third term in the
  right hand side of the above expression is small.  In case of (a'),
  we can apply Lemma~\ref{lem:kolya-patch}(b) to bound the second term
  above and conclude that:
  \begin{align*}
    \rs(x)\ge \min\{1,c r_{W}(x)^4,C \inf_{n\ge0}\Lambda_{n}^{\stable}(\cm^{n}x)r_{W,n}^{*}(x)\}.
  \end{align*}
  Using~\eqref{eq:trivial-bound-for-r}, we then conclude that
  \begin{align*}
    \Leb_{W}(r_{W}(x) < C\eps^{1/4})\le \const\eps^{1/4}.
  \end{align*}
  We are hence left to estimate the measure of points where the third
  term of~\eqref{eq:lower-bound-rs-temp} is small.  Observe that
  if $\Np_{m}$ is not defined on some $x\in W$ for some $m$, then
  $x\in\sing{\infty}$. Since $W\cap\sing\infty$ is countable, the set
  of such $x$'s forms a zero Lebesgue measure set on $W$ and can be
  neglected.  We can thus assume that $\Np_{m}(x)$ is defined for any
  $m$ and we can write, recalling the definition of $\Lambda$
  in~\eqref{e_uniformHyperbolicity}:
  $$
    \inf_{n\ge0}\Lambda_{n}^{\stable}(\cm^{n}x)r_{W,n}^{*}(x)
    =    \inf_{m\ge0}\inf_{\Np_{m}(x)\le n < \Np_{m+1}(x)}\Lambda_{n}^{\stable}(\cm^{n}x)r_{W,n}^{*}(x) $$
    $$\ge \inf_{m\ge0}\Lambda_{\Np_{m}(x)}^{\stable}(\cm^{\Np_{m}(x)}x)
    \min_{\Np_{m}(x)\le n < \Np_{m+1}(x)}\Lambda_{n-\Np_{m}(x)}^{\stable}(\cm^{n}x)r_{W,n}^{*}(x) $$
    $$\ge \Const\inf_{m\ge0}\Lambda^{m}
    \min_{\Np_{m}(x)\le n < \Np_{m+1}(x)}Cr_{W,n}^{*}(x)
    \ge \Const\inf_{m\ge0}\Lambda^{m}
    \bar r^{*}_{W,\Np_{m}(x)}(x).
$$
Hence:
\begin{align*}
  \Leb_{W}(\inf_{n\ge0}\Lambda_{n}^{\stable}(\cm^{n}x)r_{W,n}^{*}(x) <\eps) 
  &\le \sum_{m\ge0}\Leb_{W}(\bar r^{*}_{W,\Np_{m}(x)}(x) < \Lambda^{-m}\eps). 
\end{align*}
Using Lemma~\ref{lem:transient-growth-control-star} and recalling the
definition of $\hat r_{W,m}$ (see \eqref{DefHatR}) we obtain
\begin{align*}
   \sum_{m\ge0}\Leb_{W}(\bar r^{*}_{W,\Np_{m}(x)}(x) < \Lambda^{-m}\eps)&\le
  \sum_{m\ge0}\Leb_{W}(\hat r_{W,m}(x) < C\Lambda^{-m}\eps). 
\end{align*}
Then by the Growth Lemma~\ref{l_growth} we can estimate
\begin{align*}
  \Leb_{W}(\hat r_{W,m}(x) < C\hat\Lambda^{-m}\eps)&\le 
  C\Lambda^{-m}\eps\Leb_{W}W + C\theta^{m}\Lambda^{-m}\eps.
\end{align*}
Summing over $m$ and collecting all the above estimates we get
\begin{align*}
  \Leb_{W}(\inf_{n\ge0}\Lambda_{n}^{\stable}(\cm^{n}x)r_{W,n}^{*}(x) <\eps) 
  &\le C\eps.
\end{align*}
This proves items (a) and (a').

The proof of item (b) is similar to the proof of item (a).  Once again
we can neglect the points $x\in W$ where $\Np_{m}$ is not defined for
some $m$. Next,
\begin{align} \label{RsThree}
  \rs(x) \ge \min\{1&,%
         \min_{1\le n < \Np_{k}(x)}c\Lambda_{n}^{\stable}(\cm^{n} x)\ds(\cm^{n} x,\hsing{-1}),\\
         &C \inf_{n \ge\Np_{k}(x)}\Lambda_{n}^{\stable}(\cm^{n}x)r_{W,n}^{*}(x)\}.
         \notag
\end{align}
Choose $k$ so that $\Const\Lambda^{k} < \eta$. The
 assumption of part (b) implies that
 \begin{align*}
   \min_{1\le n < \Np_{k}(x)}c\Lambda_{n}^{\stable}(\cm^{n}
   x)\ds(\cm^{n} x,\hsing{-1}) \ge \eps
 \end{align*}
so only the last term in \eqref{RsThree} could be small.
On the other hand arguing as in part (a) we gather
 \begin{align*}
  \Leb_W\left(
  \inf_{n \ge\Np_{k}(x)}\Lambda_{n}^{\stable}(\cm^{n}x)r_{W,n}^{*}(x)<\eps\right)
   &\le \sum_{m\ge k}\Leb_{W}(\hat r_{W,m}(x) < C\Lambda^{-m}\eps)\\
   &\le \Const \Lambda^{-k}\eps.
 \end{align*}
completing the proof.
\end{proof}
\subsection{Absolute continuity of the holonomy map.}
In this subsection we discuss regularity properties of the holonomy
map.  Let $W_1, W_2\subset \cspi$ be two mature \admiss{} unstable
curves which are close to each other.  More precisely, fix a small
number $\bd > 0$.  Let $\hol$ be the holonomy map defined
by~\eqref{DefHolonomy} and recall the sets $\Omega_1, \Omega_2$
defined by~\eqref{DomainRangeHol}.  We assume that
\begin{align}\label{eq:W12-closeness}
   \sup_{x_{1}\in\Omega_1} d_\alpha(x_{1}, \hol x_{1})\leq \bd.
\end{align}
Recall moreover the definition of unstable Jacobian~\eqref{Jak} and
that $\Leb_{W}$ denotes the Lebesgue measure induced by the
$\admName$-metric.
\begin{prp}
\label{PrAC1}
{\sc (Absolute Continuity-1)} For $\phi\in L^{1}(W_1)$
\begin{align*}
  \int_{\Omega_1} \phi(x_1) d\Leb_{W_{1}}(x_1) &=\int_{\Omega_2} \phi(\hol^{-1} x_{2}) J(\hol^{-1} x_{2}) d\Leb_{W_{2}}(x_{2}).\end{align*}
\end{prp}
\begin{cor}
  If $A\subset \Omega_1$ has zero $\Leb_{W_{1}}$-measure, then
  $\Leb_{W_{2}}(\hol A)=0$.
\end{cor}
\begin{proof}
  Let $B=\hol A$ and assume by contradiction that $\mes\,B > 0$. Then
  since $J$ is bounded from below\footnote{
    Lemma~\ref{LmJak} implies a uniform upper bound, and exchanging
    the roles of $W_{1}$ and $W_{2}$ yields the desired lower bound},
  Proposition~\ref{PrAC1} implies that
  \begin{align*}
    \Leb_{W_{1}}A = \Leb_{W_{1}}(\hol^{-1} B)=\int_B {J(\hol^{-1} x_2)}{d\Leb_{W_{2}}(x_2)}>0.& \qedhere
  \end{align*}
\end{proof}
\begin{proof}[Proof of Proposition \ref{PrAC1}]
 For the ease of notation, we will denote with $dx$ the
  integration with respect to $d\Leb_{W_{1}}(x)$ (or
  $d\Leb_{W_{2}(x)}$, as will be clear from the context).  First of
  all, we can assume that $\phi\in C(W_{1})$; the general case follows
  by the density of $C(W_{1})$ in $L^{1}(W_{1})$.  Moreover, by the usual
  linearity arguments, we can further assume that $\phi$ is
  non-negative.

  Choose $\eps > 0$
  arbitrarily and let $n > 0$ large to be specified later. Let
  $\{W_{j1}\}$ denote the set of shortened
  H-components\footnote{ Recall that shortened H-components were
    defined in the proof of the Growth Lemma~\ref{l_growth}} of
  $\cmp^n W_1$. Recall in particular that $|W_{j1}|\adm < 1$.  For any
  $j$, let $V_{j1}=\cmp^{-n} W_{j1}$ and choose $\brx_{j1}\in V_{j1}$.
  Observe that $|V_{j1}|\adm < \Lambda^{-n}$
  by~\eqref{e_uniformHyperbolicity}. In particular, by uniform
  continuity of $\phi$, if $n$ is sufficiently 
  large\footnote{ Recall that admissible curves have bounded Euclidean
    length, hence they have bounded $\admName$-length by
    Proposition~\ref{p_propertiesAdm}\ref{i_admeum}} (depending on
  $\eps$) then
  \begin{align*}
    \int_{\Omega_1} \phi(x_{1}) dx_{1} &=\sum_j \int_{\Omega_1\cap
                                         V_{j1}} \phi(x_{1}) dx_{1} \\%
                                   &=\sum_j \phi(\brx_{j1}) \Leb_{W_{1}}(V_{j1} \cap \Omega_1)+O(\eps).
  \end{align*}
  By the Growth Lemma~\ref{l_growth}, given $\eps > 0$ we can find
  $\eta > 0$ such that
  \begin{align*}
    \sum_j \phi(\brx_{j1}) \Leb_{W_{1}}(V_{j1} \cap \Omega_1)=
    \sum_{j}^* \phi(\brx_{j1}) \Leb_{W_{1}}(V_{j1} \cap \Omega_1)+O(\eps),
  \end{align*}
  where $\sum^*$ denotes the sum over components with
  $|W_{j1}|_\alpha\geq \eta$.

  By using Lebesgue Density Theorem and Severini--Egoroff
    Theorem, we can conclude that, for large enough $n > 0$
  \ignore{Lebesgue tells us that the pointwise limit of the density
    exists ae equal to $1$; Egoroff allows to deduce the needed
    uniformity in $j$.}
 $$
    \sum^* \phi(\brx_{j1}) \Leb_{W_{1}}(V_{j1} \cap \Omega_1)=
    \sum^{**} \phi(\brx_{j1}) \Leb_{W_{1}}(V_{j1} \cap \Omega_1)
     +O(\eps)
$$
  where the sum in $\sum^{**}$ is over the components satisfying
  \begin{equation} \label{2Star}
    |W_{j1}|_\alpha \geq \eta\text{ and } \Leb_{W_{1}}(V_{j1}\cap \Omega_1)\geq (1-\eps) |V_{j1}|_\alpha.
  \end{equation}
  Hence
  \begin{align}
  \label{SumStSt}
      \int_{\Omega_1} \phi(x) dx &= \sum^{**} \phi(\brx_j) \left|\brV_{j1}\right|_\alpha+O(\eps).
    \end{align}

Observe that the Distortion Estimates
    (Corollary~\eqref{cor:distortion}) and the fact that
      $|W_{j1}|\adm < 1$ imply that for some $C > 1$ and any $j$ so
    that $W_{j1}$ satisfies~\eqref{2Star}:
    \begin{align}\label{eq:lebesgue-density-forward}
      \Leb_{W_{j1}}(\cmp^{n}\Omega_{1})\ge (1-C\eps)|W_{j1}|\adm.
    \end{align}
    Let us fix $W_{j1}$. We want to show that there exists
    $W_{j2}\subset \cmp^n W_2$ which is sufficiently long and so that
    $\Leb_{W_{j2}}(\cmp^{n}\Omega_{2})\simeq\Leb_{W_{j1}}(\cmp^{n}\Omega_{1})$.

  Recall the definition of $Q(x)$ given in
    Section~\ref{SSInvMan}.  Let $x_{1}\in V_{j1}\cap\Omega_{1}$ and
    $y_{1} = \cmp^{n}x_{1}\in W_{j1}$: observe that, by definition,
    $W_{j1}\subset Q_{-n}(y_{1})$ and $V_{j1}\subset Q_{n}(x_{1})$.

    Let $x_{2} = \hol x_{1}\in W_{2}$. Then $x_{1}$ and
    $x_{2}$ are connected by a stable manifold, which by definition
    cannot cross the boundary of $Q_{n}$.  We conclude that
    $x_{2}\in Q_{n}(x_{1})$, which in turn implies that
    $W_{2}\cap Q_{n}(x_{1})$ is non-empty. Transversality of unstable
    curves and the boundary of $Q_{n}$ (composed of stable curves)
    then imply that $W_{2}\cap Q_{n}(x_{1})$ is connected, and since
    $\cmp^{n}$ is smooth on $Q_{n}(x_{1})$, we conclude that
    $\cmp^{n}(W_{2}\cap Q_{n}(x_{1}))$ is an H-component of
    $\cmp^{n}W_{2}$, that we denote with $\tilde W_{j2}$.  Let
    $\tilde V_{j2} = \cmp^{-n}\tilde W_{j2}$. Since $x_{1}$ is
    arbitrary, we conclude that
    $\hol(\Omega_{1}\cap V_{j1}) \subset \Omega_{2}\cap \tilde V_{j2}$.

    In other words, any shortened H-component of $\cmp^{n}W_{1}$
    cannot be linked with stable manifolds to more than one
    H-component of $\cmp^{n}W_{2}$.  Now observe that there exists two
    points $a_{1},b_{1}\in W_{j1}\cap\cmp^{n}\Omega_{1}$ that lie less
    than $C\eps |W_{j1}|\adm$ away from each of the boundary points of
    $W_{j1}$.  Otherwise, $\cmp^{n}\Omega_{1}$ would miss an interval
    of $\admName$-length larger than $C\eps|W_{j1}|\adm$ in $W_{j1}$,
    which is impossible by~\eqref{eq:lebesgue-density-forward}.  Let
    $\bar W_{j1}$ be the subcurve of $W_{j1}$ bounded by $a_{1}$ and
    $b_{1}$;
    then the triangle inequality yields:
    \begin{align}\label{eq:bar-w1j-length}
      |\bar W_{j1}|\adm &\ge (1-2C\eps)|W_{j1}|\adm.
    \end{align}
Recall that $a_{1}$ and $b_{1}$ belong to
    $W_{j1}\cap\cmp^{n}\Omega_{1}$. Hence we can define $a_{2},
    b_{2}\in\cmp_{n}\Omega_{2}\cap\tilde W_{j2}$ so that $a_{2} = \cmp^{n}\hol\cmp^{-n}a_{1}$ and
    $b_{2} = \cmp^{n}\hol\cmp^{-n}b_{1}$.  In
    particular $\dadm(a_{1},a_{2})\le \bd\Lambda^{-n}$ and
    $\dadm(b_{1},b_{2})\le \bd\Lambda^{-n}$.
    Let $\bar W_{j2}$ denote the subcurve of $\tilde W_{j2}$ bounded
    by $a_{2}$ and $b_{2}$. The triangle inequality yields
    \begin{align*}
      |\bar W_{j1}|\adm-2\bd\Lambda^{-n}\le
      |\bar W_{j2}|\adm = \dadm^{\tilde W_{j2}}(a_{2},b_{2})\le
      |\bar W_{j1}|\adm+2\bd\Lambda^{-n}
    \end{align*}
    Since $|W_{j1}|\adm > (1 -2C\eps)\eta$, we can assume $n$ to be so
    large that
      \begin{align*}
         (1-C\eps)|\bar W_{j1}|\adm \le
        |\bar W_{j2}|\adm\le
         (1+C\eps)|\bar W_{j1}|\adm.
      \end{align*}
    We now proceed to show that $\cmp^{n}\Omega_{2}\cap W_{j2}$ is
    large. More precisely we  will show that
    \begin{align} \label{W2-Omega2}
      \Leb_{\bar W_{j2}}(\cmp^{n}\Omega_{2})\ge (1-C\eps^{1/4})|\bar W_{j2}|\adm.
    \end{align}
    We want to use Corollary~\ref{CorSize} to show that there are many
    sufficiently long stable manifolds passing through $W_{j2}$ and we
    need to show that any sufficiently long stable manifold will cross
    $W_{j1}$.  In order to prove the latter statement, we argue as
    follows. First of all,
    combining~\eqref{eq:lebesgue-density-forward}
    and~\eqref{eq:bar-w1j-length} we conclude that there exists
    $\bar C$ (for instance taking $\bar C = 6C$ would do), so that
    \begin{align*}
      \Leb_{\bar W_{j1}}(\cmp^{n}\Omega_{1}) > (1-\bar C\eps)|\bar W_{j1}|\adm
    \end{align*}
    The above estimate implies that there exist
    $z_{1}^{(1)},\cdots,z_{1}^{(N)}\in\bar W_{j1}$ so that\footnote{
      Otherwise, $\cmp^{n}\Omega_{1}$ would miss an interval of length
      larger than $\bar C\eps|\bar W_{j1}|$}
    $\dadm^{\bar W_{j1}}(z_{i},z_{i+1}) < \bar C\eps|\bar
    W_{j1}|\adm$.  Let
    $z_{2}^{(k)} = \cmp^{n}\hol\cmp^{-n}z_{1}^{(k)}$. Our
    previous arguments, and the fact that stable manifolds cannot
    cross each other imply that $z_{2}^{(k)}\in\bar W_{j2}$ and, moreover,
    $\dadm(z_{1}^{(i)},z_{2}^{(i)}) < \bd\Lambda^{-n}$.  Once again,
    the triangle inequality shows, choosing a larger $n$ if needed,
    that
    $\dadm^{\bar W_{j2}}(z_{2}^{(i)},z_{2}^{(i+1)}) < \bar 2C\eps|\bar
    W_{j2}|\adm$.  Let $W_{j1}^{(i)}$ (\resp $W_{j2}^{(i)}$) be the
    subcurves in which the points $\{z_{1}^{(i)}\}$ (\resp
    $\{z_{2}^{(i)}\}$) partition $\bar W_{j1}$ (\resp $\bar W_{j2}$),
    and for any $i$, define the \emph{box} $B_{j}^{(i)}$ as the
    region bounded by $W_{j1}^{(i)}$, $W_{j2}^{(i)}$ and the two
    stable manifolds connecting the corresponding boundary points.  We
    claim that:
    \begin{align}\label{eq:bound-diameter-box}
      \text{diam}\adm B_{j}^{(i)}&\le 5\bar C\eps|\bar W_{j1}|;
    \end{align}
    In fact by the triangle inequality, the $\admName{}$-diameter of
    each cell is bounded above by the sum of the lengths of the four
    boundary curves; our previous estimates imply that
    \begin{align*}
      |W_{j1}^{(i)}|\adm + |W_{j2}^{(i)}|\adm \le 4\bar C\eps|\bar W_{j1}|;
    \end{align*}
    Since the length of the stable manifolds  can be made arbitrarily
    small by taking $n$ sufficiently large, we conclude
    that~\eqref{eq:bound-diameter-box} holds.

    On the other hand,~\eqref{eq:bound-diameter-box} implies that for
    any $z\in\bar W_{j2}$, if $\rs(z)\ge 5\bar C\eps|W_{j1}|\adm$,
    then $\Ws(z)$ will necessarily intersect $\bar W_{j1}$
    nontrivially (once again, stable manifolds cannot intersect each
    other) and thus $z\in\cmp^{n}\Omega_{2}$.  We now use
    Corollary~\ref{CorSize}(a') and the fact that
    $|W_{j2}|\adm \ge \eta$ to conclude that
    \begin{align*}
      \Leb_{W_{j2}}(\rs(z) \le 5\bar C\eps|W_{j2}|\adm) < \const\eps^{1/4}|W_{j2}|\adm.
    \end{align*}
    Hence
    \begin{align*}
      \Leb_{W_{j2}}(\tz\in \brW_{j2}: \Ws(\tz)\cap \brW_{j1}\neq
      \emptyset)\geq (1-\Const \eps^{1/4}) |\brW_{j2}|_\alpha,
    \end{align*}
    which is~\eqref{W2-Omega2}
        Therefore
    \begin{align*}
      \sum^{**} \phi(\brx_j) |\brV_{j1}|_\alpha&=
                                                 \sum^{**} \phi(\brx_j) |\brW_{j1}|_\alpha \left[\prod_{l=0}^{n-1} \jac{\cmp^{l}W_1}\cmp(\cmp^j \brx_{j})(\cmp^l \brx_j)\right]^{-1}+O(\eps)\\
                                               &=
                                                 \sum^{**} \phi(\brx_j) |\brW_{j2}|_\alpha \left[\prod_{l=0}^{n-1} \jac{\cmp^{l}W_1}\cmp(\cmp^j \brx_{j})(\cmp^l \brx_j)\right]^{-1}+O(\eps) \\
                                               &=\sum^{**} \phi(\brx_j) |\brV_{j2}|_\alpha
                                                 \left[\prod_{l=0}^{n-1} \frac{\jac{\cmp^{l}W_2}\cmp(\cmp^j \hol\brx_{j})(\cmp^l \hol\brx_j)}{\jac{\cmp^{l}W_1}\cmp(\cmp^j \brx_{j})(\cmp^l \brx_j)}\right]
                                                 +O(\eps) \\
                                               &= \sum^{**} \phi(\brx_j) |\brV_{j2}|_\alpha\; J(\brx_j)+O(\eps)+O(\theta^n).
    \end{align*}
    where the last step relies on Lemma \ref{LmJak}(a).

    Then the Bounded Distortion Corollary~\ref{cor:distortion}
    and~\eqref{W2-Omega2} yield:
    \begin{align*}
      \sum^{**} |\brV_{2j}| \phi(\brx_j) J(\brx_j) & \le \sum^{**} \mes(\brV_{j2}\cap \Omega_2) \phi(\brx_j) J(\brx_j)+O(\eps^{1/4}) \\
                                                   &\le \sum_j \mes(V_{2j}\cap \Omega_2) \phi(\brx_j) J(\brx_j)+O(\eps^{1/4}).
    \end{align*}
    Accordingly, assuming $n$ sufficiently big so that
    $\theta^{n} < \eps^{1/4}$,  and using \eqref{SumStSt}
    we obtain
    \begin{align*}
      \int_{\Omega_1} \phi(x) dx
      \leq \sum_j \mes(V_{j2}\cap \Omega_2)\phi(\brx_j) J(\brx_j)+O(\eps^{1/4}).
    \end{align*}
    Since $\eps > 0$ is arbitrary,
    \begin{align*}
      \int_{\Omega_1} \phi(x) dx\leq \int_{\Omega_2} \phi(\hol^{-1} y) J(\hol^{-1} y) dy.
    \end{align*}

    By symmetry
    \begin{align*}
      \int_{\Omega_2} \phi(\hol^{-1} y) J(\hol^{-1} y) dy\leq
      \int_{\Omega_1} \phi(\hol(\hol^{-1} x)) \frac{J(\hol(\hol^{-1} x))}{J(x)} dx=
      \int_{\Omega_1} \phi(x)  dx
    \end{align*}
    and Proposition~\ref{PrAC1} follows.
\end{proof}
\subsection{Absolute continuity of stable lamination.}
Consider a system of local coordinates $(a, b)$ in a small domain in
the phase space such that the curves $\{a=\text{const}\}$ are
unstable.  Define the set
\begin{align*}
 \Rect_{a_1, a_2}=\{&x: a_1\leq a(x)\leq a_2\text{ and }\\&
  \Ws(x) \cap \{a=a_1\}\neq\emptyset, \quad
  \Ws(x) \cap \{a=a_2\}\neq\emptyset \}.
\end{align*}
Consider another coordinate system $(u,s)$ on $\Rect_{a_1, a_2}$ such
that
\begin{align*}
  x(u,s)=\Ws(x(a_1, u))\cap \{a=s\}.
\end{align*}

Define the measure $d\nu=du ds$ on $\Rect_{a_1, a_2}$. For $i = 1, 2$,
let\footnote{Note that
  $\Omega_1=\{x\in W_1: W^s(x)\cap W_2\neq \emptyset\}$ where
  $W_j=\{a(x)=a_j\}.$ Therefore the notation $\Omega$ is consistent
  with \eqref{DomainRangeHol}.}
$\Omega_i=\Rect_{a_1, a_2} \cap \{a=a_i\},$ and define the sets:
\begin{align*}
  Z_{u_1, u_2}&=\{x\in \Rect_{a_1, a_2}: u_1\leq u(x) \leq u_2\},\\
  Z_{u_1, u_2; s_1, s_2}&=\{x\in \Rect_{a_1, a_2}: u_1\leq u(x) \leq u_2, s_1\leq s(x) \leq s_2\},\\
  Z_{u_1, u_2;s}&=\{x\in \Rect_{a_1, a_2}: u_1\leq u(x) \leq u_2, s(x)=s\},\\
  Z_{u;s_1, s_2}&=\{x\in \Rect_{a_1, a_2}: s_1\leq s(x) \leq s_2, u(x)=u\}.
\end{align*}
\begin{prp}\label{PrAC2}{\sc (Absolute Continuity-2)}
  The measure $\nu$ is equivalent to the restriction of the Lebesgue
  measure on $\Rect_{a_1, a_2}.$
\end{prp}
\begin{proof}
  Note that all smooth measures are equivalent, so below $\Leb$ will
  denote the measure defined by $d\Leb=da\; db$.  Note that
  $$\nu(Z_{u_1, u_2; s_1, s_2})= \nu_{Z_{u_{1},u_{2};a_{1}}}([u_1,
  u_2]\cap \Omega_1) (s_2-s_1),$$ where $\nu_{*}$ is the
  restriction of the measure $\nu$ on the set identified in the
  subscript.  On the other hand, by Proposition~\ref{PrAC1}, we have
  \begin{align*}
    \Leb(Z_{u_1,u_2; s_1, s_2}) &=\int _{s_1}^{s_2} \Leb_{\{a = s\}}(Z_{u_1, u_2; s}) ds\\
                                &=\int_{s_1}^{s_2} \int_{[u_1, u_2]\cap \Omega_1} J_{\hol_s}(x(a_1, u)) du ds,
  \end{align*}
  where $J_{\hol_s}$ is the Jacobian of the holonomy map
  $\hol_s: \Omega_1\to Z_{u_1, u_2; s}$.

  Since $J(\hol_s)$ is uniformly bounded from above and below, there is
  a constant $K > 1$ such that for each $[u_1, u_2],$ $[s_1, s_2]$ we have
  \begin{align*}
    {K}\inv\leq \frac{\nu(Z_{u_1, u_2; s_1, s_2})}{ \Leb(Z_{u_1, u_2; s_1, s_2})}\leq K,
  \end{align*}
  proving the proposition.
\end{proof}
\begin{cor}
  \label{AC1-AC2}
  The following are equivalent
  \begin{enumerate}
  \item $\Leb(A)=0$
  \item for almost every $x,$ $\mes(A\cap \Ws(x))=0.$
  \item for almost every $x,$ $\mes(A\cap \Wu(x))=0.$
  \end{enumerate}
\end{cor}
\begin{proof}
  We prove the equivalence of (a) and (b). The equivalence of (a) and
  (c) follows from analogous arguments.

  It suffices to prove the result under the assumption that
  $A\subset \Rect_{a_1, a_2}$ for some $a_1, a_2$.  But then
$$\Leb(A)=0\Leftrightarrow \nu(A)=0
\Leftrightarrow \text{ for a.e. }(u, s) \in\Omega_1\times [a_1, a_2] \quad \mes(A\cap Z_{u, a_1, a_2})=0 $$
$$\Leftrightarrow \text{ for a.e. }x\in \Rect_{a_1, a_2} \quad \mes(A\cap \Ws(x))=0. \qedhere $$
\end{proof}
\section{Ergodicity} \label{ScErg}
\begin{proof}[Proof of the Main Theorem]
  Fix a large number $R$.  Let $\cspi_R\subset \cspi$ be a region such that

  \begin{itemize}
  \item $\cspi\cap \{w<R\}\subset \cspi_R;$
  \item $\cspi\cap \{w<2R\}\supset \cspi_R$
  \item $\partial \cspi_R$ consists of curves in $\singp-$.
  \end{itemize}
  By Theorem \ref{ThRec} the first return map $\tcF_R:\cspi_R\to\cspi_R$
  is well defined so it is enough to show that $\tcF_R$ is ergodic for
  every $R$ sufficiently large.

  Let $\cR_0$ be the set of points $x\in \cspi_R$ such that for any
  continuous function $A$, the limits
  \begin{align*}
    \brA^+(x)&=\lim_{n\to\infty} \frac1n\sum_{j=0}^{n-1} A(\tcF_R^j x),&
    \brA^-(x)&=\lim_{n\to\infty} \frac1n\sum_{j=0}^{n-1} A(\tcF_R^{-j} x)
  \end{align*}
  exist and are equal. We shall call the common limit $\brA(x)$.  By
  Birkhoff Ergodic Theorem, the set $\cR_0$ has full Lebesgue measure
  in $\cspi_R$.  For $j > 0$ define
  \begin{align*}
    \cR_j(x)&=\{x\in \cR_{j-1}: \mes(\Wu(x)\setminus\cR_{j-1})
              = \mes(\Ws(x)\setminus\cR_{j-1})=0\}.
  \end{align*}
  By Corollary~\ref{AC1-AC2}, $\Leb(\cR_j^c)=0$ for all $j > 0$.
Note that since $\partial\cspi_R$ is a union of curves in
    $\singp-$, (un)stable manifolds for $\tcF_{R}$ are given by the
    intersection of (un)stable manifolds\footnote{ As a matter of
      fact, unstable manifolds are indeed the same, but stable
      manifolds might get truncated if they cross $\partial\cspi_{R}$}
    for $\cmp$ with $\cspi_{R}$.

  We now define the following equivalence relation: for
  $x_{1},x_{2}\in\cspi_{R}$, we let $x_1\sim x_2$ if and only if
  $\brA(x_1)=\brA(x_2)$ for all continuous functions $A$ on
  $\cspi_R$. If $x\in\cspi_{R}$, we denote with $\Sigma(x)$ the
  equivalence class of $x$.  To prove that $\tcF_R$ is ergodic it
  suffices to show that there exists an equivalence class of full
  measure in $\cspi_R$.

  For $K > 0$, let $Q$ be a connected component of
  $\cspi\setminus(\hsingp{K}\cup\hsingp{-K})$.  Observe that by
  construction, both $\cmp^K$ and $\cmp^{-K}$ are continuous on $Q$;
  moreover, for any $-K\le k\le K$, we have that $\Np_{k}$ is a
  constant function on $Q$ and for any $\Np_{-K}\le n\le\Np_{K}$, the
  image $\cm^{n}Q$ is contained in a single homogeneity strip. We call
  $Q$ a \emph{homogeneous $K$-cell}.  Observe that, by definition, if
  $Q$ is a homogeneous $K$-cell and $Q\cap\cspi_{R}\ne\emptyset$, then necessarily
  $Q\subset\cspi_{R}$.  Moreover, since $\cspi_{R}$ is compact,
  the Euclidean length and $\admName$-length are equivalent; we will
  use Euclidean length (and distance) for the rest of this section.

Since $w$ is bounded on $\cspi_{R}$, there is uniform
  transversality between the mature stable and mature unstable cones
  (recall~\eqref{MConeNarrow}). In particular, for any $R > 0$, there
  exists $L> 0$ so that the following holds: for any
  $x, x'\in\cspi_{R}$, let $W$ be a mature stable curve passing
  through $x$ and $W'$ a mature unstable curve passing through $x'$.
  If $r_{W}(x) > Ld(x,x')$ and $r_{W'}(x') > Ld(x,x')$, then
  $W\cap W'\ne\emptyset$.

  Observe that for any homogeneous
  $K$-cell $Q$, $x\in Q$ and $\Np_{-K}(x) < n < \Np_{K}(x)$:
  \begin{align*}
    d(\cm^{n}x,\cm^{n}\partial Q) &\le d(\cm^{n}x,\hsing1),\\
    d(\cm^{n}x,\cm^{n}\partial Q) &\le d(\cm^{n}x,\hsing{-1})
  \end{align*}
  In fact if \eg the first inequality did not hold, 
  $\cm^{n}Q$ would intersect non trivially $\hsing1$, but this means that
  either $\cm$ would not be continuous on $\cm^{n}Q$, or that
  $\cm^{n+1}Q$ intersects two homogeneity strips.  Neither of these
  possibilities is allowed by our construction.

  \begin{lem}[Local Ergodicity]
    \label{LmLocErg}
    There exists $K > 0$ (depending on $R$) such that any homogeneous
    $K$-cell $Q\subset\cspi_{R}$ is contained (mod 0) in a single
    equivalence class.
  \end{lem}
  \begin{proof}
    Let us fix a $K$-component $Q$ and let
    \begin{align*}
      d^{(K)}(x,\partial Q) = \min_{\Np_{-K}(x) < n < \Np_{K}(x)}d(\cm^{n} x,\cm^{n}\partial Q).
    \end{align*}
    Fix a small $\delta > 0$ to be specified later and define
    \begin{align*}
      Q^{\delta}&=\{x\in Q: d^{(k)}(x, \partial Q)>\delta\},
    \end{align*}
    Observe that $Q^{\delta}\ne\emptyset$ provided that $\delta$ is
    sufficiently small and that $\Leb(Q\setminus Q^{\delta})\to 0$ as
    $\delta\to 0$.  Then for any $\eps > 0$ define:
    \begin{align*}
      \cR^{\eps}&=\{x\in \cR_2: \ru(x)\ge \eps, \rs(x)\ge
                  \eps\}.
    \end{align*}
    We claim that there exists $C > 0$ so that for any $\delta > 0$ and
    sufficiently small $\eps > 0$,
    \begin{align}\label{eq:Reps-zero}
      \Leb(Q^{\delta}\setminus \cR^{\eps}) < C\eps.
    \end{align}
    In fact, assume that $\eps > 0$ is so small (relative to $\delta$)
    that for any $x\in Q^{\delta}$ we have $\dadm(\cm x,\hsing{-1}) >
    C\eps$ (where $C$ is the constant given by Corollary~\ref{CorSize}).

    Let us foliate $Q^{\delta}$ with mature \admiss{} unstable curves;
    for each such curve $W$, Corollary~\ref{CorSize}(a) implies that
    \begin{align*}
      \Leb_{W}(\rs(x) < \eps) < \Const\eps.
    \end{align*}
    Integrating over the curves, we get that
    $\Leb(Q^{\delta}\setminus\{\rs(x) < \eps\}) < \Const\eps$.
    Similarly, foliating with mature \admiss{} stable
    curves and applying Corollary~\ref{CorSize}(c), we obtain an
    analogous estimate for $\ru$, which yields~\eqref{eq:Reps-zero}.

    \begin{lem} \label{LmRegDense} For any small $\bar \eta > 0$, there
      exists $K > 0$ and $\eps_0 > 0$ such that for any $0 < \eps < \eps_0$,
      any $K$-component $Q$:
      \begin{enumerate}
      \item if $x\in Q^{\delta}$ then
        \begin{align*}
          \frac{\Leb(B(x,\eps)\cap \cR^{L \eps})}
          {\Leb(B(x,\eps))}>1-\bar\eta;
        \end{align*}
      \item If $x\in \cR^{L\eps}\cap Q^{\delta}$
        \begin{align*}
          \frac{\Leb(B(x,\eps)\cap \Sigma(x))}{\Leb (B(x,\eps))} > 1-\bar\eta.
        \end{align*}
      \end{enumerate}
    \end{lem}
    \begin{proof}
      To prove part (a), fix $\eta$ to be specified later and
      let $K$ be the $k$ given by Corollary~\ref{CorSize}(b), with the
      above choice of $\eta$.  Let $x\in Q^{\delta}$; by choosing
      $\eps_{0}$ sufficiently small (depending on $\delta$), we can
      guarantee that any point $x'\in B(x,\eps)$ satisfies\footnote{
        Recall that the $\admName$-metric and the Euclidean metric are
        equivalent} $d\adm(\cm^{n}x',\hsing{-1}) > L\eps$.  
      Foliate $B(x, \eps)$ by mature admissible unstable curves and
      disintegrate $\Leb|_{B(x,\eps)}$ on such unstable curves.
Then Corollary~\ref{CorSize}(b) implies that on any such unstable
      curve $W$
      \begin{align*}
        \Leb_{W}(\rs(x)\le L\eps)\le \eta L\eps.
      \end{align*}
      Integrating over all unstable curves we conclude that
      \begin{align*}
        \Leb(B(x,\eps)\cap \{\rs(x) < L\eps\})\le \eta L\eps^{2}.
      \end{align*}
      By foliating with mature \admiss{} stable curves and applying
      Corollary~\ref{CorSize}(d), we conclude the corresponding
      statement for $\ru$.  Collecting these two estimates we gather:
      \begin{align*}
        \frac{\Leb(B(x,\eps)\cap \cR^{L \eps})}{\Leb(B(x,\eps))}>1-\frac{2L\eta}{\pi}.
      \end{align*}
      Choosing a suitable $\eta$, we conclude the proof of item (a).

      To prove part (b), observe that by definition of $L$ we are
      guaranteed that if $x'\in B(x,\eps)$ and $x,x'\in\cR^{L\eps}$,
      then $x'\in \Sigma(x)$.  Hence $B(x,\eps)\cap \Sigma(x)\supset
      B(x,\eps)\cap\cR^{L\eps}$, and item (b) follows from item (a).
    \end{proof}
    Take $K$ so that the above lemma holds with
    $\bar\eta=\frac{1}{100}$.  Then for any
    $x\in \cR^{L\eps}\cap Q^{\delta}$ we have
    \begin{align*}
      \dfrac{\Leb(\Sigma(x)\cap B(x,\eps))}{\Leb(B(x,\eps))}\geq
      \frac{99}{100}.
    \end{align*}
    Assume now that
    \begin{equation}\label{LongClose}
      x_1, x_2\in \cR^{L\eps}\cap Q^{\delta} \text{ and } d(x_1, x_2)\leq \frac{\eps}{100}.
    \end{equation}
    Elementary geometry implies that
    \begin{align*}
      \frac{\Leb(B(x_{1},\eps)\cap B(x_{2},\eps))}{\Leb(B(x_{1},\eps))} > \frac12.
    \end{align*}
Thus
    $(B(x_{1},\eps)\cap \Sigma(x_1))\cap (B(x_{2},\eps)\cap\Sigma(x_2))$
    fills at least 25\%
    of $B_\eps(x_1)$. In particular,
    $\mes(\Sigma(x_1)\cap \Sigma(x_2))>0$.  Therefore \eqref{LongClose}
    implies that $x_1\sim x_2.$

    Next, given arbitrary $x_1, x_2\in \cR^{L\eps}\cap Q^{\delta}$,
    Lemma~\ref{LmRegDense}(a) allows to construct a chain of points
    \begin{align*}
      z_1, z_2,\cdots, z_N\in\cR^{L\eps}\cap Q^{\delta}
    \end{align*}
    such that $z_{1} = x_{1}$, $z_{N} = x_{2}$ and
    $d(z_j, z_{j+1}) < \eps/100.$ It follows that any
    $x_1, x_2\in \cR^{L\eps}\cap Q^{\delta}$ are equivalent.
    Then since $\eps$ can be taken arbitrarily
    small,~\eqref{eq:Reps-zero} implies that almost every
    $x_{1},x_{2}\in Q^{\delta}$ are equivalent. By the same token, since
    $\delta$ can be taken arbitrary small it follows that $Q$
    contains an equivalence class of full measure.
  \end{proof}
  The above lemma proves that for any $R > 0$ there exists
  $K > 0$ and a full-measure set $E\subset\cspi_R$ such that each
  equivalence class in $E$ is a union of $K$-components (mod
    0).

  We now prove that $E$ consists of a single equivalence class.
  Let $\hat E\subset E$ be an equivalence class; of course
  $\tcF_{R}\hat E = \hat E$. Moreover there exists $\hat E^{*}$ which
  is a union of homogeneous $K$-cells so that
  $\Leb(\hat E^{*}\setminus \hat E) = 0$.  Then, consider
  $\tcF_{R}^{\pm 2(K+1)}\hat E^{*}$; observe that the boundary
  $\partial\tcF_{R}^{2(K+1)}\hat E^{*}$ consist of curves in
  $\partial\cspi_{R}$ and unstable curves, whereas
  $\partial\tcF_{R}^{-2(K+1)}\hat E^{*}$ consists of curves in
  $\partial\cspi_{R}$ and stable curves.  By invariance of $\hat E$,
  the sets $\tcF_{R}^{\pm 2(K+1)}\hat E^{*}$ are equal (mod 0). We
  conclude that the boundaries are necessarily contained in
  $\partial\cspi_{R}$. Since $\cspi_{R}$ is connected, we conclude
  that $\hat E^{*} = \cspi_{R}$.
\end{proof}

\begin{rmk}
  Another approach of deducing ergodicity from local ergodicity (Lemma
  \ref{LmLocErg}) is due to Chernov and Sinai \cite{SC}. If there is
  more than one equivalence class there would be a curve $\Gamma$
  which is an arc of a discontinuity curve for some $\tcF^j$ with
  $|j|\leq K$ which separates two classes $E_1$ and $E_2$.  In
  particular, there is a point $x\in\Gamma$ and a small neighborhood
  $U$ of $x$ which consists of only two components of $E$: $E_1$ and
  $E_2$ which lie on different sides of $\Gamma$.  Suppose for example
  that $j\leq 0$ so that, by Lemma~\ref{l_localSing}, $\Gamma$ is an
  unstable curve.  Then we can assume (after possibly changing $x$),
  that $\tcF^K$ is continuous near $x,$ where $K$ is from
  Lemma~\ref{LmRegDense}.  For $l\in\{1,2\}$, let
  $\Sigma_l=\bigcup_{y\in E_l} W^s(y).$ Arguing as in the proof of
  Lemma~\ref{LmRegDense} we conclude that $\Sigma_1\cap \Sigma_2$ has
  positive measure.  This shows that in fact, $E_1$ and $E_2$ are
  equivalent, giving a contradiction.  Hence $E$ consists of a single
  class and so $\tcF_R$ is indeed ergodic.
\end{rmk}

\section{Open problems}
\label{s_conclusions}
In this section we present several possible directions of further research.\\

(I) In this paper we showed ergodicity of a class of piecewise smooth
Fermi--Ulam models. In principle we believe that this result can be
generalized to a broader, and more natural, class of wall motions.
More precisely, it should be possible to adapt our arguments to treat
motions that satisfy the same convexity conditions in the domains of
smoothness, but with more than one non-smoothness points, provided
that all of them are convex (i.e. the derivative has a positive jump).
It is more delicate to understand the behavior of Fermi--Ulam Models
with non-convex singularity points, since in principle
Proposition~\ref{p_expansionEstimate} might fail in this case
(similarly to what happens for dispersing billiards with corner points
and infinite horizon, see~\cite{BrownNandori}).  Indeed our proof of
Proposition~\ref{p_expansionEstimate} relies on the global structure
of singularities established in Section~\ref{s_globalSingularities}
and the arguments of the subsection rely on convexity of singular
points at several places.  Moreover, the results of~\cite{fum} would
also need to be generalized to prove, \eg recurrence for systems with
non-convex singularity points.  Thus, further non-trivial
investigation is required to understand the case of non-convex
singular points.

(II) Corollary \ref{CrOsc} says that almost every orbit is
oscillatory. Thus, for a typical orbit, the energy takes both large
and small values at different moments of time. It is of interest to
understand both rate of growth of energy and statistics of returns
similarly to what is done in \cite{CD2, DSV}.

(III) In Fermi--Ulam models the point mass keeps colliding with
the moving wall due to the presence of the fixed wall (a hard core
constraint). It is possible to ensure the recollisions via a soft
potential. Some results about large energy dynamics of particles in
soft potentials are obtained in \cite{dS, D1, P3}. It is assumed in
the above cited papers that the motion of the wall is smooth. One
could also consider piecewise smooth wall motions where ergodicity
seems likely under appropriate conditions.

(IV) This paper deals with the case where the velocity of the wall has
a jump. From the physical point of view it is natural to consider also
the case where acceleration has jump, but this seems much more
difficult since the energy change is much slower for large energies in
this case.

\appendix
\section{Regularity at infinity}\label{AppRegInf}
In this appendix we show that most Fermi--Ulam Models are superregular
at infinity.
\begin{lem}\label{l_boundedComplexity}
  For each $k$ the set of $\Delta$ such that $\LC_k(\Delta) > 3$ is discrete.
\end{lem}
In order to explain the proof more clearly, we first introduce a convenient change of coordinates. Let
\begin{align*}
  \xi &= \tau-1/2,&
  \eta& = I-\tau+1/2. 
\end{align*}  
If $x \in \hDom_{n_0,\cdots,n_{k-1}}$ we can express the orbit $\{x_l = \nf^lx\}_{0\le l < k}$ in $(\xi,\eta)$
coordinates as:
\begin{align*}
  \xi_{l+1} &= -(\eta_l-n_l),&
  \eta_{l+1}&= \kappa(\eta_l-n_l) +\xi_l + n_l
\end{align*}
where $\kappa = (2-\Delta) > 2$.  Let us define $\teta_l = \eta_l-n_l\in[-1/2,1/2]$ and the \emph{reduced
  itineraries} $\dn_l = n_{l+1}-n_l$. Then
\begin{align}
\label{EtaXiForward}
  \xi_{l+1} &= -\teta_l,&
  \teta_{l+1}&= \kappa\teta_l +\xi_l - \dn_l.
\end{align}
Iterating, we obtain
\begin{align}\label{e_forwardPolynomial}
  \teta_l = P_l(\kappa)\teta_0+P_{l-1}(\kappa)\xi_0-\sum_{j= 0}^{l-1}P_{l-j-1}(\kappa)\dn_j
\end{align}
where $P_l$ satisfies the recursive relation $P_{l+2} = \kappa P_{l+1}- P_l$, with $P_0(\kappa) = 1$ and
$P_1(\kappa) = \kappa$.  In particular, $P_l$ is a monic\footnote{ i.e. the coefficient of degree $l$ is equal to $1$}
polynomial of degree $l$.

 \eqref{EtaXiForward} can be rewritten as follows
\begin{align}
\label{EtaXiBackward}
 \teta_l&=- \xi_{l+1} &
  \xi_{l}&= \teta_{l+1}-\kappa\teta_l +\dn_l=\kappa\xi_{l+1} +\teta_{l+1}+\dn_l
\end{align}

Comparing \eqref{EtaXiForward} and \eqref{EtaXiBackward} we obtain the following analogue of
\eqref{e_forwardPolynomial}
\begin{align}\label{e_backwardPolynomial}
  \xi_0 = P_l(\kappa)\xi_l+P_{l-1}(\kappa)\teta_{l}+\sum_{j = 0}^{l-1}P_j(\kappa)\dn_j.
\end{align}
\begin{proof}[Proof of Lemma~\ref{l_boundedComplexity}]
Assume that $\LC_k(\Delta,x) > 3$. Then $x$ admits $4$ different
  itineraries, i.e.  four different choices of $k$-tuples which we
  denote with $\nt0,\nt1,\nt2,\nt3$ respectively.  Without loss of
  generality we will assume\footnote{ Otherwise we consider $\nf^mx$
    rather than $x$, where $m$ is the least index so that
    $\nt{i}_m\ne\nt{j}_m$ for some $0 \le i,j < 4$.} that
  $\nt{i}_0\ne\nt{j}_0$ for some $0 \le i,j < 4$.  Observe that
  $\nt{i}_0$ can take only two possible values (in case
  $\eta_l\in\bZ+1/2$). There are thus two possibilities, which can be
  described (again without loss of generality) as follows:
  \begin{enumerate}
  \item $\nt0_0 = \nt1_0 \ne \nt2_0 = \nt3_0,$
  \item $\nt0_0 = \nt1_0 = \nt2_0 \ne \nt3_0.$
  \end{enumerate}
  Let us first tackle case (a). Let $m'$ (\resp $m''$) denote the least index so that $\nt0_{m'}\ne\nt1_{m'}$
  (\resp $\nt2_{m''}\ne\nt3_{m''}$). By~\eqref{e_forwardPolynomial} we conclude that
  \begin{align*}
    \tetat0_{m'} &= P_{m'}\tetat0_0+P_{m'-1}\xit0_0-\sum_{j = 0}^{m'-1}P_{m'-j-1}\dnt0_j,\\
    \tetat2_{m''} &= P_{m''}\tetat2_0+P_{m''-1}\xit2_0-\sum_{j = 0}^{m''-1}P_{m''-j-1}\dnt2_j.
  \end{align*}
  Observe that by assumption $\tetat0_0 = -\tetat2_0,$ so that one of
  the numbers is $-\frac{1}{2}$ and the other is $+\frac{1}{2}$
  (otherwise $\nt0_0 = \nt 2_0$) and $\xit0_0 = \xit2_0$.  Multiplying
  the first equation by $P_{m''-1}$ and the second one by $P_{m'-1}$
  and subtracting we obtain
  \begin{align*}
    P_{m''-1}\tetat0_{m'}-P_{m'-1}\tetat2_{m''} &= (P_{m'}P_{m''-1}+P_{m''}P_{m'-1})\tetat0_{0}+\cO(\kappa^{m'+m''-2})
  \end{align*}
  Since $\tetat0_0, \tetat0_{m'}, \tetat2_{m''} = \pm1/2$ and $P_l$ is
  monic, we conclude that the above condition can be written in the
  form
  \begin{align*}
    Q(\kappa;\tetat0_0,\tetat0_{m'},\tetat2_{m''},\dnt0_0,\cdots,\dnt0_{m'-1},\dnt2_0,\dots,\dnt2_{m''-1}) = 0
  \end{align*}
  where $Q$ is a nonzero polynomial of degree $m'+m''-1$ in $\kappa$.
  since $k$ is fixed, for each $R > 2$ and $2\le \kappa < R$ we have
  only finitely many choices of the reduced itineraries $\dnt{i}$. Hence
  if we remove a discrete set of parameters, the above
  equation cannot hold for any itinerary.

  Let us now consider case (b).  We claim that in this case one of the
  itineraries (e.g. $\dnt0$) is such that there exists $l < m$ with
  $\tetat0_l = \pm1/2$ and $\tetat0_m = \pm1/2$.  In fact let $l$ be
  the least index so that $\nt{i}_l\ne\nt{j}_l$ for some $i\ne j$,
  which implies that $\tetat{i}_l = \pm1/2$ for $i = 0,1,2$. On the
  other hand, $\nt{i}_l$ can take only two possible values, thus we
  can assume without loss of generality that $\nt{0}_l = \nt{1}_l$.
  But $\nt0$ and $\nt1$ differ so there must exist $m > l$ so that
  $\nt{0}_m\ne\nt{1}_m$, which implies that $\tetat{0}_m = \pm1/2$.

  Thus by~\eqref{e_forwardPolynomial} we have
  \begin{align*}
  \tetat0_{m} = P_{m-l}\tetat0_l + P_{m-l-1}\xit0_l - \sum_{j =
    0}^{m-l-1}P_{m-l-j-1}\dnt0_{l+j}
  \end{align*}
  while \eqref{e_backwardPolynomial} and the fact that
  $\xit1_0=-\tetat0_0$ give
  \begin{align*}
    {-}\tetat0_{0} = P_{l-1}\xit0_l + P_{l-2}\tetat0_l + \sum_{j = 0}^{l-2}P_j\dnt0_{j+1}.
  \end{align*}

  Multiplying the first equation by $P_{l-1}$ and the second by
  $P_{m-l-1}$ and subtracting we obtain
  \begin{align*}
    P_{l-1}\tetat0_m{+}P_{m-l-1}\tetat0_0 = (P_{m-l}P_{l-1}-P_{m-l-1}P_{l-2})\tetat0_l+O(\kappa^{m-2}).
  \end{align*}
  Once again the above condition can be written in the form
  \begin{align*}
    Q(\kappa;\tetat0_0,\tetat0_l,\tetat0_m,\dnt0_0,\cdots,\dnt0_m) = 0
  \end{align*}
  where $Q$ is a nonzero polynomial of degree $m-1$.
  Using the same arguments as in case (a) we can conclude
  the proof.
\end{proof}
\bibliographystyle{abbrv}
\bibliography{fum}

\def\cprime{$'$}
\begin{thebibliography}{10}

\bibitem{Anosov}
D.~V. Anosov.
\newblock Geodesic flows on closed {R}iemannian manifolds of negative
  curvature.
\newblock {\em Trudy Mat. Inst. Steklov.}, 90:209, 1967.

\bibitem{AS}
D.~V. Anosov and J.~G. Sina\u{\i}.
\newblock Certain smooth ergodic systems.
\newblock {\em Uspehi Mat. Nauk}, 22(5 (137)):107--172, 1967.

\bibitem{kolya-patch}
P.~{B{\'a}lint}, J.~{De Simoi}, and I.~P. {T{\'o}th}.
\newblock A proof of {T}heorem 5.67 in ``{C}haotic {B}illiards''\\ by {C}hernov
  and {M}arkarian.
\newblock {S}hort note, 2019.

\bibitem{Br}
A.~Brahic.
\newblock Numerical study of a simple dynamical system. i. the associated plane
  area-preserving mapping.
\newblock {\em Astron. Astrophys.}, 12(1-2):98--110, 1971.

\bibitem{BrownNandori}
M.~{B}rown and P.~{N}\'andori.
\newblock Statistical properties of type {D} dispersing billiards.
\newblock preprint, 2019.

\bibitem{BSC}
L.~A. Bunimovich, Y.~G. Sinai, and N.~I. Chernov.
\newblock Statistical properties of two-dimensional hyperbolic billiards.
\newblock {\em Uspekhi Mat. Nauk}, 46(4(280)):43--92, 192, 1991.

\bibitem{CD1}
N.~Chernov and D.~Dolgopyat.
\newblock Brownian {B}rownian motion. {I}.
\newblock {\em Mem. Amer. Math. Soc.}, 198(927):viii+193, 2009.

\bibitem{CD2}
N.~Chernov and D.~Dolgopyat.
\newblock The {G}alton board: limit theorems and recurrence.
\newblock {\em J. Amer. Math. Soc.}, 22(3):821--858, 2009.

\bibitem{ChM}
N.~Chernov and R.~Markarian.
\newblock {\em Chaotic billiards}, volume 127 of {\em Mathematical Surveys and
  Monographs}.
\newblock American Mathematical Society, Providence, RI, 2006.

\bibitem{ChZ}
N.~Chernov and H.-K. Zhang.
\newblock On statistical properties of hyperbolic systems with singularities.
\newblock {\em J. Stat. Phys.}, 136(4):615--642, 2009.

\bibitem{C3}
N.~I. Chernov.
\newblock Ergodic and statistical properties of piecewise linear hyperbolic
  automorphisms of the {2}-torus.
\newblock {\em J. Statist. Phys.}, 69(1-2):111--134, 1992.

\bibitem{C4}
N.~I. Chernov.
\newblock Limit theorems and {M}arkov approximations for chaotic dynamical
  systems.
\newblock {\em Probab. Theory Related Fields}, 101(3):321--362, 1995.

\bibitem{C2}
N.~I. Chernov.
\newblock Decay of correlations and dispersing billiards.
\newblock {\em J. Statist. Phys.}, 94(3-4):513--556, 1999.

\bibitem{SC}
N.~I. Chernov and Y.~G. Sinai.
\newblock Ergodic properties of some systems of two-dimensional disks and
  three-dimensional balls.
\newblock {\em Uspekhi Mat. Nauk}, 42(3(255)):153--174, 256, 1987.

\bibitem{CZ}
B.~Chirkov and G.~Zaslavsky.
\newblock On the mechanism of fermi acceleration in the one-dimensional case.
\newblock {\em Sov. Phys. Doklady}, 159(2):98--110, 1964.

\bibitem{dS}
J.~{D}e {S}imoi.
\newblock Stability and instability results in a model of {F}ermi acceleration.
\newblock {\em Discrete Contin. Dyn. Syst.}, 25(3):719--750, 2009.

\bibitem{fum}
J.~{De Simoi} and D.~{Dolgopyat}.
\newblock {Dynamics of some piecewise smooth Fermi-Ulam models}.
\newblock {\em Chaos}, 22(2):026124, June 2012.

\bibitem{jmogy}
J.~De~Simoi and I.~P. T{\'o}th.
\newblock An expansion estimate for dispersing planar billiards with corner
  points.
\newblock {\em Ann. Henri Poincar\'e}, 15(6):1223--1243, 2014.

\bibitem{D1}
D.~Dolgopyat.
\newblock Bouncing balls in non-linear potentials.
\newblock {\em Discrete Contin. Dyn. Syst.}, 22(1-2):165--182, 2008.

\bibitem{D-icmp}
D.~Dolgopyat.
\newblock Piecewise smooth perturbations of integrable systems.
\newblock In {\em X{VII}th {I}nternational {C}ongress on {M}athematical
  {P}hysics}, pages 52--66. World Sci. Publ., Hackensack, NJ, 2014.

\bibitem{DF}
D.~Dolgopyat and B.~Fayad.
\newblock Unbounded orbits for semicircular outer billiard.
\newblock {\em Ann. Henri Poincar\'e}, 10(2):357--375, 2009.

\bibitem{DN-Mech}
D.~Dolgopyat and P.~N\'{a}ndori.
\newblock Infinite measure mixing for some mechanical systems.
\newblock {\em Preprint (2018), arXiv:1812.01174}.

\bibitem{DN-Tube}
D.~Dolgopyat and P.~N\'{a}ndori.
\newblock Nonequilibrium density profiles in {L}orentz tubes with thermostated
  boundaries.
\newblock {\em Comm. Pure Appl. Math.}, 69(4):649--692, 2016.

\bibitem{DSV}
D.~Dolgopyat, D.~Sz{\'a}sz, and T.~Varj{\'u}.
\newblock Recurrence properties of planar {L}orentz process.
\newblock {\em Duke Math. J.}, 142(2):241--281, 2008.

\bibitem{Do}
R.~Douady.
\newblock {\em Th\`ese de 3-\`eme cycle}.
\newblock 1982.

\bibitem{F49}
E.~Fermi.
\newblock On the origin of the cosmic radiation.
\newblock {\em Phys. Rev.}, 75:1169--1174, 1949.

\bibitem{F54}
E.~Fermi.
\newblock Galactic magnetic fields and the origin of the cosmic radiation.
\newblock {\em Ap. J.}, 119:1--6, 1954.

\bibitem{Hopf}
E.~Hopf.
\newblock Statistik der geod\"{a}tischen {L}inien in {M}annigfaltigkeiten
  negativer {K}r\"{u}mmung.
\newblock {\em Ber. Verh. S\"{a}chs. Akad. Wiss. Leipzig Math.-Phys. Kl.},
  91:261--304, 1939.

\bibitem{KatStr}
A.~Katok, J.-M. Strelcyn, F.~Ledrappier, and F.~Przytycki.
\newblock {\em Invariant manifolds, entropy and billiards; smooth maps with
  singularities}, volume 1222 of {\em Lecture Notes in Mathematics}.
\newblock Springer-Verlag, Berlin, 1986.

\bibitem{LL}
S.~Laederich and M.~Levi.
\newblock Invariant curves and time-dependent potentials.
\newblock {\em Ergodic Theory Dynam. Systems}, 11(2):365--378, 1991.

\bibitem{LenciI}
M.~Lenci.
\newblock {Semi-dispersing billiards with an infinite cusp. {I}}.
\newblock {\em Communications in Mathematical Physics}, 230(1):133--180, 2002.

\bibitem{LenciII}
M.~Lenci.
\newblock Semidispersing billiards with an infinite cusp. {II}.
\newblock {\em Chaos}, 13(1):105--111, 2003.

\bibitem{LY}
M.~Levi and J.~You.
\newblock Oscillatory escape in a {D}uffing equation with a polynomial
  potential.
\newblock {\em J. Differential Equations}, 140(2):415--426, 1997.

\bibitem{LW}
C.~Liverani and M.~P. Wojtkowski.
\newblock Ergodicity in {H}amiltonian systems.
\newblock In {\em Dynamics reported}, volume~4 of {\em Dynam. Report.
  Expositions Dynam. Systems (N.S.)}, pages 130--202. Springer, Berlin, 1995.

\bibitem{Pesin}
Y.~B. Pesin.
\newblock Dynamical systems with generalized hyperbolic attractors: hyperbolic,
  ergodic and topological properties.
\newblock {\em Ergodic Theory Dynam. Systems}, 12(1):123--151, 1992.

\bibitem{P2}
L.~D. Pustyl'nikov.
\newblock A problem of {U}lam.
\newblock {\em Teoret. Mat. Fiz.}, 57(1):128--132, 1983.

\bibitem{P1}
L.~D. Pustyl'nikov.
\newblock The existence of invariant curves for mappings that are close to
  degenerate and the solution of the {F}ermi-{U}lam problem.
\newblock {\em Mat. Sb.}, 185(6):113--124, 1994.

\bibitem{P3}
L.~D. Pustyl{\cprime}nikov.
\newblock Poincar\'e models, rigorous justification of the second law of
  thermodynamics from mechanics, and the {F}ermi acceleration mechanism.
\newblock {\em Uspekhi Mat. Nauk}, 50(1(301)):143--186, 1995.

\bibitem{Simanyi}
N.~Sim\'{a}nyi.
\newblock Towards a proof of recurrence for the {L}orentz process.
\newblock In {\em Dynamical systems and ergodic theory ({W}arsaw, 1986)},
  volume~23 of {\em Banach Center Publ.}, pages 265--276. PWN, Warsaw, 1989.

\bibitem{Sinai}
Y.~G. Sinai.
\newblock Dynamical systems with elastic reflections. {E}rgodic properties of
  dispersing billiards.
\newblock {\em Uspehi Mat. Nauk}, 25(2 (152)):141--192, 1970.

\bibitem{T2}
M.~Tsujii.
\newblock Absolutely continuous invariant measures for piecewise real-analytic
  expanding maps on the plane.
\newblock {\em Comm. Math. Phys.}, 208(3):605--622, 2000.

\bibitem{T1}
M.~Tsujii.
\newblock Piecewise expanding maps on the plane with singular ergodic
  properties.
\newblock {\em Ergodic Theory Dynam. Systems}, 20(6):1851--1857, 2000.

\bibitem{U}
S.~M. Ulam.
\newblock On some statistical properties of dynamical systems.
\newblock In {\em Proc. 4th {B}erkeley {S}ympos. {M}ath. {S}tatist. and
  {P}rob., {V}ol. {III}}, pages 315--320. Univ. California Press, Berkeley,
  Calif., 1961.

\bibitem{W-particles}
M.~P. Wojtkowski.
\newblock Systems of classical interacting particles with nonvanishing
  {L}yapunov exponents.
\newblock In {\em Lyapunov exponents ({O}berwolfach, 1990)}, volume 1486 of
  {\em Lecture Notes in Math.}, pages 243--262. Springer, Berlin, 1991.

\bibitem{Wojt}
M.~P. Wojtkowski.
\newblock {Two applications of {J}acobi fields to the billiard ball problem}.
\newblock {\em Journal of Differential Geometry}, 40(1):155--164, 1994.

\bibitem{Z}
V.~Zharnitsky.
\newblock Instability in {F}ermi-{U}lam ``ping-pong'' problem.
\newblock {\em Nonlinearity}, 11(6):1481--1487, 1998.

\bibitem{Zh}
J.~Zhou.
\newblock Piecewise smooth {F}ermi-{U}lam pingpong with potential.
\newblock {\em Preprint (2019), arXiv:1912.01154}.

\end{thebibliography}

\end{document}
%
